\documentclass[10pt]{amsart}

\usepackage{amsmath,amssymb,amscd,amsthm,amsxtra}

\usepackage{graphicx, mathrsfs}
\usepackage{amssymb,amsmath,amsthm}
\usepackage{mathrsfs,dsfont}

\makeatletter
\DeclareRobustCommand\widecheck[1]{{\mathpalette\@widecheck{#1}}}
\def\@widecheck#1#2{%
   \box\z@\hbox{\m@th$#1#2$}%
   \box\tw@\hbox{\m@th$#1%
      \widehat{%
         \vrule\@width\z@\@height\ht\z@
         \vrule\@height\z@\@width\wd\z@}$}%
   \dp\tw@-\ht\z@
   \@tempdima\ht\z@ \advance\@tempdima2\ht\tw@ \divide\@tempdima\thr@@
   \box\tw@\hbox{%
      \raise\@tempdima\hbox{\scalebox{1}[-1]{\lower\@tempdima\box\tw@}}}%
   {\ooalign{\box\tw@ \cr \box\z@}}}
\makeatother

\newtheorem{theorem}{Theorem} [section]

\newtheorem{lemma}[theorem]{Lemma}
\newtheorem{proposition}[theorem]{Proposition}
\newtheorem{remark}[theorem]{Remark}
\newtheorem{example}{Example}

\newtheorem{definition}[theorem]{Definition}

\begin{document}
\title[Local existence for Randomized GP hierarchies]{Local existence of solutions to Randomized Gross-Pitaevskii hierarchies}
\author[Vedran Sohinger]{Vedran Sohinger}
\address{
University of Pennsylvania, Department of Mathematics, David Rittenhouse Lab, Office 3N4B, 209 South 33rd Street, Philadelphia, PA 19104-6395, USA}
\email{vedranso@math.upenn.edu}
\urladdr{http://www.math.upenn.edu/~vedranso/}
\thanks{V. S. was supported by a Simons Postdoctoral Fellowship.}

\maketitle

\begin{abstract}

In this paper, we study the local-in-time existence of solutions to randomized forms of the Gross-Pitaevskii hierarchy on periodic domains. In particular, we study the \emph{independently randomized Gross-Pitaevskii hierarchy} and the \emph{dependently randomized Gross-Pitaevskii hierarchy}, which were first introduced in the author's joint work with Staffilani \cite{SoSt}. For these hierarchies, we construct local-in-time low-regularity solutions in spaces which contain a random component. The constructed density matrices will solve the full randomized hierarchies, thus extending the results from \cite{SoSt}, where solutions solving arbitrarily long subhierarchies were given.

Our analysis will be based on the truncation argument which was first used in the deterministic setting in the work of T. Chen and Pavlovi\'{c} \cite{CP4}. The presence of randomization in the problem adds additional difficulties, most notably to estimating the Duhamel expansions that are crucial in the truncation argument. These difficulties are overcome by a detailed analysis of the Duhamel expansions. In the independently randomized case, we need to keep track of which randomization parameters appear in the Duhamel terms, whereas in the dependently randomized case, we express the Duhamel terms directly in terms of the initial data. In both cases, we can obtain stronger results with respect to the time variable if we assume additional regularity on the initial data.

\end{abstract}

\section{Introduction}

\subsection{Background: The Gross-Pitaevskii hierarchy and its randomized forms}
\label{Setup of the problem}

Given a spatial domain $\Lambda=\mathbb{R}^d$ or $\Lambda=\mathbb{T}^d$, the \emph{Gross-Pitaevskii hierarchy on $\Lambda$} is defined to be the following infinite system of linear partial differential equations:
\begin{equation}
\label{GrossPitaevskiiHierarchy}
\begin{cases}
i \partial_t \gamma^{(k)} + (\Delta_{\vec{x}_k}-\Delta_{\vec{x}_k'}) \gamma^{(k)}=\sum_{j=1}^{k} B_{j,k+1} (\gamma^{(k+1)})\\
\gamma^{(k)}\big|_{t=0}=\gamma_0^{(k)}.
\end{cases}
\end{equation}
In the above notation, each $\gamma_0^{(k)}$ is a function $\gamma_0^{(k)}: \Lambda^k \times \Lambda^k \rightarrow \mathbb{C}$, which we henceforth refer to as a  \emph{density matrix of order} $k$. Moreover, $\gamma^{(k)}=\gamma^{(k)}(t)$ is a time-dependent density matrix of order $k$. By $\Delta_{\vec{x}_k}$ and $\Delta_{\vec{x}_k'}$, we denote the Laplace operator in the first and second factor set of $k$ spatial variables. Namely:
$$\Delta_{\vec{x}_k}:=\sum_{j=1}^{k} \Delta_{x_j}\,,\,\Delta_{\vec{x}_k'}:=\sum_{j=1}^{k} \Delta_{x'_j}.$$
Finally, $B_{j,k+1}$ denotes the \emph{collision operator}. This is a linear operator mapping density matrices of order $k+1$ to density matrices of order $k$ and it is precisely defined in Section \ref{Notation}. As in \cite{SoSt}, we will not assume any additional symmetry properties of the solutions $(\gamma^{(k)})_k$ to \eqref{GrossPitaevskiiHierarchy}.

The Gross-Pitaevskii hierarchy is closely related to the nonlinear Schr\"{o}dinger equation. Physically, both objects occur in the context of Bose-Einstein condensation. This is a state of matter of dilute bosonic particles which are cooled to a temperature close to absolute zero. At such a temperature, the particles tend to occupy the lowest quantum state. This one-particle state corresponds to the solution of a nonlinear Schr\"{o}dinger equation. Following this interpretation, the nonlinear Schr\"{o}dinger equation is sometimes called the \emph{Gross-Pitaevskii equation}, after the work of Gross \cite{Gross} and Pitaevskii \cite{Pitaevskii}. The physical phenomenon of Bose-Einstein condensation was originally predicted by Bose \cite{Bose} and Einstein \cite{Einstein} in 1924-1925. Their theoretical prediction was verified by experiments conducted independently by the research teams of Cornell and Wieman \cite{CW} and Ketterle \cite{Ket} in 1995. These two groups were jointly awarded the Nobel Prize for Physics in 2001.

A further connection between these two objects can be seen in the derivation of NLS-type equations from the equations describing the dynamics of $N$ bosonic particles as $N \rightarrow \infty$. The dynamics of $N$ bosonic particles are determined by the $N$-body Schr\"{o}dinger equation, which, in turn gives rise to a hierarchy called the \emph{BBGKY hierarchy}. This hierarchy is similar to \eqref{GrossPitaevskiiHierarchy}, but it has dependence on the number of particles $N$. As $N \rightarrow \infty$, one formally obtains the Gross-Pitaevskii hierarchy as the limit of the BBGKY hierarchy. Making this limiting procedure rigorous takes a lot of effort. This strategy was pioneered by Spohn \cite{Spohn}. In \cite{Spohn}, the author was able to rigorously derive the nonlinear Hartree equation 
$iu_t+\Delta u = (V*|u|^2)u$ on $\mathbb{R}^d$, for $V=V(x) \in L^{\infty}(\mathbb{R}^d)$.  We note that a key part of the above analysis was devoted to the study of conditional uniqueness of solutions to a hierarchy similar to \eqref{GrossPitaevskiiHierarchy}. Extensions of Spohn's result to the singular case of the Coulomb potential 
$V(x)=\pm \frac{1}{|x|}$ when $d=3$ were given in the subsequent works by Bardos, Golse and Mauser \cite{BGM} and Erd\H{o}s and Yau \cite{EY}.  The question of deriving the nonlinear Hartree equation was revisited in the work of Fr\"{o}hlich, Tsai, and Yau \cite{FrTsYau1,FrTsYau2,FrTsYau3}.
In a series of monumental works, Erd\H{o}s, Schlein and Yau \cite{ESY2,ESY3,ESY4,ESY5} use this framework in order to provide a rigorous derivation of the defocusing cubic nonlinear Schr\"{o}dinger equation on $\mathbb{R}^3$. The uniqueness step in the latter works was achieved by means of a Feynman graph expansion.

Finally, let us note that a structural connection between the Gross-Pitaevskii hierarchy and the nonlinear Schr\"{o}dinger equation is given in the framework of \emph{factorized solutions}. Namely, suppose that $\phi$ solves the defocusing cubic \emph{nonlinear Schr\"{o}dinger equation (NLS)} on $\Lambda$:
\begin{equation}
\notag
%\label{NLS2}
\begin{cases}
i \partial_t \phi + \Delta \phi = |\phi|^2\phi,\,\,\mbox{on}\,\,\mathbb{R}_t \times \Lambda\\
\phi|_{t=0}=\phi_0,\,\mbox{on}\,\Lambda.
\end{cases}
\end{equation}
Let $| \cdot \rangle \langle \cdot|$ denote the \emph{Dirac bracket}, which is given by $|f \rangle \langle g| (x,x'):=f(x) \overline{g(x')}$.
Then:
$$\gamma^{(k)}(t,\vec{x}_k;\vec{x}_k'):=\prod_{j=1}^k \phi(t,x_j) \overline{\phi(t,x_j')}=|\phi \rangle \langle \phi|^{\otimes k}(t,\vec{x}_k;\vec{x}_k')$$
solves \eqref{GrossPitaevskiiHierarchy} with initial data $\gamma_0^{(k)}=|\phi_0 \rangle \langle \phi_0|^{\otimes k}$. These are defined to be the \emph{factorized solutions} of \eqref{GrossPitaevskiiHierarchy}. Physically, they represent the one-particle state that was noted in the discussion above.

An alternative approach to part of the uniqueness analysis of \cite{ESY1,ESY2,ESY3,ESY4,ESY5} for $\Lambda=\mathbb{R}^3$ was given in the work of Klainerman and Machedon \cite{KM}. Their approach was based on a combinatorial reformulation of the Feynman graph argument and on a spacetime estimate of the following type:
\begin{equation}
\label{SpacetimeBoundalpha}
\big\|S^{(k,\alpha)} B_{j,k+1}\,\mathcal{U}^{(k+1)}(t)\, \gamma^{(k+1)}_0\big\|_{L^2([0,T] \times \Lambda^k \times \Lambda^k)} \lesssim \big\|S^{(k+1,\alpha)}\gamma^{(k+1)}_0\big\|_{L^2(\Lambda^{k+1} \times \Lambda^{k+1})}.
\end{equation}
Here, $\alpha$ is a fixed regularity exponent,  $T>0$ is a fixed time and the implied constant depends on $T$. The operator $\mathcal{U}^{(k)}(t)$ denotes the analogue of the free Schr\"{o}dinger evolution for density particles of order $k$, i.e. the free evolution associated to the operator $i \partial_t + \big(\Delta_{\vec{x}_k}-\Delta_{\vec{x}'_k}\big)$. The operator $S^{(k,\alpha)}$ corresponds to taking $\alpha$ fractional derivatives of density matrices of order $k$. The precise definitions of $\mathcal{U}^{(k)}(t)$ and $S^{(k,\alpha)}$ are given in \eqref{free_evolution} and \eqref{FractionalDerivative} below. We note that estimates similar to \eqref{SpacetimeBoundalpha} were already used in \cite{ESY2,ESY3,ESY4,ESY5} and in the works mentioned above. In \cite{KM}, the novelty was to prove \eqref{SpacetimeBoundalpha} by using methods reminiscent of those used in the proof of null-form estimates for the wave equation in the previous work of the same authors \cite{KM2}.
The Feynman diagram argument from \cite{ESY1,ESY2,ESY3,ESY4,ESY5} was reformulated in terms of a combinatorial \emph{boardgame argument}. By applying the spacetime estimate and the above combinatorial formulation, the authors were able to  prove a uniqueness result for \eqref{GrossPitaevskiiHierarchy} in a class of density matrices containing the factorized solutions.

The estimate \eqref{SpacetimeBoundalpha} and related bounds are used in the aforementioned works in order to contract the Duhamel expansions that occur in the study of \eqref{GrossPitaevskiiHierarchy}. In particular, by an appropriate use of \eqref{SpacetimeBoundalpha}, it is possible to control very long Duhamel expansions. The range of $\alpha$ for which it is possible to prove the spacetime estimate typically determines in what regularity classes one can prove results for \eqref{GrossPitaevskiiHierarchy} by using the mentioned strategies. 

Estimates related to \eqref{SpacetimeBoundalpha} have recently been studied outside of the realm of the Gross-Pitaevskii hierarchy. In particular, generalizations of \eqref{SpacetimeBoundalpha} for $\Lambda=\mathbb{R}^d$ were recently studied on their own right in the work of Beckner \cite{Beckner}. The motivation for the latter analysis was to develop a method for understanding more general forms of Stein-Weiss integrals which involve restriction to a smooth submanifold \cite{Beckner2}. 
 
In the work of Kirkpatrick, Schlein, and Staffilani \cite{KSS}, the authors prove that \eqref{SpacetimeBoundalpha} holds on $\mathbb{T}^2$ whenever $\alpha>\frac{3}{4}$. Using this result and an energy argument, they could obtain a rigorous derivation of the cubic defocusing nonlinear Schr\"{o}dinger equation on $\mathbb{T}^2$ and $\mathbb{R}^2$ from many-body quantum dynamics by the strategy outlined above. This was the rigorous derivation of the NLS on a periodic domain. The periodic problem for the Gross-Pitaevskii hierarchy was previously studied by Erd\H{o}s, Schlein, and Yau \cite{ESY1}, as well as by Elgart, Erd\H{o}s, Schlein, and Yau \cite{EESY}. In these works, the authors consider the problem when $\Lambda=\mathbb{T}^3$ and they obtain all of the steps in Spohn's strategy except for uniqueness. The uniqueness step was recently obtained by the author in \cite{VS2}.
We note that the problem that we are considering also makes physical sense if we assume that the bosonic particles are interacting in a bounded medium, as would be the case in a laboratory. 

In the author's previous work with Gressman and Staffilani \cite{GSS}, \eqref{SpacetimeBoundalpha} was shown to hold on $\mathbb{T}^3$ for $\alpha>1$. As an application of this result, it was possible to prove conditional uniqueness for \eqref{GrossPitaevskiiHierarchy} on $\mathbb{T}^3$ in an appropriate class $\mathcal{A}$ of density matrices which posses $\alpha>1$ fractional derivatives and which satisfy a certain a priori bound. We refer the reader to \cite{GSS} for a precise definition. The relevant class $\mathcal{A}$ was shown to be non-empty. In particular, it was noted that $\mathcal{A}$ contains the factorized solutions of the corresponding regularity. 

By constructing an explicit counterexample, it was shown in \cite{GSS} that the estimate \eqref{SpacetimeBoundalpha} \emph{does not hold on $\mathbb{T}^3$ when $\alpha=1$}. This is in sharp contrast with what happens when $\Lambda=\mathbb{R}^3$ \cite{KM} and when $\Lambda=\mathbb{T}^2, \mathbb{R}^2$ \cite{KSS}, where the spacetime estimate was shown to hold in the case $\alpha=1$. Such a phenomenon is consistent with the intuition that dispersion becomes weaker on periodic domains and with the intuition that dispersion becomes weaker when the dimension of the periodic domain becomes larger. 

The regularity level $\alpha=1$ gives the natural energy space for \eqref{GrossPitaevskiiHierarchy}. Hence, the approach based on the application of the spacetime estimate as in \eqref{SpacetimeBoundalpha} and energy arguments, which was used to derive the NLS equation from many-body dynamics on $\mathbb{T}^2$, can not be directly applied on $\mathbb{T}^3$. A rigorous derivation of the NLS equation from many-body dynamics on $\mathbb{T}^3$ was recently obtained by the author in \cite{VS2} by using different techniques, based on arguments from the work of T. Chen, Hainzl, Pavlovi\'{c}, and Seiringer \cite{CHPS} as well as Elgart, Erd\H{o}s, Schlein, and Yau \cite{EESY}.

It was noted in \cite{GSS} that, on $\mathbb{T}^3$, the estimate \eqref{SpacetimeBoundalpha} does hold when $\alpha=1$ in the special case of factorized objects. In other words, it was noted that \eqref{SpacetimeBoundalpha} holds if $\alpha=1$ and if $\gamma_0^{(k+1)}= |\phi \rangle \langle \phi|^{\otimes (k+1)}$ for some $\phi \in H^1(\mathbb{T}^3)$. A question which arises in this context is whether the spacetime estimate is true on $\mathbb{T}^3$ for a wider range of $\alpha$ in some \emph{generic sense}.

In the author's previous work with Staffilani \cite{SoSt}, an affirmative answer to this question was given by the use of probabilistic methods. This approach builds on previous applications of probability theory to the study of nonlinear dispersive equations. The general idea is that an instability can be overcome by introducing randomness into the model. The hope is that, for generic values of the random parameter, the instability does not occur. The first results of this type were the almost-sure global well-posedness results for the nonlinear Schr\"{o}dinger equation in low regularity by Bourgain \cite{B,B2,B3,B4}. They, in turn, build on the previous work of Lebowitz, Rose and Speer \cite{LRS}, and Zhidkov \cite{Zhidkov}. A related variant of the probabilistic approach which applies to equations which have less structure was recently developed by Burq and Tzvetkov \cite{BT1}. The latter approach has its precursors in the local-in-time analysis in the work of Bourgain \cite{B2}.
In all of the aforementioned works, the general idea was to use the randomness in order to extend the range of regularity exponents for which one can study the PDE. There has been a lot of work done on the application of randomization techniques to the study of nonlinear dispersive equations. A more detailed discussion of these techniques and additional relevant sources are given in Subsection \ref{Previously known results}.

As opposed to the works mentioned above and in Subsection \ref{Previously known results} below, in \cite{SoSt}, we do not randomize in the initial data. Instead, we \emph{randomize in the collision operator}. More precisely, we replace the collision operator $B_{j,k+1}$ by the \emph{randomized collision operator} $[B_{j,k+1}]^{\omega}$ for a given random parameter $\omega$ belonging to an appropriate probability space $\Omega$. The operator $[B_{j,k+1}]^{\omega}$ is obtained from $B_{j,k+1}$ by randomizing the Fourier coefficients by means of a collection of standard Bernoulli random variables. A precise definition of the randomized collision operator is given in \eqref{Bjkomega}, \eqref{Bjkomega2}, and \eqref{Bjkrandomized} below. With $[B_{j,k+1}]^{\omega}$ defined in this way, the following estimate was shown to hold for $\Lambda=\mathbb{T}^3$:
\\
\\
\emph{\textbf{Theorem 3.1 from \cite{SoSt}:}
Suppose that $\alpha>\frac{3}{4}$. Then, there exists $C_0>0$, depending only on $\alpha$, such that for all $k \in \mathbb{N}$, $j \in \{1,2,\ldots,k\}$, and for all density matrices $\gamma_0^{(k+1)}$ of order $k+1$ on $\mathbb{T}^3$, the following estimate holds:}
\begin{equation}
\label{Theorem1bound2}
\|S^{(k,\alpha)} \, [B_{j,k+1}]^{\omega} \,\gamma_{0}^{(k+1)}\|_{L^2(\Omega \times \mathbb{T}^{3k} \times \mathbb{T}^{3k})} \leq C_0 \|S^{(k+1,\alpha)} \gamma_0^{(k+1)}\|_{L^2(\mathbb{T}^{3(k+1)} \times \mathbb{T}^{3(k+1)})}.
\end{equation}
In particular, by unitarity of $\mathcal{U}^{(k+1)}(t)$:
\begin{equation}
\label{Theorem1bound}
\big\|S^{(k,\alpha)} \, [B_{j,k+1}]^{\omega} \, \mathcal{U}^{(k+1)}(t) \, \gamma_{0}^{(k+1)}\big\|_{L^2(\Omega \times [0,T] \times \mathbb{T}^{3k} \times \mathbb{T}^{3k})} \leq
\end{equation}
$$C_0 \sqrt{T} \big\|S^{(k+1,\alpha)} \gamma_0^{(k+1)}\big\|_{L^2(\mathbb{T}^{3(k+1)} \times \mathbb{T}^{3(k+1)})}.
$$
We can view the randomized estimate \eqref{Theorem1bound} as a randomized extension of the spacetime estimate \eqref{SpacetimeBoundalpha}. As we see, in the randomized setting, it holds for a wider range of regularity exponents $\alpha>\frac{3}{4}$. The ideas in the proof of this estimate were inspired by the original techniques used to obtain an almost sure gain in integrability as a result of the randomization due to Rademacher \cite{Rademacher}, with related work by Paley and Zygmund \cite{PaleyZygmund1,PaleyZygmund2,PaleyZygmund3}, as well as Marcinkiewicz and Zygmund \cite{MarcinkiewiczZygmund} and Khintchine, see \cite{Wolff}. The structure of the collision operator required us to perform a combinatorial analysis of the frequencies which pair up due to the randomization. 

Having constructed $[B_{j,k+1}]^{\omega}$, in \cite{SoSt} we studied several randomized modifications of the problem \eqref{GrossPitaevskiiHierarchy}. In particular, we studied the \emph{dependently randomized Gross-Pitaevskii hierarchy}:
\begin{equation}
\label{Introduction_RandomizedGP1}
\begin{cases}
i\partial_t \gamma^{(k)}+(\Delta_{\vec{x}_k} -\Delta_{\vec{x}'_k})\gamma^{(k)}=\sum_{j=1}^{k}[B_{j,k+1}]^{\omega}(\gamma^{(k+1)}) \\
\gamma^{(k)}|_{t=0}=\gamma_0^{(k)}.
\end{cases}
\end{equation}
Here $\omega \in \Omega$ is a random parameter and each $\gamma^{(k)}$ can depend on $\omega$. However, $\gamma_0^{(k)}$ does not depend on $\omega$.

Furthermore, we studied the \emph{independently randomized Gross-Pitaevskii hierarchy}:
\begin{equation}
\label{Introduction_RandomizedGP2}
\begin{cases}
i\partial_t \gamma^{(k)}+(\Delta_{\vec{x}_k} -\Delta_{\vec{x}'_k})\gamma^{(k)}=\sum_{j=1}^{k}[B_{j,k+1}]^{\omega_{k+1}}(\gamma^{(k+1)}) \\
\gamma^{(k)}|_{t=0}=\gamma_0^{(k)}.
\end{cases}
\end{equation}
Here, $\omega^{*}:=(\omega_k)_{k \geq 2}$ is a sequence of elements of the probability space $\Omega$. In this case, the randomizations of the collision operator are mutually independent at each level. This is emphasized by denoting the parameters $\omega_{k+1}$ with different indices. Each $\gamma^{(k)}$ can depend on the full sequence of  random parameters $\omega^{*}=(\omega_2,\omega_3,\omega_4,\ldots)$. As above, $\gamma_0^{(k)}$ does not have any random dependence. We note that \eqref{Introduction_RandomizedGP2} reduces to \eqref{Introduction_RandomizedGP1} when all of the $\omega_{k+1}$ are equal to $\omega$.

For the independently randomized hierarchy, we were able to apply the estimate \eqref{Theorem1bound2} in order to study sequences of Duhamel terms which solve \emph{arbitrarily long subhierarchies of} \eqref{Introduction_RandomizedGP2} with zero initial data. Namely, we starte from a fixed \emph{deterministic} sequence of time-dependent density matrices $(\gamma^{(k)}(t))_k$ which satisfies uniformly in $t$ the following bound:
\begin{equation}
\label{aprioriboundIL}
\|S^{(k,\alpha)}\gamma^{(k)}(t)\|_{L^{\infty}_{t} L^2(\mathbb{T}^{3k} \times \mathbb{T}^{3k})} \leq C_1^k
\end{equation}
for some constant $C_1>0$, independent of $k$. The condition \eqref{aprioriboundIL} is natural in light of the Sobolev-type spaces $\mathcal{H}^{\alpha}_{\xi}$, which were first introduced in the work of T. Chen and Pavlovi\'{c} \cite{CP1}. The precise definition of these spaces is recalled in \eqref{Space1} below. Furthermore, this bound is motivated by the a priori estimate needed in the conditional uniqueness proof of Klainerman and Machedon \cite{KM}. 

Having defined the sequence $(\gamma^{(k)}(t))_k$ in this way, we let, for fixed $n \in \mathbb{N}$ and for independently chosen random parameters $\omega_{k+1},\omega_{k+2},\ldots,\omega_{n+k} \in \Omega$:
\begin{equation}
\label{DuhamelRandomI2L}
\sigma^{(k)}_{n;\,\omega_{k+1},\omega_{k+2},\ldots,\omega_{n+k}}(t):=
\end{equation}
$$(-i)^n \int_{0}^{t} \int_{0}^{t_{k+1}} \cdots \int_{0}^{t_{n+k-1}} \mathcal{U}^{(k)}(t_k-t_{k+1}) \, [B^{(k+1)}]^{\omega_{k+1}} \,\mathcal{U}^{(k+1)}(t_{k+1}-t_{k+2})$$
$$ [B^{(k+2)}]^{\omega_{k+2}} \cdots \,\mathcal{U}^{(n+k-1)}(t_{n+k-1}-t_{n+k}) \, [B^{(n+k)}]^{\omega_{n+k}} \, \gamma^{(n+k)}(t_{n+k}) \,dt_{n+k} \cdots dt_{k+2} \, dt_{k+1}.$$
In this notation, the superscript $k$ denotes the order of the density matrix whereas the subscript $n$ denotes the length of the Duhamel expansion. Moreover, $[B^{(k+1)}]^{\omega}:=\sum_{j=1}^{k}[B_{j,k+1}]^{\omega}$ denotes the \emph{full randomized collision operator}, as defined in \eqref{Bomegak+1} below.

For $\ell=1,2,\ldots,n$, let us define:
\begin{equation}
\label{gammatilde}
\tilde{\gamma}^{(\ell)}:=\sigma^{(\ell)}_{n-\ell+1;\,\omega_{\ell+1}, \omega_{\ell+2}, \ldots, \,\omega_{n+1}}.
\end{equation}
We observe that then, the $\tilde{\gamma}^{(\ell)}$ solve:

\begin{equation}
\notag
\begin{cases}
i \partial_t \tilde{\gamma}^{(k)} + (\Delta_{\vec{x}_k}-\Delta_{\vec{x}'_k})\tilde{\gamma}^{(k)}=\sum_{j=1}^{k} [B_{j,k+1}]^{\omega_{k+1}} (\tilde{\gamma}^{(k+1)})\\
\tilde{\gamma}^{(k)}\big|_{t=0}=0.
\end{cases}
\end{equation}
for all $k \in \{1,2,\ldots,n-1\}$. In other words, the $\tilde{\gamma}^{(j)}$ defined in \eqref{gammatilde} give us an \emph{arbitrarily long subset of solutions} to the full hierarchy \eqref{Introduction_RandomizedGP2}
with zero initial data. We then say that they solve an \emph{arbitrarily long subhierarchy of} \eqref{Introduction_RandomizedGP2}. For $\sigma^{(k)}_{n;\,\omega_{k+1},\omega_{k+2},\ldots,\omega_{n+k}}(t)$ defined as in \eqref{DuhamelRandomI2L}, we could iteratively apply the estimate \eqref{Theorem1bound2} in order to prove:
\\
\\
\emph{\textbf{Theorem 5.2 from \cite{SoSt}:}
Let $\alpha>\frac{3}{4}$ and $k \in \mathbb{N}$. Then, there exists $T>0$ depending only on $\alpha$ and on the constant $C_1$ in \eqref{aprioriboundIL} such that:
\begin{equation}
\label{Theorem2boundL}
\,\, \sup_{t \in [0,T]} \, \big\|S^{(k,\alpha)} \sigma^{(k)}_{n;\,\omega_{k+1},\omega_{k+2},\ldots,\omega_{n+k}}(t)\big\|_{L^2 \big(\Omega_{k+1} \times \Omega_{k+2} \times \cdots \times \Omega_{n+k}; \,L^2(\mathbb{T}^{3k} \times \mathbb{T}^{3k})\big)} \rightarrow 0\,\,
\end{equation} 
as $n \rightarrow \infty.$
}
\\
\\
It was noted in \cite{SoSt} that it is possible to replace the finite product $\Omega_{k+1} \times \Omega_{k+2} \times \cdots \times \Omega_{n+k}$ on the right-hand side of \eqref{Theorem2boundL} with the infinite product $\Omega^{*}:=\prod_{j \geq 2} \Omega_{j}$ following the work of Kakutani \cite{Kakutani} and Kolmogorov \cite{Kolmogorov}. In this way, the convergence obtained in \eqref{Theorem2boundL} is measured in the same space. One can view \eqref{Theorem2boundL} as a randomized analogue of the analysis of the Duhamel terms which was used in the proof of uniqueness results for the Gross-Pitaevskii hierarchy \eqref{GrossPitaevskiiHierarchy} in \cite{KM}, and subsequently in \cite{GSS,KSS}. An important feature of the result is that $T$ remains constant as $n \rightarrow \infty$, which allows us to work on a fixed time interval. 
	
In the study of the dependently randomized Gross-Pitaevskii hierarchy \eqref{Introduction_RandomizedGP1}, it is not possible to iteratively use the estimate \eqref{Theorem1bound2} as in the independently randomized setting. Furthermore, is it not possible to directly estimate a full Duhamel expansion for arbitrary initial data, as is shown in Example 1 of Subsection 6.1 of \cite{SoSt}.
However, one can obtain an estimate for the full Duhamel expansion of \emph{non-resonant objects}. In other words, we recall that in \cite{SoSt}, it was necessary to consider Duhamel expansions of objects in the \emph{non-resonant class} $\mathcal{N}$. Let us note that the non-resonance conditions resemble the idea of \emph{Wick ordering}, which was used in the study of the NLS in \cite{B3,COh}, as well as in a related approach in \cite{NS}. Moreover, the objects in $\mathcal{N}$ are assumed to satisfy an a priori bound as in \eqref{aprioriboundIL}. The non-resonant class is precisely defined in Subsection 6.3 of \cite{SoSt} and the corresponding estimate in this class is given in Proposition 6.2 of \cite{SoSt}.

In particular, for a fixed $(\gamma^{(k)}(t))_k$ \emph{belonging to the non-resonant class $\mathcal{N}$} and for a fixed $\omega \in \Omega$, we define:
\begin{equation}
\notag
\label{DuhamelRandomI1L}
\sigma^{(k)}_{n;\,\omega}(t):=
\end{equation}
$$(-i)^n \int_{0}^{t} \int_{0}^{t_{k+1}} \cdots \int_{0}^{t_{n+k-1}} \mathcal{U}^{(k)}(t_k-t_{k+1}) \, [B^{(k+1)}]^{\omega} \,\mathcal{U}^{(k+1)}(t_{k+1}-t_{k+2})$$
$$ [B^{(k+2)}]^{\omega} \cdots \,\mathcal{U}^{(n+k-1)}(t_{n+k-1}-t_{n+k}) \, [B^{(n+k)}]^{\omega} \, \gamma^{(n+k)}(t_{n+k}) \,dt_{n+k} \cdots dt_{k+2} \, dt_{k+1}.$$ 
In this way, we obtain solutions to arbitrarily long subhierarchies of  \eqref{Introduction_RandomizedGP1} similarly as above.
The following result was shown to hold:
\\
\\
\emph{\textbf{Theorem 6.4 from \cite{SoSt}:}	
Let $\alpha \geq 0$ and $k \in \mathbb{N}$. There exists $T>0$, depending on $\alpha$, $k$, and on the constant $C_1$ in the definition of $\mathcal{N}$ such that, for $\sigma^{(k)}_{n;\,\omega}$ as defined in \eqref{DuhamelRandomI1L} above: 
\begin{equation}
\notag
%\label{Theorem3bound}
\sup_{t \in [0,T]} \, \big\|S^{(k,\alpha)} \sigma^{(k)}_{n;\,\omega}(t)\big\|_{L^2 \big(\Omega \times \mathbb{T}^{3k} \times \mathbb{T}^{3k} \big)} \rightarrow 0
\end{equation}
as $n \rightarrow \infty.$
}

The Gross-Pitaevskii hierarchy has recently been studied as a Cauchy problem in a series of works by T. Chen and Pavlovi\'{c} \cite{CP1,CP_survey_article,CP4,CP3,CP}, as well as in their joint works with Tzirakis \cite{CPT1,CPT2}, and in the subsequent work of Z. Chen and Liu \cite{CL}. The Cauchy problem associated to the Hartree equation for infinitely many particles has recently been studied by Lewin and Sabin \cite{LewinSabin1,LewinSabin2}.
In these papers, the authors prove analogues of results known for the Cauchy problem associated to the nonlinear Schr\"{o}dinger equation in the context of the Cauchy problem associated to the Gross-Pitaevskii hierarchy or to the Hartree equation for infinitely many particles. In the former case, the motivation for this analysis is that the Gross-Pitaevskii hierarchy \eqref{GrossPitaevskiiHierarchy} can be thought of as a generalization of the defocusing cubic nonlinear Schr\"{o}dinger equation in the sense that from each solution $\phi$ to the defocusing cubic NLS, we obtain a factorized solution to \eqref{GrossPitaevskiiHierarchy}. By introducing a coupling constant, one can also study the focusing problem.
By defining a different collision operator, it is possible to obtain a Gross-Pitaevskii hierarchy which is related to the quintic NLS \cite{CP1,CP_survey_article,CP2,CP4,CP3,CPT1,CPT2}, as well as the NLS with more general nonlinearities \cite{Xie}. In \cite{SoSt}, it was noted that \eqref{Introduction_RandomizedGP1} admits factorized solutions coming from an NLS-type equation with a cubic random nonlinearity. In this context, we can think of the randomization in the collision operator as a \emph{nonlinear form of randomization}. Details of this construction are given in Subsection 4.1 of \cite{SoSt}.

The above analysis can be extended to the general case of $\Lambda=\mathbb{T}^d$ when $d \in \mathbb{N}$ without much change. For the independently randomized hierarchy, the condition $\alpha>\frac{3}{4}$ gets replaced by $\alpha>\frac{d}{4}$, and for the dependently randomized hierarchy, the condition $\alpha \geq 0$ does not change. We refer the interested reader to Remark 3.7 in \cite{SoSt} and to the discussion in the introductory paragraph of Section 6 in \cite{SoSt}. In the discussion that follows, we will primarily work with the case $\Lambda=\mathbb{T}^3$ for simplicity of notation and for its physical interpretation.

There has been a great amount of activity in the study of problems related to the Gross-Pitaevskii hierarchy, both from the theoretical and from the experimental point of view. Additional references to those mentioned above are cited in Subsection \ref{Previously known results} below. In what follows, we will explain the statement of our problem and we will give the main results. 

\subsection{Statement of the problem. Main Results}
\label{Statement of the problem. Main Results}
As was explained above, in \cite{SoSt}, we were able to construct solutions to arbitrarily long subhierarchies of the randomized hierarchies \eqref{Introduction_RandomizedGP1} and \eqref{Introduction_RandomizedGP2}. In this paper, we will be interested in constructing solutions \emph{to the full hierarchies} \eqref{Introduction_RandomizedGP1} and \eqref{Introduction_RandomizedGP2}, at least locally in time. Our analysis of this problem was inspired by the work of T. Chen and Pavlovi\'{c} \cite{CP4} in the deterministic setting. In \cite{CP4}, the authors construct local-in-time solutions to  \eqref{GrossPitaevskiiHierarchy} for $\Lambda=\mathbb{R}^d$ by using a truncation method. 
More precisely, given a parameter $N \in \mathbb{N}$, they consider the \emph{truncated Gross-Pitaevskii hierarchy}, i.e the system for $k \leq N$:

\begin{equation}
\label{TruncatedGP}
\begin{cases}
i \partial_t \, \gamma_N^{(k)}+(\Delta_{\vec{x}_k}-\Delta_{\vec{x}'_k}) \, \gamma_N^{(k)}=\sum_{j=1}^{k} B_{j,k+1} \big(\gamma_N^{(k+1)}\big)\\
\gamma_N^{(k)} \big|_{t=0} = \gamma_0^{(k)}.
\end{cases}
\end{equation}
with the additional condition that:
\begin{equation}
\label{TruncatedGPcondition}
\gamma_N^{(k)} \equiv 0,\,\,\mbox{whenever}\,\,k>N.
\end{equation}
Similarly as before, each $\gamma_N^{(k)}$ is a density matrix of order $k$. 
They show that it is possible to construct local-in-time solutions $\gamma_N^{(k)}$ to \eqref{TruncatedGP} which satisfy the additional condition \eqref{TruncatedGPcondition} by using explicit Duhamel expansions. Moreover, it is shown that there exists $T>0$, depending on the initial data, such that $\gamma_N^{(k)}(t)$ has a limit as $N \rightarrow \infty$ whenever $t \in [0,T]$. The limit is observed to solve \eqref{GrossPitaevskiiHierarchy} on $[0,T]$. This is purely an existence result and there is no statement involving uniqueness of these solutions. 

In our paper, we would like to adapt this approach to the randomized models \eqref{Introduction_RandomizedGP1} and \eqref{Introduction_RandomizedGP2}
and construct local-in-time solutions to these hierarchies. As we will see, the randomness in the model will pose certain challenges. These will be overcome by an appropriate application of the randomized estimates from \cite{SoSt}. Throughout the following discussion, we will consider the case $\Lambda=\mathbb{T}^3$ unless it is stated otherwise. 

Let us recall the definition of a Sobolev-type norm associated to \eqref{GrossPitaevskiiHierarchy}, which was first introduced in the work of T. Chen and Pavlovi\'{c} \cite{CP1}. If $\Sigma=(\sigma^{(k)})$ is a sequence such that each $\sigma^{(k)}$ is a density matrix of order $k$, then for $\alpha,\xi>0$, one defines:
$$\big\|\sigma\|_{\mathcal{H}^{\alpha}_{\xi}}:=\sum_{k \in \mathbb{N}} \xi^k \cdot \|S^{(k,\alpha)}\sigma^{(k)}\|_{L^2(\mathbb{T}^{3k} \times \mathbb{T}^{3k})}.$$
The main results that we prove are:
\\
\\
\emph{\textbf{Theorem 1:}
Let us fix $\alpha>\frac{3}{4}$, and  $\xi,\xi'>0$, such that $\frac{\xi}{\xi'}$ is sufficiently small. 
Suppose that $\Gamma(0):=(\gamma_0^{(k)})_k \in \mathcal{H}^{\alpha}_{\xi'}(\mathbb{T}^3)$. Then, there exists $T>0$, depending on $\alpha,\xi,\xi'$, and $[\Gamma\,]^{\omega^{*}}=([\gamma^{(k)}]^{\omega^{*}})_{k} \in L^{\infty}_{t \in [0,T]} L^2(\Omega^{*}) \mathcal{H}^{\alpha}_{\xi}$, such that, for all $k \in \mathbb{N}$:
\begin{equation}
\notag
%\label{Theorem1_Introduction}
\big\|[\gamma^{(k)}]^{\omega^{*}}(t)
-\,\mathcal{U}^{(k)}(t) \, \gamma_0^{(k)}+i \int_{0}^{t} \,\mathcal{U}^{(k)}(t-s)\, [B^{(k+1)}\,]^{\omega_{k+1}} [\gamma^{(k+1)}]^{\omega^{*}}(s) \,ds \big\|_{L^{\infty}_{t \in [0,T]} L^2(\Omega^{*}) H^{\alpha}(\mathbb{T}^{3k} \times \mathbb{T}^{3k})}=0.
\end{equation}
Moreover, $[\Gamma\,]^{\omega^{*}}$ satisfies the a priori bound:
\begin{equation}
\label{A_priori_bound1}
\big\|[\Gamma\,]^{\omega^{*}}\big\|_{L^{\infty}_{t \in [0,T]} L^2(\Omega^{*}) \mathcal{H}^{\alpha}_{\xi}} \lesssim_{\,\xi,\,\xi',\,\alpha} \big\|\Gamma(0)\big\|_{\mathcal{H}^{\alpha}_{\xi'}}.
\end{equation}
The space $L^{\infty}_{t \in [0,T]}L^2(\Omega^{*})\mathcal{H}^{\alpha}_{\xi}$ is defined in \eqref{Space2B} below.
}
\\
\\
Here, we note that, for every $k \in \mathbb{N}$: 
$$[\gamma^{(k)}]^{\omega^{*}} : I \times \Omega^{*} \times \mathbb{T}^{3k} \times \mathbb{T}^{3k} \rightarrow \mathbb{C}.$$
The sequence $[\Gamma\,]^{\omega^{*}}$ from Theorem 1 should be thought of as a local-in-time solution to the independently randomized Gross-Pitaevskii hierarchy \eqref{Introduction_RandomizedGP2} in the sense given above. Theorem 1 is given as Theorem \ref{Theorem1} below. 
The conditions relating $\alpha, T, \xi$ and  $\xi'$  are given in \eqref{xi_prime} and \eqref{xi}.

If we add additional regularity assumptions on the initial data $\Gamma(0)=(\gamma_0^{(k)})_k$, we can prove the following stronger result:
\\
\\
\emph{\textbf{Theorem 2:}
Let us fix $\alpha>\frac{3}{4}$, and $\xi',\xi>0$, with $\frac{\xi}{\xi'}$ sufficiently small. Furthermore, let $\alpha_0>\alpha$ be given. Suppose that $\Gamma(0)=(\gamma_0^{(k)})_k \in \mathcal{H}^{\alpha_0}_{\xi'}(\mathbb{T}^3)$. Then, there exists $T>0$, depending on $\alpha,\alpha_0,\xi,\xi'$, and $[\Gamma\,]^{\omega^{*}}=([\gamma^{(k)}]^{\omega^{*}})_{k} \in C_{t \in [0,T]} L^2(\Omega^{*}) \mathcal{H}^{\alpha}_{\xi}$, such that, for all $k \in \mathbb{N}$, and for all $t \in [0,T]$:
\begin{equation}
\notag
%\label{Theorem1_Introduction}
\big\|[\gamma^{(k)}]^{\omega^{*}}(t)
-\mathcal{U}^{(k)}(t) \, \gamma_0^{(k)}+i \int_{0}^{t} \,\mathcal{U}^{(k)}(t-s)\, [B^{(k+1)}\,]^{\omega_{k+1}} [\gamma^{(k+1)}]^{\omega^{*}}(s) \,ds \big\|_{L^2(\Omega^{*}) H^{\alpha}(\mathbb{T}^{3k} \times \mathbb{T}^{3k})}=0.
\end{equation}
The sequence $[\Gamma\,]^{\omega^{*}}$ satisfies the a priori bound:
\begin{equation}
\label{A_priori_bound2}
\big\|[\Gamma\,]^{\omega^{*}}\big\|_{L^{\infty}_{t \in [0,T]} L^2(\Omega^{*}) \mathcal{H}^{\alpha}_{\xi}} \lesssim_{\,\xi,\,\xi',\,\alpha,\,\alpha_0} \big\|\Gamma(0)\big\|_{\mathcal{H}^{\alpha}_{\xi'}}.
\end{equation}
The space $C_{t \in [0,T]} L^2(\Omega^{*}) \mathcal{H}^{\alpha}_{\xi}$ is defined in \eqref{Space2C} below.
}
\\
\\
In other words, if assume the initial data to be infinitesimally more regular than $\alpha$, we can obtain a solution of regularity $\alpha$ \emph{for all $t$} in the time interval, as opposed to \emph{for almost all $t$}. By analogy with the theory of nonlinear dispersive equations, we can view Theorem 1 as the construction of a \emph{weak solution} and Theorem 2 as the construction of a \emph{strong solution} to the hierarchy. This terminology can be found, for example in \cite{Tao}. 

From the proof of Theorem 2, it follows that the sequences $[\Gamma\,]^{\omega^{*}}$ constructed in the above two theorems coincide on their common domains of definition. It is possible that the time interval in Theorem 2 is a bit smaller due to the additional dependence on $\alpha_0$. Theorem 2 is given as \ref{Theorem2} below. The precise conditions on the parameters are given in the statement of Theorem \ref{Theorem2}.

For the dependently randomized Gross-Pitaevskii hierarchy \eqref{Introduction_RandomizedGP1}, we can prove the following analogues of the above results:
\\
\\
\emph{\textbf{Theorem 3:}
Let us fix $\alpha \geq 0$, and  $\xi,\xi'>0$, such that $\frac{\xi}{\xi'}$ is sufficiently small. 
Suppose that $\Gamma(0)=(\gamma_0^{(k)})_k$ belongs to  $\mathcal{H}^{\alpha}_{\xi'}(\mathbb{T}^3) \cap \mathcal{N}$. Here, $\mathcal{N}$ denotes the non-resonant class, which is precisely defined in Definition \ref{Non-resonant} below. Then, there exists $T>0$, depending on $\alpha,\xi,\xi'$, and $[\Gamma\,]^{\omega}=([\gamma^{(k)}]^{\omega})_{k} \in L^{\infty}_{t \in [0,T]} L^2(\Omega) \mathcal{H}^{\alpha}_{\xi}$, such that, for all $k \in \mathbb{N}$:
\begin{equation}
\notag
%\label{Theorem1_Introduction}
\big\|[\gamma^{(k)}]^{\omega}(t)
-\mathcal{U}^{(k)}(t) \, \gamma_0^{(k)}+i \int_{0}^{t} \,\mathcal{U}^{(k)}(t-s)\, [B^{(k+1)}\,]^{\omega_{k+1}} [\gamma^{(k+1)}]^{\omega}(s) \,ds \big\|_{L^{\infty}_{t \in [0,T]} L^2(\Omega) H^{\alpha}(\mathbb{T}^{3k} \times \mathbb{T}^{3k})}=0.
\end{equation}
Moreover, $[\Gamma\,]^{\omega}$ satisfies the a priori bound:
\begin{equation}
\label{A_priori_bound3}
\big\|[\Gamma\,]^{\omega}\big\|_{L^{\infty}_{t \in I} L^2(\Omega) \mathcal{H}^{\alpha}_{\xi}} \lesssim_{\,\xi,\,\xi',\,\alpha} \big\|\Gamma(0)\big\|_{\mathcal{H}^{\alpha}_{\xi'}}.
\end{equation}
The space $L^{\infty}_{t \in [0,T]}L^2(\Omega)\mathcal{H}^{\alpha}_{\xi}$ is defined in \eqref{Space3B} below.
}
\\
\\
The sequence $[\Gamma\,]^{\omega}$ should be thought of as a local-in-time solution to the dependently randomized Gross-Pitaevskii hierarchy \eqref{Introduction_RandomizedGP1}.
We note that in Theorem 3, the range of regularity exponents is $\alpha \geq 0$, so it is possible to construct solutions from $L^2$ initial data.
Theorem 3 is given as Theorem \ref{Theorem1GP2} below. The conditions relating $\alpha, T, \xi$, and $\xi'$ are given in \eqref{xi_prime2} and \eqref{xi2}.

Furthermore, we can prove the following analogue of Theorem 2 for the dependently randomized Gross-Pitaevskii hierarchy:
\\
\\
\emph{\textbf{Theorem 4:}
Let us fix $\alpha \geq 0$, and  $\xi',\xi>0$, such that $\frac{\xi}{\xi'}$ is sufficiently small. Moreover, let $\alpha_0>\alpha$ be given. Suppose that $\Gamma(0)=(\gamma_0^{(k)})_k$ belongs to $\mathcal{H}^{\alpha_0}_{\xi'}(\mathbb{T}^3) \cap \mathcal{N}$. Then, there exists $T>0$, depending on $\alpha,\xi,\xi'$, and $[\Gamma\,]^{\omega}=([\gamma^{(k)}]^{\omega})_{k} \in C_{t \in [0,T]} L^2(\Omega) \mathcal{H}^{\alpha}_{\xi}$, such that, for all $k \in \mathbb{N}$, and for all $t \in [0,T]$:
\begin{equation}
\notag
%\label{Theorem1_Introduction}
\big\|[\gamma^{(k)}]^{\omega}(t)
-\mathcal{U}^{(k)}(t) \, \gamma_0^{(k)}+i \int_{0}^{t} \,\mathcal{U}^{(k)}(t-s)\, [B^{(k+1)}\,]^{\omega_{k+1}} [\gamma^{(k+1)}]^{\omega}(s) \,ds \big\|_{L^2(\Omega) H^{\alpha}(\mathbb{T}^{3k} \times \mathbb{T}^{3k})}=0.
\end{equation}
The sequence $[\Gamma\,]^{\omega}$ satisfies the a priori bound:
\begin{equation}
\label{A_priori_bound4}
\big\|[\Gamma\,]^{\omega}\big\|_{L^{\infty}_{t \in I} L^2(\Omega) \mathcal{H}^{\alpha}_{\xi}} \lesssim_{\,\xi,\,\xi',\,\alpha,\,\alpha_0} \big\|\Gamma(0)\big\|_{\mathcal{H}^{\alpha}_{\xi'}}.
\end{equation}
The space $C_{t \in [0,T]} L^2(\Omega)\mathcal{H}^{\alpha}_{\xi}$ is defined in \eqref{Space3C} below.
}
\\
\\
In other words, we can construct \emph{strong solutions} to \eqref{Introduction_RandomizedGP1} of regularity $\alpha$ if the initial data is infinitesimally more regular.
As in the independently randomized setting, it is the case that the sequences $[\Gamma\,]^{\omega}$ constructed in Theorem 3 and in Theorem 4 coincide on their common domains of definition. Theorem 4 is given as Theorem \ref{Theorem2GP2} below. The precise conditions on the parameters are given in the statement of Theorem \ref{Theorem2GP2}.

These results should be viewed in the context of a \emph{low-regularity almost-sure local existence theory}. 
Here, we adapt the point of view from \cite{CP1} and the related works cited above, in the sense that we view the Gross-Pitaevskii hierarchy as a Cauchy problem and proving analogues of known results for nonlinear Schr\"{o}dinger equation. More precisely, we can obtain solutions to the randomized Gross-Pitaevskii hierarchies \eqref{Introduction_RandomizedGP1} and \eqref{Introduction_RandomizedGP2} in an almost-sure sense evolving from initial data in a low-regularity space. In the case of \eqref{Introduction_RandomizedGP2}, the regularity for the initial data is $\alpha>\frac{3}{4}$ and for \eqref{Introduction_RandomizedGP1}, it is $\alpha \geq 0$. The difference from the existing literature noted above and in Subsection \ref{Previously known results} is that \emph{the initial data is not random}, but that \emph{the randomness is in the collision operators}, i.e in the equation itself. Our solutions $[\Gamma\,]^{\omega^{*}}$ and $[\Gamma\,]^{\omega}$ solve the hierarchies in spaces which involve the random parameter, and not in a deterministic space. In other words, the solutions we obtain exist \emph{for all initial data}, but they are solutions in the \emph{almost-sure sense} due to the random structure of the spaces $L^{\infty}_{t \in I} L^2(\Omega^{*}) \mathcal{H}^{\alpha}_{\xi}$ and 
$L^{\infty}_{t \in I} L^2(\Omega) \mathcal{H}^{\alpha}_{\xi}$.

From the construction of the spaces used above, we can obtain bounds on the solutions that we construct. For example, the solution $[\Gamma\,]^{\omega^{*}} \in L^{\infty}_{t \in [0,T]} L^2(\Omega^{*}) \mathcal{H}^{\alpha}_{\xi}$ to \eqref{Introduction_RandomizedGP2} constructed in Theorem 1, will satisfy for all $k \in \mathbb{N}$:
\begin{equation}
\label{aprioriboundIL_randomized}
\big\|S^{(k,\alpha)}[\gamma^{(k)}]^{\omega^{*}}\big\|_{L^{\infty}_{t \in [0,T]} L^2(\Omega^{*}) L^2(\mathbb{T}^{3k} \times \mathbb{T}^{3k})} \leq C^k
\end{equation}
for some $C>0$, depending on $\alpha, T, \xi,\xi'$ and $\Gamma(0)$. This follows immediately from the a priori bound \eqref{A_priori_bound1}. We can view \eqref{aprioriboundIL_randomized} as a randomized version of the a priori bound \eqref{aprioriboundIL}.
Hence, we have shown that the randomized version of the estimate holds for solutions to the \emph{full hierarchy}.
Analogous results can be deduced from the a priori estimates \eqref{A_priori_bound2}, \eqref{A_priori_bound3}, and \eqref{A_priori_bound4}. Additional a priori bounds on the constructed solutions, which correspond to probabilistic versions of the assumption used in the work of Klainerman and Machedon \cite{KM} are given in Remarks \ref{BhatGamma_omega_starA} and \ref{BhatGamma_omega_A} below.

As was noted in the earlier discussion, the analysis of our paper generalizes to the case $\Lambda=\mathbb{T}^d$, with a minor modifications in the statements. Namely, Theorem 1 and Theorem 2 hold for $\Lambda=\mathbb{T}^d$ if the condition $\alpha>\frac{3}{4}$ gets replaced by the condition $\alpha>\frac{d}{4}$. Here, we need to use Remark 3.7 of \cite{SoSt} which states that the $d$-dimensional analogue of \eqref{Theorem1bound2} holds for $\alpha>\frac{d}{4}$. Moreover, Theorem 3 and Theorem 4 hold for $\Lambda=\mathbb{T}^d$ with no change in the assumptions. The only change is that the definition of the non-resonant class is now modified as in Remark \ref{Non_resonant_Td}. All of the relevant estimates remain unchanged by the discussion from the introduction to Section 6 of \cite{SoSt}.
These observations are summarized as Remarks \ref{Remark_higher_dimensions1} and  \ref{Remark_higher_dimensions2} below.

Let us note that, by time-reversibility, all of the results stated above also hold for negative times, i.e. on the interval $[-T,0]$ as well. In general, it is possible to construct solutions on any interval $I$ of length $T$ given that $\Gamma(t_0) \in \mathcal{H}^{\alpha}_{\xi}$ at some $t_0 \in I$. For simplicity of notation, we will consider the case $I=[0,T]$ and $t_0=0.$
We note that the initial data is always taken to be deterministic.

%\textbf{APPLICATION OF RANDOMIZED ESTIMATES}

%\textbf{RECALL THE WORK OF T. CHEN AND PAVLOVI\'{C}}
%\textbf{REFER TO SECTION ABOUT ADDITIONAL REFERENCES}
%\textbf{STATE THE MAIN THEOREMS}

%\textbf{MAYBE ADD A DISCUSSION ON DEPENDENCE ON THE INITIAL DATA!}
%\textbf{ADD A DISCUSSION ABOUT WHAT HAPPENS ON $\mathbb{T}^d$}
%\textbf{COMMENT ON THE LOW-REGULARITY INITIAL DATA}
%\textbf{COMMENT ON THE A PRIORI BOUND}

\subsection{Previously known results}
\label{Previously known results}
In this subsection, we will summarize some known results on the connections between the Gross-Pitaevskii hierarchy and the nonlinear Schr\"{o}dinger equation in addition to the ones that we had mentioned above. A more comprehensive discussion about the history of this problem and of many relevant results is given in the expository works \cite{LSSY} and \cite{Schlein}.
Moreover, we will note some further results on the use of probabilistic techniques in the study of nonlinear dispersive PDE. A detailed summary of the main techniques and of the earlier results obtained by using these techniques can be found in the expository works \cite{1B24} and \cite{Zh}.

In addition to the aforementioned references, there is a vast literature on the connection between the derivation of NLS-type equations and hierarchies of the type \eqref{GrossPitaevskiiHierarchy}, as well as on related problems. Simultaneously with development of the strategy of Spohn \cite{Spohn}, mentioned above, a different direction based on Fock space techniques was taken in the work of Hepp \cite{Hepp} and Ginibre and Velo \cite{GV1,GV2}. Here, the authors were also able to derive NLS-type equations. The Fock space-based technique was later applied in \cite{XC1,XC2,FKP,FKS,GM,GMM1,GMM2}. Ideas related to the strategy of Spohn were subsequently applied to similar problems in \cite{ABGT,AGT,BGM,BePoSc1,BePoSc2,BdOS,CP,CT,XC3,XC4,ChenHolmer1,ChenHolmer2,ChenHolmer3,ES,FGS,FL,MichelangeliSchlein,Xie}. 

The question of the rate of convergence in the derivation of the NLS was first studied by Rodnianski and Schlein \cite{RodnianskiSchlein}. Subsequent results on this aspect of the problem have been proved in \cite{Anapolitanos,BdOS,ChenLeeSchlein,XC1,XC2,ErdosSchlein,FKP,GM,GMM1,GMM2,KP,Lee,Luhrmann,MichelangeliSchlein,Pickl1,Pickl2}. 

The Gross-Pitaevskii hierarchy has been studied at the $N$-body level in \cite{LS,LSY,LSY2} and later in \cite{LewinNamRougerie3}. A connection of the above problems with optical lattice models was explored in \cite{ALSSY1,ALSSY2}. We refer the interested reader to a more detailed discussion about the previously mentioned aspects of the problem given in the introduction of \cite{GSS}.

Recently, a proof of unconditional uniqueness for the cubic Gross-Pitaevskii hierarchy when $\Lambda=\mathbb{R}^3$ has been obtained in the work of T. Chen, Pavlovi\'{c}, Hainzl, and Seiringer \cite{CHPS} by using the \emph{Quantum de Finetti Theorem}. Techniques based on this theorem have subsequently been used in order to show scattering results in \cite{ChHaPavSei2}, as well as in order to obtain uniqueness in low regularities in \cite{HTX}. Furthermore, these ideas were a crucial ingredient in the author's derivation of the defocusing cubic nonlinear Schr\"{o}dinger equation on $\mathbb{T}^3$ \cite{VS2}. Recently, the Quantum de Finetti was used in the study of the Chern-Simons-Schr\"{o}dinger hierarchy in \cite{CS}. The Quantum de Finetti theorem is a quantum analogue of the classical theorem of de Finetti concerning exchangeable sequences of random variables \cite{deFinetti1,deFinetti2}, with later extensions in \cite{DiaconisFreedman,Dynkin,HewittSavage}. The quantum version of the theorem states that, under certain assumptions, density matrices are given as an average over factorized states. This type of result was first proved in the $C^*$ algebra context in \cite{HudsonMoody,Stormer}. Connections to density matrices were noted in \cite{AmmariNier1,AmmariNier2,LewinNamRougerie}. The first application of these ideas to the uniqueness problem was in \cite{CHPS}.

Randomization techniques have been shown to be useful in the study of nonlinear dispersive equations at low regularities. In particular, the idea is to apply some form of randomization in order to construct solutions in case the deterministic methods such as the \emph{high-low method} of Bourgain \cite{1B7}, or the \emph{I-method} of Colliander, Keel, Staffilani, Takaoka, and Tao \cite{CKSTT} are known not to apply. As was mentioned above, this approach was pioneered by Bourgain \cite{B,B2,B3,B4}, with precursors in the work of Lebowitz, Rose, and Speer \cite{LRS} and Zhidkov \cite{Zhidkov}. The main idea is that almost-sure global existence can be obtained from the existence of an invariant Gibbs measure associated to the equation. The Gibbs measure makes sense for low-regularity initial data and it is supported away from the set of initial data which obstructs the application of the deterministic methods. The invariance of the Gibbs measure heuristically replaces a conservation law at the level of low regularity. We note that related problems had also been studied by McKean and Vaninsky \cite{McKeanVaninsky1,McKeanVaninsky2,McKeanVaninsky3}.

The construction of the invariant measure 
depends on the structure of the equation. In particular, it is known to be applicable primarily in the Hamiltonian context. In the more general setting, it is still possible to apply the ideas of randomization without appealing to the existence of an invariant Gibbs measure. The key point is to note that in 
\cite{B}, the construction of the invariant measure induces a randomization of the initial data at the level of the Fourier coefficients. This randomization allows one to prove a local result in time. In a more general context, one can \emph{randomize the initial data} without appealing to an invariant measure. In particular, given a function $f=\sum_{n} c_n e^{inx}$ in $L^2(\Lambda)$ as a Fourier series, one can consider its \emph{randomization} $f^{\omega}$, which is obtained by multiplying each Fourier coefficient $c_n$ by an appropriate random variable $h_n(\omega)$. One typically takes $(h_n(\omega))_n$ to be a sequence of independent, identically distributed random variables with expected value equal to zero. This method is sometimes called a \emph{rough randomization}. A precise definition is given in \eqref{randomization_of_a_function} below.
	
	This more general approach was applied in the study of almost-sure local theory for supercritical wave equations in the work of Burq and Tzvetkov \cite{BT1}. In this work, the authors used the fact that \emph{randomization improves integrability almost surely}. As was noted in the introduction, this observation was first made by Rademacher \cite{Rademacher}. Related results were proved by Paley and Zygmund \cite{PaleyZygmund1,PaleyZygmund2,PaleyZygmund3}, as well as by Marcinkiewicz and Zygmund \cite{MarcinkiewiczZygmund} and Khintchine, see \cite{Wolff}. In the recent work \cite{AyacheTzvetkov}, this question of almost sure improved integrability was revisited by alternative means in the special case of Gaussian random variables. As was noted earlier, the idea of the proof of the almost sure gain in integrability is an important step in the proof of \eqref{Theorem1bound2} in \cite{SoSt}.
The gain in integrability due to randomization is related to the more general phenomenon of \emph{hypercontractivity} 
\cite{Federbush,Glimm,LGross1,LGross2,Nelson1}.

Both the invariant measure and the rough randomization approach have been shown to be useful in the study of nonlinear dispersive equations.
We refer the interested reader to the additional relevant works \cite{BenyiOhPocovnicu1,BenyiOhPocovnicu2,BourgainBulut,BourgainBulut2,BourgainBulut3,BTT,BurqTzvetkov,BT1,1BT2,BurqandTzvetkov,Cacciafesta_deSuzzoni,COh, Deng1, Deng2,DengTzvetkovVisciglia,LM,NORBS,NRBSS,Oh1,Oh2,Oh3,OhSulem,Richards,deSuzzoni2,deSuzzoni,deSuzzoni3,deSuzzoniTzvetkov,1Tho,TT,Tzv1,Tzv2,Tzv3,Xu} and the references therein. 
The rough randomization technique has recently also been applied in the context of the Navier-Stokes equations in supercritical regimes \cite{DengCui1,DengCui2,NahmodPavlovicStaffilani,ZhangFang}. 

In \cite{SoSt}, probabilistic methods similar to those used in the study of nonlinear dispersive equations were applied in order to study the Gross-Pitaevskii hierarchy. We note that different probabilistic methods in other contexts had been used in previous work on problems concerning the $N$-body Schr\"{o}dinger equation. In \cite{BenArousKSch}, the authors prove a central limit theorem for the fluctuations around Hartree dynamics obtained in the limit $N \rightarrow \infty$. This work builds on the results previously proved in 
\cite{CramerEisert,CushenHudson, GVV,Hayashi,HeppLieb,JaksicPautratPillet,Kuperberg}. 
Furthermore, in \cite{AdCoKo}, the authors interpret a many-body system of mutually repellent bosons as an expectation with respect to a marked Poisson process. This formulation is used in order to obtain an explicit variational formulation of the limiting free energy. Similar methods were used before in the setting without interaction in \cite{Fichtner, Rafler}. Finally, we note that in \cite{ChatterjeeDiaconis}, the authors give a rigorous proof of the phenomenon of Bose-Einstein condensation with positive probability under the assumption that the temperature is sufficiently low. The probability measure with respect to which this is measured is the canonical Gibbs measure on configuration states. A more detailed discussion about these problems is given in the introduction of \cite{SoSt}.

Let us note that probabilistic methods have also been applied in the study of the $N$-body Schr\"{o}dinger equation in the experimental literature. In particular, we note that the experiment in \cite{Zwierlein1} is based on an application of a Monte Carlo method of thermodynamic measurements adapted to a unitary Fermi gas across the superfluid phase transition. This experiment succeeds in confirming the theory of strongly interacting matter originally proposed by Bardeen, Cooper and Schrieffer \cite{BaCoSc1,BaCoSc2,Co}. The randomization method in this context is the \emph{Bold diagrammatic Monte Carlo (BDMC)} method, and it was first developed in the experimental works \cite{ProkofievSvistunov2,ProkofievSvistunov3,ProkofievSvistunov1}. The same method was later used in the context of summation of a formal summation of Feynman graphs in \cite{Zwierlein2}.

\subsection{Ideas and techniques used in the proofs}

We will construct local solutions to the randomized hierarchies \eqref{Introduction_RandomizedGP1} and \eqref{Introduction_RandomizedGP2} by following the strategy given in \cite{CP4} in the deterministic setting. In order to adapt this strategy to the randomized setting, we will need to use the randomized estimates from \cite{SoSt}. In the study of the independently randomized hierarchy \eqref{Introduction_RandomizedGP2}, it is necessary to pay close attention to which randomization parameters occur in the expressions which we are considering, whereas in analyzing the dependently randomized hierarchy \eqref{Introduction_RandomizedGP1}, one works directly with the initial data $\Gamma(0)$. Each of these points will be explained in more detail below.
\subsubsection{\textbf{The strategy from \cite{CP4}}}
Let us first recall the strategy developed in \cite{CP4} for the deterministic problem.
The starting point for this strategy is to consider the \emph{truncated problem} associated to \eqref{GrossPitaevskiiHierarchy}, which is given in \eqref{TruncatedGP} above. In other words, for fixed $N \in \mathbb{N}$, one replaces in \eqref{GrossPitaevskiiHierarchy} the full initial data $\Gamma(0)=(\gamma_0^{(k)})_{k}$ with the truncated initial data $\mathbf{P}_{\leq N} \Gamma(0)=(\gamma_1^{(0)},\gamma_2^{(0)},\ldots,\gamma_N^{(0)},0,0,\ldots)$. Moreover, one assumes that the condition \eqref{TruncatedGPcondition} holds. In particular, by using this condition when $k=N+1$, it follows that $\gamma_N^{(N)}$ solves:

\begin{equation}
\notag
\begin{cases}
i \partial_t \, \gamma_N^{(N)} + (\Delta_{\vec{x}_N}-\Delta_{\vec{x}'_N}) \,\gamma_N^{(N)}=0\\
\big[\gamma_N^{(N)}\big]^{\omega^{*}}\big|_{t=0}=\gamma_0^{(N)}.
\end{cases}
\end{equation}
It is hence possible to directly solve for $\gamma_N^{(N)}$ as:
\begin{equation}
\label{k=N}
\gamma_N^{(N)}(t)=\,\mathcal{U}^{(N)}(t) \, \gamma_0^{(N)}.
\end{equation}
One then substitutes \eqref{k=N} into \eqref{TruncatedGP} for $k=N-1$ and uses Duhamel's principle in order to solve for $\gamma_{N}^{(N-1)}$. This procedure is then iterated and one obtains an explicit solution $\Gamma_N =(\gamma^{(k)}_N)_{k}$ of \eqref{TruncatedGP} which satisfies \eqref{TruncatedGPcondition}. This explicit solution is written as a sum of Duhamel expansion terms. Having constructed $\Gamma_N$, the next step of the strategy from \cite{CP4} is to show that the sequence of density matrices $\Gamma_N$ converges to a limit $\Gamma$ as $N \rightarrow \infty$. Finally, the last step is to check that $\Gamma$ solves the Gross-Pitaevskii hierarchy. The norms in which one obtains convergence of $\Gamma_N$ to $\Gamma$ and the sense in which $\Gamma$ is a solution to the hierarchy are given precisely in the discussion below.

In \cite{CP4}, the authors take the full initial data $\Gamma(0)$ to belong to $\mathcal{H}^{\alpha}_{\xi'}$, for appropriate $\alpha$ and $\xi'>0$ and they deduce that the sequence $(\Gamma_N)_N$ is Cauchy in the space $L^{\infty}_{t \in [0,T]} \mathcal{H}^{\alpha}_{\xi}$, for $T, \xi>0$, which are determined by $\alpha$ and $\xi'$. Since $L^{\infty}_{t \in [0,T]} \mathcal{H}^{\alpha}_{\xi}$ is a complete metric space, the sequence $(\Gamma_N)_N$ converges  to a limit $\Gamma$ strongly in $L^{\infty}_{t \in [0,T]} \mathcal{H}^{\alpha}_{\xi}$. In practice, the latter step requires an intermediate step in which one checks that the sequence $(\widehat{B}\, \Gamma_N)_N$
is Cauchy in $L^{\infty}_{t \in [0,T]} \mathcal{H}^{\alpha}_{\xi''}$ for an appropriate $\xi'' \in (\xi,\xi')$. The operator $\widehat{B}\,$ is defined in \eqref{B_hat} below. At this point of the analysis, one has to explicitly estimate the Duhamel expansions obtained in the construction of $\Gamma_N\,$. In order to achieve this in the deterministic setting, one recombines the large number of terms that occur in the Duhamel expansion. The recombination of terms is done by applying the combinatorial \emph{boardgame argument} in \cite{KM}, which is based on the earlier Feynman graph techniques in \cite{ESY1,ESY2,ESY3,ESY4,ESY5}. The sequence $(\widehat{B}\,\Gamma_N)_N$ is hence shown to converge strongly to a limit $\theta$ in $L^{\infty}_{t \in [0,T]} \mathcal{H}^{\alpha}_{\xi''}$. Using this fact and recalling \eqref{TruncatedGP}, the authors are able to show that the sequence $(\Gamma_N)_N$ is Cauchy in $L^{\infty}_{t \in [0,T]}\mathcal{H}^{\alpha}_{\xi}$ and hence has that it has a strong limit $\Gamma \in L^{\infty}_{t \in [0,T]}\mathcal{H}^{\alpha}_{\xi}$.

The final step is to show that $\Gamma\, \in L^{\infty}_{t \in [0,T]}$ is a local-in-time solution to \eqref{GrossPitaevskiiHierarchy} in the following sense:
$$\big\|\Gamma(t)-\mathcal{U}(t) \, \Gamma(0)\, + i\int_{0}^{t} \,\mathcal{U}(t-s)\,\widehat{B}\,\Gamma(s)\,ds\,\big\|_{L^{\infty}_{t \in [0,T]} \mathcal{H}^{\alpha}_{\xi}}=0$$
The strategy in \cite{CP4} is originally stated for $\Lambda=\mathbb{R}^d$. It carries through to the case $\Lambda=\mathbb{T}^d$ as long as one can prove the key spacetime estimate \eqref{SpacetimeBoundalpha} for the regularity exponent $\alpha$. For a more detailed discussion about this point, we refer the reader to Remark 1.6 in \cite{GSS}.

\subsubsection{\textbf{The analysis in the randomized setting}}
Our goal is to adapt the strategy from \cite{CP4} to the randomized setting. As we will see, the randomization will present several challenges in the analysis.  The key to overcome these challenges is to use the tools developed in \cite{SoSt}. In the discussion that follows, the spatial domain will be $\Lambda=\mathbb{T}^3$. Moreover, we will always assume that $\alpha \geq 0$ in our study of the dependently randomized case \eqref{Introduction_RandomizedGP1}, and that $\alpha>\frac{3}{4}$ in our study of the independently randomized case \eqref{Introduction_RandomizedGP2}.
We will first outline the main ideas used in the study of the independently randomized case \eqref{Introduction_RandomizedGP2}, since the analysis is simpler in this context. After that, we will summarize the techniques used in the study of the dependently randomized case \eqref{Introduction_RandomizedGP1}.

The construction of solutions to truncated hierarchies corresponding to \eqref{Introduction_RandomizedGP1}  and \eqref{Introduction_RandomizedGP2} carries over without much change from the deterministic setting. The explicit solutions are given in \eqref{gammaNkomega} and \eqref{gammaNkomega2} below. Here, we use the shorthand notation for the Duhamel expansion terms in the randomized setting given in \eqref{Duhamel_0}, \eqref{Duhamel_j}, \eqref{Duhamel_02}, and \eqref{Duhamel_j2}.

The step in which one shows that the solutions to the truncated hierarchies converge in an appropriate norm to a solution of the full hierarchy is significantly different in the randomized setting. We first analyze the independently randomized hierarchy \eqref{Introduction_RandomizedGP2}. In this case, we construct the solutions $[\Gamma_N]^{\omega^{*}}$ to the associated truncated hierarchy and we show that $([\widehat{B}\,]^{\omega^{*}}[\Gamma_N]^{\omega^{*}})_N$ is Cauchy in the space $L^{\infty}_{t \in [0,T]} L^2(\Omega^{*}) \mathcal{H}^{\alpha}_{\xi}$, defined in \eqref{Space2B} below. In order to prove that $([\widehat{B}\,]^{\omega^{*}} [\Gamma_N]^{\omega^{*}})_N$ is Cauchy, one needs to have a good bound on the Duhamel terms. As was noted in Subsection 5.1 of \cite{SoSt}, it is not possible to apply the combinatorial recombination argument in this context. Instead, we argue as in \cite{SoSt}, and we use the following integral identity:
\begin{equation}
\notag
%\label{Integral_Identity_Introduction}
\int_{0}^{t_{j}} \int_{0}^{t_{j+1}} \cdots \int_{0}^{t_{j+k-1}} \,dt_{j+k}\,\cdots \,dt_{j+2} \,dt_{j+1} = \frac{t_j^k}{k!}\,.
\end{equation}
Let us note that, in this way, we obtain a factorial gain in the denominator. This compensates for the factorially large number of terms in the Duhamel expansion. We can use this gain due to the fact that there is no time integral in the estimate \eqref{Theorem1bound2}. 
A similar method was first used in the study of the one-dimensional deterministic problem in \cite{CP2}. The estimate that we use for the Duhamel expansions in this context is given in Proposition \ref{DuhamelEstimate1}. The proof of the Cauchy property of $([\widehat{B}\,]^{\omega^{*}}[\Gamma_N]^{\omega^{*}})_N$ is given in Proposition \ref{CauchySequence1}.

As was seen in Example 1 from Subsection 6.1 of \cite{SoSt}, it is not always possible to estimate $[B_{j,k+1}]^{\omega} \gamma^{(k+1)}$ by arguing as in the proof of \eqref{Theorem1bound2} in the case when $\gamma^{(k+1)}$ has some additional $\omega$ dependence. In order to avoid this problem when working with $\theta^{\,\omega^{*}}=\lim_{N \rightarrow \infty}([\widehat{B}\,]^{{\omega}^{*}} [\Gamma_N]^{\omega^{*}})_N$, we need to record which random parameter occurs in each element of the limiting sequence. In particular, in Proposition \ref{limit2}, we obtain that, for all $k \geq 2$, $\big(\theta^{\,\omega^{*}}\big)^{(k)}$ does not depend on $\omega_2, \omega_3, \ldots, \omega_k$. Here, we recall the notation $\omega^{*}=(\omega_2,\omega_3,\omega_4,\ldots)$. In the further analysis, we use Proposition \ref{RandomizedEstimate2}, which gives us a way of avoiding the difficulty when we want to apply \eqref{Theorem1bound2} in the case when there are several random parameters. In particular, we note that it is possible to estimate $[B_{j,k+1}]^{\omega} \gamma^{(k+1)}$ when $\gamma^{(k+1)}$ depends on a sequence of random parameters $(\tilde{\omega}_j)_j$, all of which are independent from $\omega$. As in \cite{CP4}, we use a limiting procedure and note that, $\theta^{\,\omega^{*}}$ satisfies:
\begin{equation}
\label{theta_omega_star_Introduction}
\big\|\theta^{\,\omega^{*}}-\big[\widehat{B}\,\big]^{\omega^{*}}\,\mathcal{U}(t)\,\Gamma(0)+i\int_{0}^{t} \big[\widehat{B}\,\big]^{\omega^{*}}\,\mathcal{U}(t-s)\,\theta^{\,\omega^{*}}(s)\,ds \,\big\|_{L^{\infty}_{t \in [0,T]} L^2(\Omega^{*}) \mathcal{H}^{\alpha}_{\xi_0}}=0,
\end{equation}
for all $\xi_0>0$.
This result is given as Proposition \ref{theta_omega_star_equation}.
In the randomized setting, we need to use the observation from Proposition \ref{RandomizedEstimate2} and the known dependence of each component of $\theta^{\,\omega^{*}}$ on the random parameters in order to finish the proof of \eqref{theta_omega_star_Introduction}.

We show that $([\Gamma_N]^{\omega^{*}})_N$ is Cauchy in $L^{\infty}_{t \in [0,T]} L^2(\Omega^{*}) \mathcal{H}^{\alpha}_{\xi}$ in Proposition \ref{CauchySequence2}. Similarly as for $\theta^{\,\omega^{*}}$ in Proposition \ref{limit2}, we note that the limit $[\Gamma\,]^{\omega^{*}}$ has the property that, for all $k \geq 2$, $\big([\Gamma\,]^{\omega^{*}}\big)^{(k)}$ does not  depend on the random parameters $\omega_2, \omega_3, \ldots, \omega_k$. As in \cite{CP4}, in the proof of the Cauchy property of $([\Gamma_N]^{\omega^{*}})_N$, we need to use the Cauchy property of $([\widehat{B}\,]^{\omega^{*}} [\Gamma_N]^{\omega^{*}})_N$. This step is given in \eqref{second_termGammaN1N2} below. In Proposition \ref{Gamma_omega_star_equation}, we note that $[\Gamma\,]^{\omega^{*}}$ and $\theta^{\,\omega^{*}}$ are linked by the equation:
\begin{equation}
\notag
\big\|\big[\Gamma\,\big]^{\omega^{*}}-\mathcal{U}(t)\,\Gamma(0)+i\int_{0}^{t} \mathcal{U}(t-s)\,\theta^{\,\omega^{*}}(s)\,ds \,\big\|_{L^{\infty}_{t \in [0,T]} L^2(\Omega^{*}) \mathcal{H}^{\alpha}_{\xi_0}}=0,
\end{equation}
for all $\xi_0>0$.
We can put all of these results together and deduce Theorem \ref{Theorem1}, which states that, for all $\xi_0>0$ the limit $[\Gamma\,]^{\omega^{*}}$ satisfies:
\begin{equation}
\notag
\big\|\big[\Gamma\,\big]^{\omega^{*}}(t)-\mathcal{U}(t)\,\Gamma(0)+i \int_{0}^{t}\,\mathcal{U}(t-s)\,\big[\widehat{B}\,\big]^{\omega^{*}} \big[\Gamma\,\big]^{\omega^{*}}(s)\,ds\,\big\|_{L^{\infty}_{t \in [0,T]} L^2(\Omega^{*}) \mathcal{H}^{\alpha}_{\xi_0}}=0.
\end{equation}

If we assume the initial data to be infinitesimally more regular than $\alpha$, i.e. if $\Gamma(0) \in \mathcal{H}^{\alpha_0}_{\xi'}$ for some $\alpha_0>0$, it is possible to prove that $[\Gamma\,]^{\omega^{*}}$ is \emph{continuous in time} in the sense that it belongs to the space $C_{t \in [0,T]} L^2(\Omega^{*}) \mathcal{H}^{\alpha}_{\xi}$ defined in \eqref{Space3C} below. The key in this analysis is to use the difference estimate from Lemma \ref{DifferenceBound}, which states that, for all density matrices $\sigma^{(k)}$ of order $k$, for all $\beta_0>\beta>0$ and for all $t \in \mathbb{R}, \delta>0$, the following estimate holds:
\begin{equation}
\label{DifferenceBound_Introduction}
\big\|\big(\,\mathcal{U}^{(k)}(t+\delta)-\mathcal{U}^{(k)}(t)\big) \sigma_0^{(k)}\big\|_{H^{\beta}(\Lambda^k \times \Lambda^k)} \leq C \delta^{r} \, \cdot \, \big\|\sigma_0^{(k)}\big\|_{H^{\beta_0}(\Lambda^k \times \Lambda^k)}.
\end{equation}
Here, the constants $C,r>0$ depend on $\beta$ and $\beta_0$. We use \eqref{DifferenceBound_Introduction} to prove that $([\Gamma_N]^{\omega^{*}})_N$ \emph{is equicontinous in} $L^2(\Omega^{*}) \mathcal{H}^{\alpha}_{\xi}$ in the sense of Definition \ref{equicontinuousOmegastar} below.
Here, we use the sequence $([\Gamma_N]^{\omega^{*}})_N)$ which was constructed earlier, but for possibly a slightly smaller $T$ and for a $\xi'$ chosen to be possibly smaller in terms of $\xi$. The reason for this change in the parameters is the additional $\alpha_0$ dependence in the problem. We will typically not explicitly mention this distinction. The proof of the equicontinuity of $([\Gamma_N]^{\omega^{*}})_N$ and the related results are given in Theorem \ref{Theorem2}. The main point is that we can now obtain a solution $[\Gamma\,]^{\omega^{*}}$ of \eqref{Introduction_RandomizedGP2}, which is a solution \emph{pointwise in time}, as opposed to being a solution \emph{for almost all times $t$}.

In the context of the dependently randomized Gross-Pitaevskii hierarchy \eqref{Introduction_RandomizedGP1}, we need to argue differently. In particular, since there is only one random parameter $\omega$, it is not possible to apply the estimate from Proposition \ref{RandomizedEstimate2}, which relies on the random parameter occurring in the collision operator being independent from the random parameters occurring in the density matrix. We overcome this difficulty by explicitly writing out the Duhamel expansions in terms of the initial data and leaving them in this form. 
Let us again recall Example 1 from Subsection 6.1 of \cite{SoSt}, by which we can not estimate Duhamel expansions of order greater than 1 for arbitrary initial data. Instead, we need to restrict the initial data to to belong to a \emph{non-resonant class} of density matrices. The non-resonant class $\mathcal{N}$ is precisely defined in Definition \ref{Non-resonant} below. Let us note that this similar to the class of non-resonant density matrices which were previously used in \cite{SoSt}.

For non-resonant initial data, we recall an important randomized estimate, which can be deduced from \cite{SoSt} and which is given as Proposition \ref{non-resonant} below. A crucial observation is that we can use this estimate in order to bound the Duhamel expansion terms in Proposition \ref{DuhamelEstimate2}. In other words, we do not iteratively use \eqref{Theorem1bound2} as in the independently randomized case. Instead, we use Proposition \ref{DuhamelEstimate2} in order to estimate each separate Duhamel expansion term and we show that $([\widehat{B}\,]^{\omega} [\Gamma_N]^{\omega})_N$ is Cauchy in $L^{\infty}_{t \in [0,T]} L^2(\Omega) \mathcal{H}^{\alpha}_{\xi}$. This result is the content of Proposition \ref{CauchySequence1GP2}. The space $L^{\infty}_{t \in [0,T]} L^2(\Omega) \mathcal{H}^{\alpha}_{\xi}$ is defined in \eqref{Space3B} below.
Consequently, in Proposition \ref{limit2GP2}, we obtain that $([\widehat{B}\,]^{\omega}[\Gamma_N]^{\omega})_N$ converges strongly in $L^{\infty}_{t \in [0,T]} L^2(\Omega) \mathcal{H}^{\alpha}_{\xi}$ to a limit $\theta^{\,\omega}$. In Proposition \ref{theta_omega_equation}, we observe that $\theta^{\,\omega}$ satisfies:
\begin{equation}
\label{theta_omega_Introduction}
\big\|\theta^{\,\omega}-\big[\widehat{B}\,\big]^{\omega}\,\mathcal{U}(t)\,\Gamma(0)+i\int_{0}^{t} \big[\widehat{B}\,\big]^{\omega}\,\mathcal{U}(t-s)\,\theta^{\,\omega}(s)\,ds \,\big\|_{L^{\infty}_{t \in I} L^2(\Omega) \mathcal{H}^{\alpha}_{\xi_0}}=0,
\end{equation}
for all $\xi_0>0$. The proof of \eqref{theta_omega_Introduction} differs from the proof of \eqref{theta_omega_star_Introduction} in the sense that we are not allowed to iteratively use randomized estimates due to the fact that there is only one random parameter $\omega$. In this context, it is more difficult to estimate $\theta^{\,\omega}-[\widehat{B}\,]^{\omega} [\Gamma_N]^{\omega}$. This expression appears in the third term on the right-hand side of \eqref{theta_omega_differenceGP2} below. 
We estimate this term by first considering the truncation:
\begin{equation}
\notag
\big[\Phi_{N}^{M}\big]^{\omega}:=\big[\widehat{B}\,\big]^{\omega} \big[\Gamma_M\big]^{\omega}-\big[\widehat{B}\,\big]^{\omega} \big[\Gamma_N\big]^{\omega},
\end{equation}
defined in \eqref{Phi_N^M}, and by letting $M \rightarrow \infty$. We hence avoid working with infinite sums of terms involving only one random parameter. 

In Proposition \ref{CauchySequence2GP2}, it is shown that $([\Gamma_N]^{\omega})_N$ has a strong limit $[\Gamma\,]^{\omega}$ in $L^{\infty}_{t \in [0,T]} L^2(\Omega)  \mathcal{H}^{\alpha}_{\xi}.$ This step and the rest of the analysis are similar to the arguments given in the independently randomized case. Finally, in Theorem \ref{Theorem1GP2}, it is shown that $[\Gamma\,]^{\omega}$ satisfies:
\begin{equation}
\notag
\big\|\big[\Gamma\,\big]^{\omega}(t)-\mathcal{U}(t)\,\Gamma(0)+i \int_{0}^{t}\,\mathcal{U}(t-s)\,\big[\widehat{B}\,\big]^{\omega} \big[\Gamma\,\big]^{\omega}(s)\,ds\,\big\|_{L^{\infty}_{t \in [0,T]} L^2(\Omega) \mathcal{H}^{\alpha}_{\xi_0}}=0,
\end{equation}
for all $\xi_0>0$.

In the dependently randomized setting, it is also possible to consider infinitesimally more regular initial data and construct solutions of regularity $\alpha$ which are continuous in time. The main difference is in checking the equicontinuity properties of $\big(\big[\widehat{B}\,\big]^{\omega} \big[\Gamma_N\,\big]^{\omega}\big)_N$, where the random parameter $\omega$ appears more than once. This issue is resolved by looking at the explicit Duhamel expansions as before. In particular, we recall the analysis from Subsection 6.2 of \cite{SoSt} in Lemma \ref{Duhamel_Expansion} below. Using these ideas, we write out the difference of two Duhamel terms at different times in Lemma \ref{Duhamel_Expansion_Difference}. It is important to note that the frequencies in these expansions are treated as mutually distinct formal symbols to which we can assign the values $\pm 1$ in the definition of the sum given in \eqref{Sigma_Star} below. An explicit example of the difference of two Duhamel terms is given in Example \ref{Example}, following Lemma \ref{Duhamel_Expansion_Difference}. We can then use Lemma \ref{Duhamel_Expansion_Difference} in order to show that $[\Gamma\,]^{\omega}$ belongs to the space $C_{t \in [0,T]}L^2(\Omega) \mathcal{H}^{\alpha}_{\xi}$, which is defined in \eqref{Space3C} below. This is the content of Theorem \ref{Theorem2GP2}. As in the independently randomized setting, we obtain a  solution to the hierarchy which is defined \emph{pointwise in $t$}.

\subsection{Organization of the paper}

In Section \ref{Notation}, we define the relevant notation and we recall several useful general facts. Section \ref{The independently randomized Gross-Pitaevskii hierarchy} is devoted to the study of the independently randomized Gross-Pitaevskii hierarchy \eqref{Introduction_RandomizedGP2}. 
In Proposition \ref{DuhamelEstimate1}, we prove a randomized Duhamel estimate.
Furthermore, in Subsection \ref{The truncated Gross-Pitaevskii hierarchy corresponding to RandomizedGP1}, we define and analyze the associated truncated hierarchy. In Subsection \ref{A local-in-time result for initial data of regularity alpha}, we construct solutions in $L^{\infty}_{t \in [0,T]} L^2(\Omega^{*}) \mathcal{H}^{\alpha}_{\xi}$ to the full hierarchy \eqref{Introduction_RandomizedGP2}, which evolve from initial data of regularity $\alpha>\frac{3}{4}$. The main result is given in Theorem \ref{Theorem1}. Subsection \ref{additional_regularity} is devoted to the construction of solutions in $L^{\infty}_{t \in [0,T]} L^2(\Omega^{*}) \mathcal{H}^{\alpha}_{\xi}$, which evolve from initial data of regularity $\alpha_0>\alpha$ and which are continuous in time. The main result of this subsection is the content of Theorem \ref{Theorem2}. 

Section \ref{The dependently randomized Gross-Pitaevskii hierarchy} is devoted to the study of the dependently randomized Gross-Pitaevskii hierarchy \eqref{Introduction_RandomizedGP1}. 
In Proposition \ref{DuhamelEstimate2}, we prove an estimate for Duhamel expansions evolving from non-resonant initial data. The definition of the non-resonant class is given in Definition \ref{Non-resonant}. 
Subsection \ref{The truncated Gross-Pitaevskii hierarchy corresponding to RandomizedGP2} is devoted to the study the corresponding truncated hierarchy. In Subsection \ref{A local-in-time result for non-resonant initial data of regularity alpha}, we construct local-in-time solutions in $L^{\infty}_{t \in [0,T]} L^2(\Omega^{*}) \mathcal{H}^{\alpha}_{\xi}$ to the full hierarchy \eqref{Introduction_RandomizedGP1}, which evolve from non-resonant initial data of regularity $\alpha \geq 0$. The main result of this subsection is given in Theorem \ref{Theorem1GP2}.
Similarly as Subsection \ref{additional_regularity}, Subsection \ref{additional_regularityGP2}  is devoted to the construction of solutions in $L^{\infty}_{t \in [0,T]} L^2(\Omega) \mathcal{H}^{\alpha}_{\xi}$, which evolve from initial data of regularity $\alpha_0>\alpha$, and which are continuous in time. 
Sub-subsection \ref{An explicit formula for the difference of two Duhamel expansions} is devoted to obtaining an explicit formula for the difference of two Duhamel expansions. This formula is given in Lemma \ref{Duhamel_Expansion_Difference}. The main result of Subsection \ref{additional_regularityGP2} is Theorem \ref{Theorem2GP2}, which is proved in Sub-subsection \ref{Proof of Theorem2GP2}.

\subsection{Acknowledgements} The author would like to thank Gigliola Staffilani for valuable discussions. He would also like to thank Thomas Chen, Philip Gressman, Antti Knowles, Andrea Nahmod, and Nata\v{s}a Pavlovi\'{c} for their helpful comments. V.S. was supported by a Simons Postdoctoral Fellowship.
 
\section{Notation}
\label{Notation}

Let us introduce some notation and terminology which we will use in our paper. 
If $a$ and $b$ are positive quantities, we say that $a \lesssim b$ if $a \leq Cb$ for some constant $C>0$. We also sometimes write this as $a=\mathcal{O}(b)$ and we implicitly keep track of the constant $C$. If $C$ depends on the quantities $q_1,q_2,\ldots,q_k$, then we write $a \lesssim_{\,q_1,\,q_2,\ldots,\,q_k} b$. 
If $a \lesssim b$ and $b \lesssim a$, then we write $a \sim b$.

Throughout the paper, $\Lambda$ will denote the spatial domain, which we will take to be $\mathbb{T}^3$, unless it is specified otherwise. The discussions in the paper will also apply for $\Lambda=\mathbb{T}^{d}$, when $d \geq 1$. For a more detailed discussion, we refer the reader to Remarks \ref{Remark_higher_dimensions1} and \ref{Remark_higher_dimensions2} below. 

We will use the notational convention that the sum $\sum_{i=A}^{B}(\cdots)$ equals zero if $A>B$. This convention will be used, for instance, in \eqref{gammaNkomega} and in \eqref{gammaNkomega2} below.
We will sometimes abbreviate the Gross-Pitaevskii hierarchy as the \emph{GP hierarchy}.

$I$ will usually denote a time interval of the form $[0,T]$ for a fixed $T>0$.

In our notation, the greek letter $\zeta$ will be used to denote a frequency variable. Frequency variables will also be denoted by $\xi_{k}$ for integer indices $k \geq 1$. 
The symbols $\xi,\xi'$ and $\xi_0$ will usually denote regularity parameters in the spaces defined in equations \eqref{Space1}, \eqref{Space2}, \eqref{Space3} below. 

Given a spatial domain $\Lambda$ as above, we define a \emph{density matrix of order $k$} or $k$-\emph{particle density matrix} to be a function:
\begin{equation}
\notag
\sigma^{(k)}: \Lambda^k \times \Lambda^k \rightarrow \mathbb{C}.
\end{equation}

As was noted above, in our paper, \emph{we will work on the spatial domain $\Lambda:=\mathbb{T}^3$}, unless it is noted otherwise.
Hence, the frequencies will be elements of $\mathbb{Z}^3$. For a given function $f \in L^2(\Lambda)$, and for a frequency $\xi \in \mathbb{Z}^3$, we define the Fourier transform of $f$ evaluated at $\xi$ by:
$$\widehat{f}(\xi):=\int_{\Lambda} f(x) e^{- i \cdot \langle x, \xi \rangle} \,dx.$$
The quantity $\langle \cdot, \cdot \rangle$ denotes the inner product on $\mathbb{R}^3$.
\\
When defining the Fourier transform of a density matrix, we will use the same convention as in \cite{GSS,SoSt}. In particular, given $\sigma^{(k)}$, a density matrix of order $k$ which belongs to $L^2(\Lambda^k \times \Lambda^k)$, we define its Fourier transform 
$(\sigma^{(k)})\,\,\widehat{}\,\,$ as follows:

For all $\vec{\xi}_k=(\xi_1,\ldots,\xi_k),\vec{\xi}'_k=(\xi'_1,\ldots,\xi'_k) \in (\mathbb{Z}^3)^k$:
$$(\sigma^{(k)})\,\,\widehat{}\,\,(\vec{\xi}_k;\vec{\xi}'_k): = \int_{\Lambda^k \times \Lambda^k}\sigma^{(k)}(\vec{x}_k;\vec{x}'_k)
e^{-i \cdot \sum_{j=1}^{k} \langle x_j, \xi_j \rangle + i \cdot \sum_{j=1}^{k} \langle x'_j, \xi'_j \rangle} \,d\vec{x}_k \, d\vec{x}'_k.$$
We will usually write the Fourier transform $(\sigma^{(k)})\,\,\widehat{}\,\,$ of the density matrix $\sigma^{(k)}$ as $\widehat{\sigma}^{(k)}$ for simplicity.

Let us consider the operator $i \partial_t + \big(\Delta_{\vec{x}_k}-\Delta_{\vec{x}'_k}\big)$, acting on density matrices of order $k$. Its associated \emph{free evolution operator} $\mathcal{U}^{(k)}(t)$ is defined as:
\begin{equation}
\label{free_evolution}
\mathcal{U}^{(k)}(t)\,\sigma^{(k)}:=e^{it \sum_{j=1}^{k} \Delta_{x_j}} \sigma^{(k)} e^{-it \sum_{j=1}^{k} \Delta_{x_j'}}.\end{equation}
Here $\sigma^{(k)}$ is a fixed density matrix of order $k$.
By construction:
$$\Big(i \partial_t + (\Delta_{\vec{x}_k}-\Delta_{\vec{x}_k'})\Big)\,\mathcal{U}^{(k)}(t)\,\sigma^{(k)}=0.$$
We note that $\mathcal{U}^{(k)}(t)$ is the analogue of the free Schr\"{o}dinger operator for $k$ density matrices of order $k$. 

Given $\alpha \in \mathbb{R}$, we define the operator of \emph{differentiation of order $\alpha$} which acts on density matrices of order $k$. This is done by using the Fourier transform. For $\sigma^{(k)}$, a density matrix of order $k$, we define a new density matrix of order $k$, $S^{(k,\alpha)}\sigma^{(k)}$, by:
\begin{equation}
\label{FractionalDerivative}
\big(S^{(k,\alpha)} \sigma^{(k)} \big)\,\,\widehat{}\,\,(\xi_1,\ldots,\xi_k;\xi'_1,\ldots,\xi'_k):=
\end{equation}
$$\langle \xi_1 \rangle^{\alpha} \cdots \langle \xi_k \rangle^{\alpha} \cdot \langle \xi'_1 \rangle^{\alpha} \cdots \langle \xi'_k \rangle^{\alpha} \cdot \widehat{\sigma}^{(k)} (\xi_1, \ldots, \xi_k;\xi'_1, \ldots, \xi'_k).$$ 
Here,
$$\langle x \rangle : = \sqrt{1+|x|^2}$$
denotes the \emph{Japanese bracket}. 

We will sometimes write
$\|S^{(k,\alpha)} \sigma^{(k)}\|_{L^2(\Lambda^k \times \Lambda^k)}$ 
as $\|\sigma^{(k)}\|_{H^{\alpha}(\Lambda^k \times \Lambda^k)}.$
This is consistent with the definition of fractional Sobolev norms for functions. We will use both notations in our paper.

The \emph{collision operator} is an operator which acts on density matrices. More precisely, given $k \in \mathbb{N}$ and $j \in \{1,\ldots,k\}$, the collision operator $B_{j,k+1}$ acts on density matrices of order $k+1$. For $\sigma^{(k+1)}$, a fixed density matrix of order $k+1$, we define: 
\begin{equation}
\notag
B_{j,k+1}\, \big(\sigma^{(k+1)}\big):=B^{+}_{j,k+1}\big(\sigma^{(k+1)}\big)-B^{-}_{j,k+1}\big(\sigma^{(k+1)}\big),
\end{equation}
where:

\begin{equation}
\notag
B^{+}_{j,k+1} \,\big(\sigma^{(k+1)}\big)(\vec{x}_k; \vec{x}_k'):=\int_{\Lambda} \delta(x_j-x_{k+1})\, \sigma^{(k+1)}(\vec{x}_k,x_{k+1};\vec{x}_k',x_{k+1})\, dx_{k+1}
\end{equation}
and
\begin{equation}
\notag
B^{-}_{j,k+1} \,\big(\sigma^{(k+1)}\big)(\vec{x}_k; \vec{x}_k'):=\int_{\Lambda} \delta(x_j'-x_{k+1}) \,\sigma^{(k+1)}(\vec{x}_k,x_{k+1};\vec{x}_k',x_{k+1})\, dx_{k+1}.
\end{equation}
Here $\delta$ denotes the Dirac delta function.
For example, when $j=1$:
$$B_{1,k+1} \, \big(\sigma^{(k+1)}\big)(\vec{x}_k;\vec{x}_k')= \int_{\Lambda} \sigma^{(k+1)} (\vec{x}_k,x_1;\vec{x}_k',x_1)\,dx_1-\int_{\Lambda} \sigma^{(k+1)} (\vec{x}_k,x'_1;\vec{x}_k',x'_1)\,dx'_1.$$
The calculation for general $j \in \{1,\ldots,k\}$ is similar.
In particular, we note that $B_{j,k+1}$ is a linear operator which maps density matrices of order $k+1$ to density matrices of order $k$. We will usually abbreviate $B_{j,k+1} \, \big(\sigma^{(k+1)}\big)$ by $B_{j,k+1} \,\sigma^{(k+1)}$. The \emph{full collision operator} $B^{(k+1)}$ is given by:

\begin{equation}
\label{Bk+1_sum}
B^{(k+1)}:=\sum_{j=1}^{k} B_{j,k+1}.
\end{equation}

The Fourier transform of the density matrix $B^{+}_{1,k+1} \, \sigma^{(k+1)}$ can be computed as:
\begin{equation}
\label{FourierTransform1}
(B^{+}_{1,k+1} \, \sigma^{(k+1)})\,\,\widehat{}\,\,(\vec{\xi}_k; \vec{\xi'}_k)= 
\end{equation}
$$\sum_{{\xi}_{k+1}, \, {\xi'}_{k+1} \in \mathbb{Z}^3} \, \widehat{\sigma}^{(k+1)}(\xi_1-\xi_{k+1}+\xi'_{k+1}, \xi_2, \ldots, \xi_k, \xi_{k+1}; \xi'_1,\ldots, \xi'_k,\xi'_{k+1}).$$
Analogously, we can compute:
\begin{equation}
\label{FourierTransform2}
(B^{-}_{1,k+1} \, \sigma^{(k+1)})\,\,\widehat{}\,\,(\vec{\xi}_k; \vec{\xi'}_k)= 
\end{equation} 
$$\sum_{{\xi}_{k+1}, \, {\xi'}_{k+1} \in \mathbb{Z}^3} 
\,\widehat{\sigma}^{(k+1)}(\xi_1, \ldots, \xi_k, \xi_{k+1}; \xi'_1-\xi'_{k+1}+\xi_{k+1}, \xi'_2, \ldots, \xi'_k,\xi'_{k+1}).$$
The formula for general $(B^{\pm}_{j,k+1} \, \sigma^{(k+1)})\,\,\widehat{}\,\,(\vec{\xi}_k; \vec{\xi'}_k)$ is similar.

Let us recall the definition of the \emph{randomized collision operator} from \cite{SoSt}. We first define a sequence of random variables indexed by the set of frequencies. In the case $\Lambda=\mathbb{T}^3$, this indexing set will be $\mathbb{Z}^3$. More precisely, we consider the sequence $(h_{\zeta})_{\zeta \in \mathbb{Z}^3}$ of independent, identically distributed real Bernoulli random variables which have expected value zero and standard deviation $1$. In other words, each $h_{\zeta}$ takes value $1$ or $-1$ with probability $\frac{1}{2}$. 

Having defined $h_{\zeta}$ as above, we can define the \emph{randomization of a function}. Namely, given $f \in L^2(\Lambda)$ and $\omega \in \Omega$, we define the function $f^{\omega}$ by:
$$\widehat{f^{\omega}}(\zeta):=h_{\zeta}(\omega) \cdot \widehat{f}(\zeta),$$
whenever $\zeta \in \mathbb{Z}^3$.
In other words, if $f$ is given as the Fourier series:
$$f(x)=\sum_{\zeta \in \mathbb{Z}^3} c_{\zeta} \cdot e^{i x \zeta},$$
then the randomization of $f$, which we denote by $f^{\omega}$, is given as:
\begin{equation}
\label{randomization_of_a_function}
f^{\omega}(x):=\sum_{\zeta \in \mathbb{Z}^3} h_{\zeta}(\omega) \cdot c_{\zeta} \cdot e^{ix\zeta}.
\end{equation}
By Plancherel's Theorem, we note that $f^{\omega} \in L^2(\Lambda)$ and that $\|f^{\omega}\|_{L^2}=\|f\|_{L^2}$. Here, we are using the precise choice of $(h_{\zeta})_{\zeta}$ as a sequence of Bernoulli random variables.

As in \cite{SoSt}, we use idea of randomizing on the Fourier domain and we \emph{randomize} the collision operator $B_{j,k+1}$ in order to obtain the \emph{randomized collision operator} $[B_{j,k+1}]^{\omega}$. For $(h_{\zeta})_{\zeta \in \mathbb{Z}^3}$, as above, we recall \eqref{FourierTransform1} and we define, for fixed $\omega \in \Omega$ and for a fixed density matrix $\sigma^{(k+1)}$ of order $k+1$:
\begin{equation}
\label{Bjkomega}
([B^{+}_{1,k+1}]^{\omega} \, \sigma^{(k+1)})\,\,\widehat{}\,\,(\vec{\xi}_k; \vec{\xi}_k'):= 
\end{equation}
$$ \sum_{\xi_{k+1},\,\xi'_{k+1} \in \mathbb{Z}^3} h_{\xi_1}(\omega) \cdot h_{\xi_1-\xi_{k+1}+\xi'_{k+1}}(\omega) \cdot h_{\xi_{k+1}}(\omega) \cdot h_{\xi'_{k+1}}(\omega) 
$$
$$\cdot \, \widehat{\sigma}^{(k+1)}(\xi_1-\xi_{k+1}+\xi'_{k+1}, \xi_2, \ldots, \xi_k, \xi_{k+1}; \xi'_1,\ldots, \xi'_k,\xi'_{k+1}).$$
Furthermore, we recall \eqref{FourierTransform2} and we define:

\begin{equation}
\label{Bjkomega2}
([B^{-}_{1,k+1}]^{\omega} \sigma^{(k+1)})\,\,\widehat{}\,\,(\vec{\xi}_k; \vec{\xi}_k'):= 
\end{equation}
$$ \sum_{\xi_{k+1},\,\xi'_{k+1} \in \mathbb{Z}^3} h_{\xi_1}(\omega) \cdot h_{\xi_1-\xi'_{k+1}+\xi_{k+1}}(\omega) \cdot h_{\xi_{k+1}}(\omega) \cdot h_{\xi'_{k+1}}(\omega) 
$$
$$\cdot \, \widehat{\sigma}^{(k+1)}(\xi_1, \ldots, \xi_k, \xi_{k+1}; \xi'_1-\xi'_{k+1}+\xi_{k+1}, \xi'_2, \ldots, \xi'_k,\xi'_{k+1}).$$

The quantity $[B^{\pm}_{j,k+1}]^{\omega}$ is defined similarly for $j \in \{1,2,\ldots,k\}$. The \emph{randomized collision operator} is then defined as:
\begin{equation}
\label{Bjkrandomized}
[B_{j,k+1}]^{\omega}:=[B^{+}_{j,k+1}]^{\omega}-[B^{-}_{j,k+1}]^{\omega}.
\end{equation}
By construction, it follows that $[B_{j,k+1}]^{\omega}=B_{j,k+1}$, if $\omega \in \Omega$ is chosen such that $h_{\zeta}(\omega)=1$ for all $\zeta \in \mathbb{Z}^3$ or if $h_{\zeta}(\omega)=-1$ for all $\zeta \in \mathbb{Z}^3$. Furthermore, by analogy with \eqref{Bk+1_sum}, we define the \emph{full randomized collision operator} $\big[B^{(k+1)}\,\big]^{\omega}$ as:
\begin{equation}
\label{Bomegak+1}
\big[B^{(k+1)}\,\big]^{\omega}:=\sum_{j=1}^{k} \big[B_{j,k+1}\big]^{\omega}.
\end{equation}
As a convention, will extend the definition of $B_{j,k+1}$ and $\big[B_{j,k+1}\,\big]^{\omega}$ to density matrices $\sigma^{(\ell)}$ of order $\ell>k+1$. This is done by acting by the collision operators only in the variables $\vec{x}_{k+1}$ and $\vec{x}'_{k+1}$ and by treating all the other variables as parameters. We will use this convention in Subsection \ref{additional_regularityGP2}, more precisely in Sub-subsection \ref{An explicit formula for the difference of two Duhamel expansions}. In particular, we apply it in Lemma \ref{Duhamel_Expansion}, Lemma \ref{Duhamel_Expansion_Difference}, and in Example \ref{Example} of Sub-subsection \ref{An explicit formula for the difference of two Duhamel expansions}.

Let us examine more closely the sequence of Bernoulli random variables $(h_{\zeta})_{\zeta \in \mathbb{Z}^3}$ defined above.
We will denote the probability space associated to $(h_{\zeta})_{\zeta \in \mathbb{Z}^3}$ by $(\Omega,\Sigma,p)$. Here, $\Sigma$  is the corresponding sigma-algebra and $p$ is the probability measure. We will usually denote the probability space just by $\Omega$. By $L^2(\Omega)$, we will denote the $L^2$ space on $\Omega$, with respect to the probability measure $p$.

Let us define:
$$\Omega^{*}:=\prod_{k \geq 2} \Omega_k.$$
Here, each $\Omega_k=\Omega$ is the probability space associated to a sequence of Bernoulli trials as defined above. We denote each copy of the probability space by $(\Omega_k,\Sigma_k,p_k)$, in order to distinguish the factors. Here $\Sigma_k=\Sigma$ is the sigma-algebra and $p_k=p$ is the probability measure on $\Omega_k=\Omega$.
By construction, every $\omega^{*} \in \Omega^{*}$ is of the form
$$\omega^{*}=(\omega_2,\omega_3,\omega_4, \ldots)$$
where each $\omega_j \in \Omega$. We know that $\Omega^{*}$ can be given the structure of an infinite product probability space, following the work of Kakutani \cite{Kakutani}, which, in turn, is based on the work of Kolmogorov \cite{Kolmogorov}. The precise statement which we will use is the following:

\begin{theorem} 
(\cite{Kakutani},\cite{Kolmogorov})
\label{Kakutani_Kolmogorov}
There exists a probability measure $p^{*}$ on $\Omega^{*}$ such that, for all finite subsets $F \subseteq \{2,3,\ldots\}$:
\begin{equation}
\notag
p^{*} \Big(\bigcap_{i \in F} \pi_i^{-1}(A_i)\Big)=\prod_{i \in F} p_i (A_i)
\end{equation}
whenever $A_i \in \Sigma_i$ for all $i \in F$. Here, for $j \in \{2,3,\ldots\}$, $\pi_j: \Omega^{*} \rightarrow \Omega_j$ denotes the canonical projection map onto $\Omega_j$.
\\
The obtained probability space is given by $(\Omega^{*},\mathcal{B},p^{*})$, where $\mathcal{B}$ is the Borel field generated by sets of the form:
$$\bigcap_{i \in F} \pi_i^{-1}(A_i)$$
for some finite subset $F \subseteq \mathcal{I}$ and for $A_i \in \Sigma_i$.
\end{theorem}
Let us note that there is also related work on infinite product measures by Doob \cite{Doob}. We will henceforth denote by $L^2(\Omega^{*})$ the $L^2$ space on $\Omega^{*}$, with respect to the probability measure $p^{*}$. This space is well-defined by Theorem \ref{Kakutani_Kolmogorov}. 

When working in $\Omega^{*}$, we will also be interested in elements of $\Omega^{*}$ which do not depend on a certain component, say $\omega_{k+1}$. We write:
\begin{equation}
\label{tildeomegak+1}
\tilde{\omega}_{k+1} \in \prod_{\ell \neq k+1} \Omega_{\ell}
\end{equation}
in order to emphasize that $\tilde{\omega}_{k+1} \in \Omega^{*}$ is a sequence of random parameters in which $\omega_{k+1}$ does not occur, i.e. in which all of the elements are independent from $\omega_{k+1}$. This concept will be useful in Proposition \ref{RandomizedEstimate2} below.

Throughout our paper, we will study sequences of the form $\Sigma=\big(\sigma^{(k)}\big)_k$, where each $\sigma^{(k)}$ is a density matrix of order $k$, possibly depending on other parameters. $\Sigma$ will typically belong to function spaces such as those defined below. More precisely,  let us fix $\alpha \in \mathbb{R}$ and $\xi>0$. We first define the space $\mathcal{H}^{\alpha}_{\xi}=\mathcal{H}^{\alpha}_{\xi}(\Lambda)$ by:
\begin{equation}
\label{Space1}
\big\|\Sigma\big\|_{\mathcal{H}^{\alpha}_{\xi}}:=\sum_{k \in \mathbb{N}} \xi^k \cdot \big\|\sigma^{(k)}\big\|_{H^{\alpha}(\Lambda^k \times \Lambda^k)}
\end{equation}
In this case, we assumed that each $\sigma^{(k)}$ was a density matrix of order $k$ which does not involve any other parameters. Such spaces were first defined in the work of T. Chen and Pavlovi\'{c} \cite{CP1}.

If the $\sigma^{(k)}$ depend on the $\Omega^{*}$-variable, i.e. if:
\begin{equation}
\notag
\sigma^{(k)}: \Omega^{*} \times \Lambda^{k} \times \Lambda^{k} \rightarrow \mathbb{C},
\end{equation}
for all $k$, then we define:
\begin{equation}
\label{Space2}
\big\|\Sigma\big\|_{L^2(\Omega^{*}) \mathcal{H}^{\alpha}_{\xi}}:=\sum_{k \in \mathbb{N}} \xi^k \cdot \big\|\sigma^{(k)}\big\|_{L^2(\Omega^{*}) H^{\alpha}(\Lambda^k \times \Lambda^k)}.
\end{equation}
Here, $\big\|\sigma^{(k)}\big\|_{L^2(\Omega^{*}) H^{\alpha}(\Lambda^k \times \Lambda^k)}$ means that we first take the $H^{\alpha}(\Lambda^k \times \Lambda^k)$ norm of $\sigma^{(k)}$ to obtain the function of $\omega^{*}$ given by $\big\|\sigma^{(k)}(\omega^{*},\cdot)\big\|_{H^{\alpha}(\Lambda^{k} \times \Lambda^{k})}$. We then take the $L^2(\Omega^{*})$ norm of the resulting function. It is important to note that in \eqref{Space2}, we are not taking $\mathcal{H}^{\alpha}_{\xi}$ norms of $\Sigma$ and then taking $L^2(\Omega^{*})$ norms of the result. 

If, in addition, the $\sigma^{(k)}$ depend on the time variable $t \in I$:
$$\sigma^{(k)} : I \times \Omega^{*} \times \Lambda^{k} \times \Lambda^{k} \rightarrow \mathbb{C},$$
then we define:
\begin{equation}
\label{Space2B}
\big\|\Sigma\big\|_{L^{\infty}_{t \in I} L^2(\Omega^{*}) \mathcal{H}^{\alpha}_{\xi}}:=\Big\|\big\|\Sigma(t)\big\|_{L^2(\Omega^{*}) \mathcal{H}^{\alpha}_{\xi}}\Big\|_{L^{\infty}_{t \in I}}=\Big\|\sum_{k \in \mathbb{N}} \xi^k \cdot \big\|\sigma^{(k)}(t)\big\|_{L^2(\Omega^{*}) H^{\alpha}(\Lambda^k \times \Lambda^k)}\Big\|_{L^{\infty}_{t \in I}}.
\end{equation}
If the map $t \mapsto \|\Sigma(t)\big\|_{L^2(\Omega^{*}) \mathcal{H}^{\alpha}_{\xi}}$ is continuous, then we say that:
\begin{equation}
\label{Space2C}
\Sigma \in C_{t \in I} L^2(\Omega^{*}) \mathcal{H}^{\alpha}_{\xi}.
\end{equation}
In this space, we use the norm given by \eqref{Space2B}.

Similarly, if the $\sigma^{(k)}$ depend on $\Omega$, i.e. if:
\begin{equation}
\notag
\sigma^{(k)}: \Omega \times \Lambda^{k} \times \Lambda^{k} \rightarrow \mathbb{C},
\end{equation}
for all $k$, then we use the same conventions as in \eqref{Space2} and we define:

\begin{equation}
\label{Space3}
\big\|\Sigma\big\|_{L^2(\Omega) \mathcal{H}^{\alpha}_{\xi}}:=\sum_{k \in \mathbb{N}} \xi^k \cdot \big\|\sigma^{(k)}\big\|_{L^2(\Omega) H^{\alpha}(\Lambda^k \times \Lambda^k)}.
\end{equation}
Here, $\big\|\sigma^{(k)}\big\|_{L^2(\Omega) H^{\alpha}(\Lambda^k \times \Lambda^k)}$ is obtained by first applying the $H^{\alpha}(\Lambda^k \times \Lambda^k)$ norm to $\sigma^{(k)}$ and then applying the $L^2(\Omega)$ norm to the result. Similarly as in \eqref{Space2B}, if there is time-dependence:
$$\sigma^{(k)}: I \times \Omega \times \Lambda^{k} \times \Lambda^{k} \rightarrow \mathbb{C},$$ 
then we define:
\begin{equation}
\label{Space3B}
\big\|\Sigma\big\|_{L^{\infty}_{t \in I} L^2(\Omega) \mathcal{H}^{\alpha}_{\xi}}:=\Big\|\big\|\Sigma(t)\big\|_{L^2(\Omega) \mathcal{H}^{\alpha}_{\xi}}\Big\|_{L^{\infty}_{t \in I}}=\Big\|\sum_{k \in \mathbb{N}} \xi^k \cdot \big\|\sigma^{(k)}(t)\big\|_{L^2(\Omega) H^{\alpha}(\Lambda^k \times \Lambda^k)}\Big\|_{L^{\infty}_{t \in I}}.
\end{equation}
If the map $t \mapsto \|\Sigma(t)\big\|_{L^2(\Omega) \mathcal{H}^{\alpha}_{\xi}}$ is continuous, then we say that:
\begin{equation}
\label{Space3C}
\Sigma \in C_{t \in I} L^2(\Omega) \mathcal{H}^{\alpha}_{\xi}.
\end{equation}
In this space, we use the norm given by \eqref{Space3B}.

Each of the quantities \eqref{Space1}, \eqref{Space2}, \eqref{Space2B}, \eqref{Space3}, and \eqref{Space3B} defines a norm which gives rise to a Banach space. Moreover, the spaces in \eqref{Space2C} and \eqref{Space3C} are Banach spaces if we use the norms given in \eqref{Space2B} and \eqref{Space3B} respectively. In addition to \cite{CP1}, the space $\mathcal{H}^{\alpha}_{\xi}$ given in \eqref{Space1} was also used in \cite{CP2,CP4,CP3}. The spaces $L^2(\Omega^{*}) \mathcal{H}^{\alpha}_{\xi}$ and $L^2(\Omega) \mathcal{H}^{\alpha}_{\xi}$ coming from \eqref{Space2} and \eqref{Space3} respectively can be viewed as its probabilistic variants. Both the deterministic and the  probabilistic variants will be useful in our paper.

Given $\Sigma=(\sigma^{(k)})_k$ as above, we define a variant of the \emph{Laplace operator} applied to $\Sigma$ as in \cite{CP1,CP2,CP4,CP3}:
\begin{equation}
\label{Delta_plusminus}
\widehat{\Delta}_{\pm} \Sigma :=\big((\Delta_{\vec{x}_k}-\Delta_{\vec{x}'_k})\sigma^{(k)}\big)_k.
\end{equation}
We also define \emph{collision operators} applied to $\Sigma$.  
In the deterministic setting, we recall the definition from \cite{CP1,CP2,CP4,CP3}:

\begin{equation}
\label{B_hat}
\widehat{B}\,\,\Sigma:=\Big(B^{(k+1)}\,\sigma^{(k+1)}\Big)_k.
\end{equation}
Here $B^{(k+1)}$ is defined as in \eqref{Bk+1_sum}.

We recall \eqref{Bomegak+1} and we define two randomized versions of \eqref{B_hat}. Given $\omega^{*}=(\omega_2,\omega_3,\omega_4, \ldots) \in \Omega^{*}$, we define the operator $\big[\widehat{B}\,\big]^{\omega^{*}}$ by:
\begin{equation}
\label{B_hat_omega_star}
\big[\widehat{B}\,\big]^{\omega^{*}} \Sigma:=\Big(\big[B^{(k+1)}\,\big]^{\omega_{k+1}} \sigma^{(k+1)}\Big)_k.
\end{equation}

Given $\omega \in \Omega$, we define the operator $\big[\widehat{B}\,\big]^{\omega}$ by:
\begin{equation}
\label{B_hat_omega}
\big[\widehat{B}\,\big]^{\omega} \, \Sigma:=\Big(\big[B^{(k+1)}\,\big]^{\omega} \sigma^{(k+1)}\Big)_k.
\end{equation}

Given $N \in \mathbb{N}$ and $\Sigma=(\sigma^{(k)})_k$ as above, we define the projection operator $\mathbf{P}_{\leq N}$ applied to $\Sigma$ as:

\begin{equation}
\label{P<=N}
\mathbf{P}_{\leq N} \Sigma := (\sigma^{(1)},\sigma^{(2)},\ldots,\sigma^{(N)},0,0,\ldots).
\end{equation}
Furthermore, let us define:
\begin{equation}
\label{PgreaterN}
\mathbf{P}_{>N}\Sigma:=\Sigma-\mathbf{P}_{\leq N}\Sigma=(0,0,\ldots,0,\sigma^{(N+1)},\sigma^{(N+2)},\ldots).
\end{equation}

Let us recall the following useful integral identity:
\begin{equation}
\label{Integral_Identity}
\int_{0}^{t_{j}} \int_{0}^{t_{j+1}} \cdots \int_{0}^{t_{j+k-1}} \,dt_{j+k}\,\cdots \,dt_{j+2} \,dt_{j+1} = \frac{t_j^k}{k!}\,.
\end{equation}
We will use \eqref{Integral_Identity} in order to estimate Duhamel expansion terms in Proposition \ref{DuhamelEstimate1} and in Proposition \ref{DuhamelEstimate2} below. This type of observation was first used in \cite{CP2} in the one-dimensional deterministic setting. It was also subsequently used in the randomized setting in \cite{SoSt}.

\section{The independently randomized Gross-Pitaevskii hierarchy}
\label{The independently randomized Gross-Pitaevskii hierarchy}
In this section, we will study the \emph{independently randomized Gross-Pitaevskii hierarchy}:

\begin{equation}
\label{RandomizedGP1}
\begin{cases}
i \partial_t \gamma^{(k)} + (\Delta_{\vec{x}_k}-\Delta_{\vec{x}'_k}) \gamma^{(k)}=\sum_{j=1}^{k}[B_{j,k+1}]^{\omega_{k+1}}(\gamma^{(k+1)})\\
\gamma^{(k)}\big|_{t=0}=\gamma^{(k)}_0.
\end{cases}
\end{equation}
As before, the $\omega_{k+1}$ are mutually independent random parameters.
We take the spatial domain to be $\Lambda=\mathbb{T}^3$, unless it is noted otherwise.

Let us first recall the following randomized estimate, which is the content of Theorem 3.1 in \cite{SoSt}:
\begin{theorem}
\label{RandomizedEstimate1}
For all $\alpha>\frac{3}{4}$, there exists $C_0>0$, depending only on $\alpha$ such that, for all sequences of density matrices $(\gamma_0^{(k)})_k$, for all $k \in \mathbb{N}$ and $j \in \{1,2,\ldots,k\}$, the following estimate holds:
\begin{equation}
\notag
\big\|[B_{j,k+1}]^{\omega}\gamma_0^{(k+1)}\big\|_{L^2(\Omega) H^{\alpha}(\Lambda^k \times \Lambda^k)} \leq C_0 \big\|\gamma_0^{(k+1)}\|_{H^{\alpha}(\Lambda^{k+1} \times \Lambda^{k+1})}.
\end{equation}
In particular:
\begin{equation}
\notag
\big\|[B^{(k+1)}]^{\omega}\gamma_0^{(k+1)}\big\|_{L^2(\Omega) H^{\alpha}(\Lambda^k \times \Lambda^k)} \leq C_0k \cdot \big\|\gamma_0^{(k+1)}\|_{H^{\alpha}(\Lambda^{k+1} \times \Lambda^{k+1})}.
\end{equation}

\end{theorem}

\begin{remark}
\label{Remark_higher_dimensions1}
As was noted in the discussions in Subsection \ref{Setup of the problem} and Subsection \ref{Statement of the problem. Main Results}, the analysis that we will present in this section can be generalized to the case $\Lambda=\mathbb{T}^d$, but then the condition $\alpha>\frac{3}{4}$ has to be replaced with $\alpha>\frac{d}{4}.$
\end{remark}

Given a sequence $\Gamma(0)=(\gamma_0^{(k)})_k$ of density matrices \footnote{Throughout our paper, we will denote the sequence of density matrices $(\gamma_0^{(k)})_k$ by $\Gamma(0)$ instead of by $\Gamma_0$, since the $\Gamma_N$ are typically defined to be solutions to truncated hierarhcies.} and $\omega^{*} \in \Omega^{*}$, we define a new sequence of time-dependent density matrices $Duh_{\,j}^{\,\omega^{*}} (\Gamma(0))(t)$ as follows:

If $j=0$:
\begin{equation}
\label{Duhamel_0}
Duh_{\,0}^{\,\omega^{*}}(\Gamma(0))^{(k)}(t):=\mathcal{U}^{(k)}(t)\,\gamma_0^{(k)}.    
\end{equation}

If $j \geq 1$:
\begin{equation}
\label{Duhamel_j}
Duh_{\,j}^{\,\omega^{*}}(\Gamma(0))^{(k)}(t):=
\end{equation}
$$(-i)^j \int_0^t \int_0^{t_1} \cdots \int_0^{t_{j-1}} \,\mathcal{U}^{(k)}(t-t_1) \,[B^{(k+1)}]^{\omega_{k+1}} \, \mathcal{U}^{(k+1)}(t_1-t_2)\,[B^{(k+2)}]^{\omega_{k+2}}\, \cdots
$$ 
\begin{equation}
\notag
\cdots \,\mathcal{U}^{(k+j-1)}(t_{j-1}-t_j)\,[B^{(k+j)}]^{\omega_{k+j}}\,\mathcal{U}^{(k+j)}(t_j)\,\gamma_0^{(k+j)} \, dt_j\,dt_{j-1}\,\cdots\,dt_2\,dt_1.
\end{equation}
The quantity $Duh_{\,j}^{\,\omega^{*}}(\Gamma(0))$ corresponds to an iterated Duhamel expansion in the hierarchy \eqref{RandomizedGP1}.
We note that, in our convention, the subscript $j$ denotes the length of the Duhamel expansion. This is used in order to be consistent with our notation in \cite{SoSt}. Let us note that a slightly different convention is used in \cite{CP4}, where the subscript denotes the length of the Duhamel expansion plus one.

Let us now prove an estimate on the Duhamel terms $Duh_{\,j}^{\,\omega^{*}}(\Gamma(0))^{(k)}(t)$. We recall from Section \ref{Notation} that, for given $T>0$, we denote the interval $[0,T]$ by $I$. With this convention in mind, we prove: 

\begin{proposition}
\label{DuhamelEstimate1}
Let us fix $\alpha>\frac{3}{4}$.
Suppose that $T>0$. There exists $C_1>0$, depending only on $\alpha$ and a universal constant $C_2>0$ such that for all $j,k \in \mathbb{N}$:
\begin{equation}
\notag
\big\|Duh_{\,j}^{\,\omega^{*}} (\Gamma(0))^{(k)}(t) \big\|_{L^{\infty}_{t \in I} L^2(\Omega^{*}) H^{\alpha}(\Lambda^k \times \Lambda^k)} \leq (C_1T)^j \, \cdot \, C_2^k \, \cdot \, \big\|\gamma_0^{(k+j)}\big\|_{H^{\alpha}(\Lambda^{k+j} \times \Lambda^{k+j})}.
\end{equation} 
\end{proposition}

\begin{proof}
Let us fix $t \in I$ and $j,k \in \mathbb{N}$. We use Minkowski's inequality in order to estimate:

$$\big\|Duh_{\,j}^{\,\omega^{*}}(\Gamma(0))^{(k)}(t)\big\|_{L^2(\Omega^{*}) H^{\alpha}( \Lambda^k \times \Lambda^k)}=\big\|S^{(k,\alpha)}\,Duh_{\,j}^{\,\omega^{*}}(\Gamma(0))^{(k)}(t)\big\|_{L^2(\Omega^{*} \times \Lambda^k \times \Lambda^k)}$$ 
$$\leq \int_0^t \int_0^{t_1} \cdots \int_0^{t_{j-1}} \|S^{(k,\alpha)} \,\mathcal{U}^{(k)}(t-t_1) \,[B^{(k+1)}]^{\omega_{k+1}} \, \mathcal{U}^{(k+1)}(t_1-t_2) \,[B^{(k+2)}]^{\omega_{k+2}}\, \cdots$$
$$\cdots \, \mathcal{U}^{(k+j-1)}(t_{j-1}-t_j)\,[B^{(k+j)}]^{\omega_{k+j}}\,\mathcal{U}^{(k+j)}(t_j)\, \gamma_0^{(k+j)}\big\|_{L^2(\Omega^{*} \times \Lambda^k \times \Lambda^k)} dt_j \, dt_{j-1} \cdots dt_2 \,dt_1.$$
%Let us recall that $[B_{k+1}]^{\omega_{k+1}}=\sum_{j=1}^{k}[B_{j,k+1}]^{\omega_{k+1}}$. 
We can now iteratively use Theorem \ref{RandomizedEstimate1} in $\omega_{k+1}, \omega_{k+2}, \ldots, \omega_{k+j}.$ This is justified by using Fubini's Theorem as in the proof of Theorem 5.2 in  \cite{SoSt}.  It follows that the above expression is:
$$\leq \int_0^t \int_0^{t_1} \cdots \int_0^{t_{j-1}}  \, C_0^j \, \cdot \, k \, \cdot \, (k+1) \,  \cdots  \, (k+j-1) \, \cdot \big\|S^{(k+j,\alpha)}\gamma_0^{(k+j)}\big\|_{L^2(\Lambda^{k+j} \times \Lambda^{k+j})} dt_j \, dt_{j-1}\,\cdots\,dt_2\,dt_1$$
$$=\frac{t^j}{j!}\, \cdot \, C_0^j \, \cdot \, k \, \cdot \, (k+1) \,  \cdots  \, (k+j-1) \, \cdot \big\|S^{(k+j,\alpha)}\gamma_0^{(k+j)}\big\|_{L^2(\Lambda^{k+j} \times \Lambda^{k+j})}.$$
In the last equality, we used \eqref{Integral_Identity}.

We know by Stirling's formula that:
\begin{equation}
\label{Stirling's formula}
m! \sim \Big(\frac{m}{e}\Big)^m \cdot \sqrt{2\pi m}.
\end{equation}
It follows that there exists a universal constant $C>0$ such that:
$$
\big\|S^{(k,\alpha)}\,Duh_{\,j}^{\,\omega^{*}}(\Gamma(0))^{(k)}(t)\big\|_{L^2(\Omega^{*} \times \Lambda^k \times \Lambda^k)}
$$
\begin{equation}
\label{upperbound1}
\leq \frac{C^j T^j}{j^j} \, \cdot \, C_0^j \, \cdot \, (k+j)^j\, \cdot \, \big\|S^{(k+j,\alpha)}\gamma_0^{(k+j)}\big\|_{L^2(\Lambda^{k+j} \times \Lambda^{k+j})}.
\end{equation}
We note that:
$$\frac{1}{j^j} \, \cdot \, (k+j)^j=\Big(1+\frac{k}{j}\Big)^{\,j} \leq C_2^k$$
for some universal constant $C_2>0$.

Hence, the quantity in \eqref{upperbound1} is 
$$\leq (C_1 T)^j \, \cdot \, C_2^k \, \cdot \, \big\|S^{(k+j,\alpha)} \gamma_0^{(k+j)} \big\|_{L^2(\Lambda^{k+j} \times \Lambda^{k+j})}$$
$$= (C_1 T)^j \, \cdot \, C_2^k \, \cdot \, \big\|\gamma_0^{(k+j)} \big\|_{H^{\alpha}(\Lambda^{k+j} \times \Lambda^{k+j})}.$$ 
Here, $C_1=C_0 \cdot C.$ The claim now follows after taking the supremum in $t \in I$.
\end{proof}

Another useful estimate which follows from Theorem \ref{RandomizedEstimate1} is the following:

\begin{proposition} 
\label{RandomizedEstimate2} 
Let $\alpha>\frac{3}{4}$ be given. Let us fix $k \in \mathbb{N}$, and $j \in \{1,\ldots,k\}$. Suppose that $\tilde{\omega}_{k+1} \in \prod_{\ell \neq k+1} \Omega_{\ell}$ as in \eqref{tildeomegak+1}. Suppose furthermore that $\gamma_{0,\tilde{\omega}_{k+1}}^{(k+1)}$ is a density matrix of order $k+1$, which depends only on the random parameter $\tilde{\omega}_{k+1}$ (and hence is independent of $\omega_{k+1}$). Then, for the same constant $C_0$ in Theorem \ref{RandomizedEstimate1}, the following estimate holds:
\begin{equation}
\label{RandomizedEstimate2I}
\big\|[B_{j,k+1}]^{\omega_{k+1}}\gamma_{0,\tilde{\omega}_{k+1}}^{(k+1)}\big\|_{L^2(\Omega^{*}) H^{\alpha} (\Lambda^k \times \Lambda^k)} \leq C_0 \big\|\gamma_{0,\tilde{\omega}_{k+1}}^{(k+1)}\|_{L^2(\Omega^{*}) H^{\alpha} (\Lambda^{k+1} \times \Lambda^{k+1})}.
\end{equation}
In particular:
\begin{equation}
\label{RandomizedEstimate2II}
\big\|[B^{(k+1)}]^{\omega_{k+1}}\gamma_{0,\tilde{\omega}_{k+1}}^{(k+1)}\big\|_{L^2(\Omega^{*}) H^{\alpha} (\Lambda^k \times \Lambda^k)} \leq C_0k \cdot \big\|\gamma_{0,\tilde{\omega}_{k+1}}^{(k+1)}\|_{L^2(\Omega^{*}) H^{\alpha} (\Lambda^{k+1} \times \Lambda^{k+1})}.
\end{equation}
Moreover, let us assume that for all $k \in \mathbb{N}$, the set of random parameters $\tilde{\omega}_{k+1}$ and density matrices $\gamma^{(k+1)}_{0,\tilde{\omega}_{k+1}}$ are as above (i.e. they are independent from $\omega_{k+1}$).  We define the sequence: 
$$\big[\tilde{\Gamma}_0\,\big]^{\omega^{*}}:=\big(\big(\big[\tilde{\Gamma}_0\,\big]^{\omega^{*}}\big)^{(1)},\gamma^{(2)}_{0,\tilde{\omega}_2}, \gamma^{(3)}_{0,\tilde{\omega}_3},\ldots \big)$$
where $\big(\big[\tilde{\Gamma}_0\,\big]^{\omega^{*}}\big)^{(1)}$ is an arbitrary density matrix of order $1$ which depends on $\omega^{*}$.
The following estimate then holds for all $\xi, \xi_0>0$ such that $\xi_0<\xi$:

\begin{equation}
\label{RandomizedEstimate2III}
\big\|\big[\widehat{B}\,\big]^{\omega^{*}}\big[\tilde{\Gamma}_0\,\big]^{\omega^{*}}\big\|_{L^2(\Omega^{*}) \mathcal{H}^{\alpha}_{\xi_0}} \lesssim_{\,\xi,\,\xi_0,\,\alpha} \big\|\big[\tilde{\Gamma}_0\big]^{\omega^{*}}\big\|_{L^2(\Omega^{*})\mathcal{H}^{\alpha}_{\xi}}.
\end{equation}
%The implied constant depends on $\xi,\xi_1$, and $\alpha$.
\end{proposition}

\begin{proof}
We observe that \eqref{RandomizedEstimate2II} follows from \eqref{RandomizedEstimate2I} and from the fact that $[B^{(k+1)}]^{\omega_{k+1}}=\sum_{j=1}^{k}[B_{j,k+1}]^{\omega_{k+1}}$ by \eqref{Bomegak+1}. Let us now prove \eqref{RandomizedEstimate2I}. We note that:
$$\big\|S^{(k,\alpha)}[B_{j,k+1}]^{\omega_{k+1}}\gamma_{0,\tilde{\omega}_{k+1}}^{(k+1)}\big\|_{L^2(\Omega^{*} \times \Lambda^k \times \Lambda^k)}=
\big\|S^{(k,\alpha)}[B_{j,k+1}]^{\omega_{k+1}}\gamma_{0,\tilde{\omega}_{k+1}}^{(k+1)}\big\|_{L^2(\Omega_{k+1} \times \prod_{\ell \neq k+1} \Omega_{\ell} \times \Lambda^k \times \Lambda^k)}.$$
By the proof of Theorem \ref{RandomizedEstimate1}, which is given in \cite{SoSt}, it follows that we can treat all the $\omega_{\ell}$ for $\ell \neq k+1$ as fixed parameters. Hence, this expression is:
$$\leq C_0 \cdot \big\|S^{(k+1,\alpha)}\gamma_{0,\tilde{\omega}_{k+1}}^{(k+1)}\big\|_{L^2(\prod_{\ell \neq k+1}\Omega_{\ell} \times \Lambda^{k+1} \times \Lambda^{k+1})}.$$
By the property of probability measures, this quantity equals:
$$C_0 \cdot \big\|S^{(k+1,\alpha)}\gamma_{0,\tilde{\omega}_{k+1}}^{(k+1)}\big\|_{L^2(\Omega^{*} \times \Lambda^{k+1} \times \Lambda^{k+1})}=C_0 \cdot \big\|\gamma_{0,\tilde{\omega}_{k+1}}^{(k+1)}\big\|_{L^2(\Omega^{*}) H^{\alpha}(\Lambda^{k+1} \times \Lambda^{k+1})}.$$
The claims \eqref{RandomizedEstimate2I} and \eqref{RandomizedEstimate2II} now follow.

In order to prove \eqref{RandomizedEstimate2III}, we multiply both sides of \eqref{RandomizedEstimate2II} by $\xi_0^k$ and use the fact that $k \cdot \xi_0^{k} \lesssim_{\,\xi,\,\xi_0} \xi^{k+1}$, uniformly in $k$. The claim follows when we sum in $k$.
\end{proof}

\begin{remark}
 We note that the quantity $\big(\big[\tilde{\Gamma}_0\big]^{\omega^{*}}\big)^{(1)}$ does not appear on the left-hand side of \eqref{RandomizedEstimate2III}.
\end{remark}

\begin{remark}
In general, we note that if $\tilde{\omega}_{k+1}$ does not depend on $\omega_j$ for some $j \in \{2,3,\ldots,k+1\}$, then we can leave out the $L^2(\Omega_j)$ component of the norm on both sides of \eqref{RandomizedEstimate2I} and \eqref{RandomizedEstimate2II}. Namely, due to the fact that neither side depends on $\omega_j$ and since $\Omega_j$ has measure $1$ by the property of probability measures, the norms indeed do not change if we omit this factor. For future reference, we will keep the whole $\Omega^{*}$ product in order to simplify the notation and in order to work in the same space.
\end{remark}

\subsection{The truncated Gross-Pitaevskii hierarchy corresponding to \eqref{RandomizedGP1}}
\label{The truncated Gross-Pitaevskii hierarchy corresponding to RandomizedGP1}

Let us now consider the \emph{truncated Gross-Pitaevskii hierarchy corresponding to \eqref{RandomizedGP1}}. In particular, we fix $N \in \mathbb{N}$ and $\omega^{*} \in \Omega^{*}$ and we look for a  sequence $\big(\big[\gamma_N^{(k)}\big]^{\omega^{*}}\big)_k$, where each 
$\big[\gamma_N^{(k)}\big]^{\omega^{*}}$ is a time-dependent density matrix of order $k$, which is parametrized by $\omega^{*}=(\omega_2,\omega_3,\omega_4,\ldots) \in \Omega^{*}$.

%\begin{equation}
%\notag
%\big[\gamma_N^{(k)}\big]^{\omega^{*}}:
%\mathbb{R}_{t} \times \Omega^{*} \times \Lambda^{k} \times \Lambda^{k} \rightarrow \mathbb{C}
%\end{equation}
Moreover, we assume that for $k \in \mathbb{N}$ such that $k \leq N$:

\begin{equation}
\label{TruncatedRandomizedGP1}
\begin{cases}
i \partial_t \, \big[\gamma_N^{(k)}\big]^{\omega^{*}}+(\Delta_{\vec{x}_k}-\Delta_{\vec{x}'_k}) \, \big[\gamma_N^{(k)}\big]^{\omega^{*}}=\sum_{j=1}^{k} [B_{j,k+1}]^{\omega_{k+1}}\big(\big[\gamma_N^{(k+1)}\big]^{\omega^{*}}\big)\\
\big[\gamma_N^{(k)}\big]^{\omega^{*}} \big|_{t=0} = \gamma_0^{(k)}.
\end{cases}
\end{equation}
For $k>N$, we will impose the condition that:
\begin{equation}
\label{TruncatedRandomizedGP1assumption}
\big[\gamma_N^{(k)}\big]^{\omega^{*}} \equiv 0.
\end{equation}
As in \cite{CP4}, our goal now is to find $\big(\big[\gamma_N^{(k)}\big]^{\omega^{*}}\big)_k$, which satisfies \eqref{TruncatedRandomizedGP1} and \eqref{TruncatedRandomizedGP1assumption} without any statement about the uniqueness of these solutions. 

By construction, for $k=N$:
\begin{equation}
\notag
\begin{cases}
i \partial_t \, \big[\gamma_N^{(N)}\big]^{\omega^{*}} + (\Delta_{\vec{x}_N}-\Delta_{\vec{x}'_N}) \,\big[\gamma_N^{(N)}\big]^{\omega^{*}}=0\\
\big[\gamma_N^{(N)}\big]^{\omega^{*}}\big|_{t=0}=\gamma_0^{(N)}.
\end{cases}
\end{equation}
Hence, we can take:
$$\big[\gamma_N^{(N)}\big]^{\omega^{*}}(t)=\,\mathcal{U}^{(N)}(t) \, \gamma_0^{(N)}.$$
Let us note that the above expression does not depend on $\omega^{*}$ in the end. 

For $k<N$, we can find a solution $\big[\gamma_N^{(k)}\big]^{\omega^{*}}$ to \eqref{TruncatedRandomizedGP1} by Duhamel iteration. The formula that we obtain is:

\begin{equation}
\label{gammaNkomega}
\big[\gamma_N^{(k)}\big]^{\omega^{*}}(t)=\sum_{j=0}^{N-k} Duh_{\,j}^{\,\omega^{*}} (\Gamma(0))^{(k)}(t).
\end{equation}
The fact that $\big[\gamma_N^{(k)}\big]^{\omega^{*}}$, given by \eqref{gammaNkomega}, solves \eqref{TruncatedRandomizedGP1} follows by the definition of $Duh_{\,j}^{\,\omega^{*}}$ and by regressive induction on $k$.
Let us define $\big[\Gamma_N\,\big]^{\omega^{*}}=\Big(([\Gamma_N]^{\omega^{*}})^{(k)}\Big)_k$ by:
\begin{equation}
\label{GammaN}
\big(\big[\Gamma_N\,\big]^{\omega^{*}}\big)^{(k)}:=\big[\gamma_N^{(k)}\big]^{\omega^{*}}.
\end{equation}
We will henceforth use this notation.
Let us also write the \emph{truncated initial data} as follows:
\begin{equation}
\label{GammaN0}
\Gamma_N(0):=(\gamma_0^{(1)},\gamma_0^{(2)},\ldots,\gamma_0^{(N)},0,0, \ldots)=\mathbf{P}_{\leq N}\Gamma(0),
\end{equation}
for the projection operator $\mathbf{P}_{\leq N}$ defined in \eqref{P<=N}.

By construction, $\big[\Gamma_N\big]^{\omega^{*}}$ solves:

\begin{equation}
\label{GammaNomegastar}
\begin{cases}
i \partial_t \, \big[\Gamma_N \big]^{\omega^{*}} + \widehat{\Delta}_{\pm} \, \big[\Gamma_N \big]^{\omega^{*}} = \big[\widehat{B} \,\big]^{\omega^{*}} \big[\Gamma_N \big]^{\omega^{*}}\\
\big[\Gamma_N \big]^{\omega^{*}} \big|_{t=0}=\Gamma_N(0).
\end{cases}
\end{equation}
Here, we recall the definition of $\widehat{\Delta}_{\pm}$ from \eqref{Delta_plusminus} and the definition of $\big[\widehat{B}\,\big]^{\omega^{*}}$ from \eqref{B_hat_omega_star}.

So far, we have constructed a solution $\big[\Gamma_N\big]^{\omega^{*}}$ to \eqref{GammaNomegastar}. We would like to take the limit as $N \rightarrow \infty$. 

Let us first fix some notation, which we will use for the rest of this section unless it is specified otherwise:

We will fix:
\begin{equation}
\label{alpha}
\alpha>\frac{3}{4}.
\end{equation} 
Let us consider a time $T>0$ and the associated time interval $I=[0,T]$. For this $T$, we first choose $\xi'>0$ sufficiently large such that: 
\begin{equation}
\label{xi_prime}
\frac{C_1T}{\xi'}<1.
\end{equation} 
Having chosen $\xi'$, we choose $\xi \in (0,\xi')$ sufficiently small such that: 
\begin{equation}
\label{xi}
\frac{C_2 \,\xi}{\xi'}<1.
\end{equation}
Here, $C_1$ and $C_2$ are the constants from Proposition \ref{DuhamelEstimate1}.
Conversely, given $\xi,\xi'>0$ with $\xi \in (0,\xi')$, which satisfy \eqref{xi}, we can find $T>0$ which satisfies \eqref{xi_prime}. Unless it is otherwise noted, we will assume that $T,\xi,\xi'$ satisfy the above assumptions.

\subsection{A local-in-time result for initial data $\Gamma(0)$ of regularity $\alpha$}
\label{A local-in-time result for initial data of regularity alpha}

In this subsection, we will assume that $\Gamma(0) \in \mathcal{H}^{\alpha}_{\xi'}$, for $\alpha$ as in \eqref{alpha}.

Keeping in mind \eqref{xi_prime} and \eqref{xi}, we now prove the following result:

\begin{proposition}
\label{CauchySequence1}
The sequence $\big(\big[\widehat{B}\,\big]^{\omega^{*}} \big[\Gamma_N\big]^{\omega^{*}}\big)_N$ is Cauchy in $L^{\infty}_{t \in I} L^2(\Omega^{*}) \mathcal{H}^{\alpha}_{\xi}$. %Here, $\xi$ is chosen to satisfy condition \eqref{xi}.
\end{proposition}

\begin{proof}
Let us fix $t \in I$ and let us also take $N_1, N_2 \in \mathbb{N}$ such that $N_1<N_2$. For $k \in \mathbb{N}$, we note that:

$$\Big(\big[\widehat{B}\,\big]^{\omega^{*}} \big(\big[\Gamma_{N_1}\big]^{\omega^{*}}-\big[\Gamma_{N_2}\big]^{\omega^{*}}\big)\Big)^{(k)}(t)=$$
$$=\big[B^{(k+1)}\big]^{\omega_{k+1}} \Big(\big(\big[\Gamma_{N_1}\big]^{\omega^{*}}\big)^{(k+1)}-\big(\big[\Gamma_{N_2}\big]^{\omega^{*}}\big)^{(k+1)}\Big)(t)$$
$$=\big[B^{(k+1)}\big]^{\omega_{k+1}} \big(\sum_{j=0}^{N_1-k-1} Duh_{\,j}^{\,\omega^{\,*}}(\Gamma_{N_1}(0))^{(k+1)}(t)-\sum_{j=0}^{N_2-k-1} Duh_{\,j}^{\,\omega^{\,*}}(\Gamma_{N_2}(0))^{(k+1)}(t)\big)$$
$$=\big[B^{(k+1)}\big]^{\omega_{k+1}} \big(\sum_{j=0}^{N_1-k-1} Duh_{\,j}^{\,\omega^{*}}(\Gamma_{N_1}(0)-\Gamma_{N_2}(0))^{(k+1)}(t)-\sum_{j=N_1-k}^{N_2-k-1}Duh_{\,j}^{\,\omega^{*}}(\Gamma_{N_2}(0))^{(k+1)}(t)\big)$$
By construction, $(\Gamma_{N_1}(0))^{(\ell)}=(\Gamma_{N_2}(0))^{(\ell)}$ for all $\ell \leq N_1$ and so 
$$Duh_{\,j}^{\,\omega^{*}} (\Gamma_{N_1}(0))^{(k+1)}(t)=Duh_{\,j}^{\,\omega^{*}} (\Gamma_{N_2}(0))^{(k+1)}(t)$$ 
for all $j=0,\ldots,N_1-k-1$. It follows that the first sum above equals zero. Hence:
\begin{equation}
\label{BhatGammaNdifference}
\Big(\big[\widehat{B}\,\big]^{\omega^{*}} \big(\big[\Gamma_{N_1}\big]^{\omega^{*}}-\big[\Gamma_{N_2}\big]^{\omega^{*}}\big)\Big)^{(k)}(t)=
\end{equation}
\begin{equation}
\notag
-\sum_{j=N_1-k}^{N_2-k-1}\big[B^{(k+1)}\big]^{\omega_{k+1}} Duh_{\,j}^{\,\omega^{*}}(\Gamma_{N_2}(0))^{(k+1)}(t).
\end{equation}
Consequently:
$$\Big\|\Big(\big[\widehat{B}\,\big]^{\omega^{*}} \big(\big[\Gamma_{N_1}\big]^{\omega^{*}}-\big[\Gamma_{N_2}\big]^{\omega^{*}}\big)\Big)^{(k)}(t)\Big\|_{L^2(\Omega^{*})H^{\alpha}(\Lambda^k \times \Lambda^k)}$$
$$\leq \sum_{j=N_1-k}^{N_2-k-1}  \Big\|\big[B^{(k+1)}\big]^{\omega_{k+1}} Duh_{\,j}^{\,\omega^{*}}(\Gamma_{N_2}(0))^{(k+1)}(t)\Big\|_{L^2(\Omega^{*})H^{\alpha}(\Lambda^k \times \Lambda^k)}.$$
We recall the definition of $Duh_{\,j}^{\,\omega^{*}}$ in \eqref{Duhamel_0} and \eqref{Duhamel_j} in order to deduce that, in the above sum, the terms $Duh_{\,j}^{\,\omega^{*}}(\Gamma_{N_2}(0))^{(k+1)}(t)$ do not depend on $\omega_{k+1}$. 
More precisely, if $j=0$, $Duh_{\,j}^{\,\omega^{*}}(\Gamma_{N_2}(0))^{(k+1)}(t)$ does not depend on any random parameters, and if $j \geq 1$, $Duh_{\,j}^{\,\omega^{*}}(\Gamma_{N_2}(0))^{(k+1)}(t)$ depends only on the random parameters $\omega_{k+2},\omega_{k+3},\ldots,\omega_{k+j+1}$. 
In either case, there is no dependence on $\omega_{k+1}$. In particular, we can use Proposition \ref{RandomizedEstimate2} %and the fact that $[B_{n+1}]^{\omega_{n+1}}=\sum_{j=1}^{n}[B_{j,n+1}]^{\omega_{n+1}}$
to deduce that the above sum is:
$$\leq C_0 k \cdot \sum_{j=N_1-k}^{N_2-k-1} \big\|Duh_{\,j}^{\,\omega^{*}}(\Gamma_{N_2}(0))^{(k+1)}(t)\big\|_{L^2(\Omega^{*})H^{\alpha}(\Lambda^{k+1} \times \Lambda^{k+1})}.$$
By Proposition \ref{DuhamelEstimate1}, this sum is:
$$\leq C_0 k \cdot \sum_{j=N_1-k}^{N_2-k-1} (C_1T)^j \cdot C_2^{k+1} \big\|\gamma_0^{(k+j+1)}\big\|_{H^{\alpha}(\Lambda^{k+j+1} \times \Lambda^{k+j+1})}.$$
As a result, we may deduce that:
$$\sum_{k=1}^{\infty} \xi^k \cdot \big\| \big(\big[\widehat{B}\,\big]^{\omega^{*}} \big(\big[\Gamma_{N_1}\big]^{\omega^{*}}-\big[\Gamma_{N_2}\big]^{\omega^{*}})\big)^{(k)}(t)\big\|_{L^2(\Omega^{*})H^{\alpha}(\Lambda^k \times \Lambda^k)}$$
$$\leq \frac{C_0C_2}{\xi'} \cdot \sum_{k=1}^{\infty} \sum_{j=N_1-k}^{N_2-k-1} \Big(\frac{C_1T}{\xi'}\Big)^j \cdot k \cdot \Big(\frac{C_2 \, \xi}{\xi'}\Big)^k \cdot (\xi')^{k+j+1} \cdot \big\|\gamma_0^{(k+j+1)}\big\|_{H^{\alpha}(\Lambda^{k+j+1} \times \Lambda^{k+j+1})}.
$$
Since in the above sum, $k+j+1>N_1$, it follows that:
$$\Big\|\big[\widehat{B}\,\big]^{\omega^{*}} \big(\big[\Gamma_{N_1}\big]^{\omega^{*}}-\big[\Gamma_{N_2}\big]^{\omega^{*}}\big)(t)\Big\|_{L^2(\Omega^{*}) \mathcal{H}^{\alpha}_{\xi}}$$
$$\leq \frac{C_0C_2}{\xi'} \cdot \sum_{j=N_1-k}^{N_2-k-1} \Big(\frac{C_1T}{\xi'}\Big)^j \cdot \sum_{k=1}^{\infty} k \cdot \Big(\frac{C_2 \, \xi}{\xi'}\Big)^k \cdot \big\|\mathbf{P}_{>N_1} \Gamma(0)\big\|_{\mathcal{H}^{\alpha}_{\xi'}}.$$
Here, we used the definition of the norm $L^2(\Omega^{*})\mathcal{H}^{\alpha}_{\xi'}$ in \eqref{Space2} and of the operator $\mathbf{P}_{>N_1}$ in \eqref{PgreaterN}.
Since, by the assumptions \eqref{xi_prime} and \eqref{xi} we know that $\frac{C_1T}{\xi'}, \frac{C_2\,\xi}{\xi'} \in (0,1)$, it follows that:
\begin{equation}
\label{CauchySequence1bound}
\big\|\big[\widehat{B}\,\big]^{\omega^{*}}\big(\big[\Gamma_{N_1}\big]^{\omega^{*}}-\big[\Gamma_{N_2}\big]^{\omega^{*}}\big)(t)\big\|_{L^{\infty}_{t \in I}L^2(\Omega^{*})\mathcal{H}^{\alpha}_{\xi}} \lesssim_{\,T,\,\xi,\,\xi',\,\alpha} \big\|\mathbf{P}_{>N_1} \Gamma(0)\big\|_{\mathcal{H}^{\alpha}_{\xi'}}.
\end{equation}
Since $\Gamma(0) \in \mathcal{H}^{\alpha}_{\xi'}$ by assumption, it follows that $\big\|\mathbf{P}_{>N_1} \Gamma(0)\big\|_{\mathcal{H}^{\alpha}_{\xi'}} \rightarrow 0$ as $N_1 \rightarrow \infty$. We can now deduce the claim. 
\end{proof}

\begin{remark}
\label{BN1}
Let us observe that same methods used to prove \eqref{CauchySequence1bound} allow us to deduce the following bound:
\begin{equation}
\label{BN1bound}
\big\|\big[\widehat{B}\,\big]^{\omega^{*}}\big[\Gamma_{N_1}\big]^{\omega^{*}}(t)\big\|_{L^{\infty}_{t \in I} L^2(\Omega^{*}) \mathcal{H}^{\alpha}_{\xi}} \lesssim_{\,T,\,\xi,\,\xi',\,\alpha} \big\|\Gamma(0)\big\|_{\mathcal{H}^{\alpha}_{\xi'}}
\end{equation}
uniformly in $N_1$.
\end{remark}

We can now use Proposition \ref{CauchySequence1} in order to prove the following result:

\begin{proposition}
\label{limit2}
There exists $\theta^{\,\omega^{*}} \in L^{\infty}_{t \in I} L^2(\Omega^{*}) \mathcal{H}^{\alpha}_{\xi}$ such that $\big[\widehat{B}\,\big]^{\omega^{*}} \big[\Gamma_N \big]^{\omega^{*}} \rightarrow \theta^{\,\omega^{*}}$ strongly in $L^{\infty}_{t \in I} L^2(\Omega^{*}) \mathcal{H}^{\alpha}_{\xi}$ as $N \rightarrow \infty$. In addition, $\theta^{\,\omega^{*}}$ has the property that, for all $k \geq 2$, $\big(\theta^{\,\omega^{*}}\big)^{(k)}$ does not depend on $\omega_2, \omega_3, \ldots, \omega_k$.
\end{proposition}

\begin{proof}
We observe that $L^{\infty}_{t \in I} L^2(\Omega^{*}) \mathcal{H}^{\alpha}_{\xi}$ is a Banach space. Hence, by Proposition \ref{CauchySequence1}, it follows that there exists $\theta^{\,\omega^{*}} \in L^{\infty}_{t \in I} L^2(\Omega^{*}) \mathcal{H}^{\alpha}_{\xi}$ such that:

\begin{equation}
\label{limit1}
\big[\widehat{B}\,\big]^{\omega^{*}} \big[\Gamma_N\big]^{\omega^{*}} \rightarrow \theta^{\,\omega^{*}}
\end{equation}
strongly in $L^{\infty}_{t \in I} L^2(\Omega^{*}) \mathcal{H}^{\alpha}_{\xi}$ as $N \rightarrow \infty.$ 

We now fix $k \in \mathbb{N}$, and we now study the dependence on the random parameters $\omega_j$ of the density matrix $(\theta^{\,\omega^{*}})^{(k)}$. In particular, we note that, by construction:

\begin{equation}
\label{B_GammaN_omegastar}
\Big(\big[\widehat{B}\,\big]^{\omega^{*}} \big[\Gamma_N\big]^{\omega^{*}}\Big)^{(k)}(t)=\big[B^{(k+1)}\big]^{\omega_{k+1}} \Big(\big[\Gamma_N\big]^{\omega^{*}}\Big)^{(k+1)}(t)
\end{equation}
$$=\sum_{j=0}^{N-k-1} \big[B^{(k+1)}\big]^{\omega_{k+1}} Duh_{\,j}^{\,\omega^{*}} (\Gamma(0))^{(k+1)}(t)$$
$$=\big[B_{k+1}\big]^{\omega_{k+1}}\,\mathcal{U}^{(k+1)}(t)\, \gamma_0^{(k+1)} -i \int_{0}^{t} \big[B^{(k+1)}\big]^{\omega_{k+1}}
\,\mathcal{U}^{(k+1)}\,(t-t_1) \big[B^{(k+2)}\big]^{\omega_{k+2}} \,\mathcal{U}^{(k+2)}(t_1)\, \gamma_0^{(k+2)}\, dt_1+\cdots
$$
$$\cdots+(-i)^{N-k-1} \int_{0}^{t} \int_{0}^{t_1} \cdots \int_{0}^{t_{N-k-2}}
\big[B^{(k+1)}\big]^{\omega_{k+1}}\,\mathcal{U}^{(k+1)}(t-t_1)\,\big[B^{(k+2)}\big]^{\omega_{k+2}} \, \mathcal{U}^{(k+2)}(t_1-t_2)
\,\big[B^{(k+3)}\big]^{\omega_{k+3}}\,\cdots$$ 
$$\cdots \big[B^{(N)}\big]^{\omega_{N}} \, \mathcal{U}^{(N)}(t_{N-k-1}) \,\gamma_0^{(N)}\,dt_{N-k-1} \cdots dt_2 \, dt_1.
$$
We note that the above expression depends only on the random parameters $\omega_{k+1},\omega_{k+2},\ldots,\omega_{N}$.
In particular, if $k \geq 2$, then it does not depend on $\omega_2, \omega_3, \ldots, \omega_k$.

Let us observe that, from \eqref{limit1}, it follows that:

$$\Big(\big[\widehat{B}\,\big]^{\omega^{*}} \big[\Gamma_N\big]^{\omega^*}\Big)^{(k)} \rightarrow \Big(\theta^{\,\omega^{*}}\Big)^{(k)}$$
strongly in $L^{\infty}_{t \in I}L^2(\Omega^{*})H^{\alpha}(\Lambda^k \times \Lambda^k)$ as $N \rightarrow \infty$.
From this convergence, we may deduce that whenever $k \geq 2$, $\big(\theta^{\,\omega^{*}}\big)^{(k)}$ does not depend on $\omega_2,\omega_3, \ldots, \omega_k$. The claim now follows.
\end{proof}

\begin{remark}
\label{theta_omega_star_boundA}
If we use \eqref{BN1bound} from Remark \ref{BN1} and Proposition \ref{limit2}, and if we let $N_1 \rightarrow \infty$, it follows that:
\begin{equation}
\label{theta_omega_star_bound1}
\big\|\theta^{\,\omega^{*}}\big\|_{L^{\infty}_{t \in I} L^2(\Omega^{*}) \mathcal{H}^{\alpha}_{\xi}} \lesssim_{\,T,\,\xi,\,\xi',\,\alpha} \big\|\Gamma(0)\big\|_{\mathcal{H}^{\alpha}_{\xi'}}.
\end{equation}
\end{remark}
%We can summarize the above discussion as follows:

Since $\theta^{\,\omega^{*}}$ constructed above is only known to be an element of $L^{\infty}_{t \in I} L^2(\Omega^{*}) \mathcal{H}^{\alpha}_{\xi}$, it is defined for \emph{almost every value of $t \in I$}. This is in contrast to \eqref{gammaNkomega}, where the functions were defined for all $t \in I$ by construction. In the discussion that follows, we will hence work with identities in the space $L^{\infty}_{t \in I} L^2(\Omega^{*}) \mathcal{H}^{\alpha}_{\xi}$ and in similar spaces, where the functions are defined for almost all values of $t \in I$. For a discussion about which additional assumptions we need to add in order to work with pointwise values for all $t \in I$, we refer the reader to Subsection \ref{additional_regularity}.

Let us now apply the strategy from \cite{CP4} and prove that the $\theta^{\,\omega^{*}}$ constructed in Proposition \ref{limit2} satisfies the following property:

\begin{proposition}
\label{theta_omega_star_equation}
For $\theta^{\,\omega^{*}}$ from Proposition \ref{limit2} and for all $\xi_0>0$, the following identity holds:
$$\big\|\theta^{\,\omega^{*}}-\big[\widehat{B}\,\big]^{\omega^{*}}\,\mathcal{U}(t)\,\Gamma(0)+i\int_{0}^{t} \big[\widehat{B}\,\big]^{\omega^{*}}\,\mathcal{U}(t-s)\,\theta^{\,\omega^{*}}(s)\,ds \,\big\|_{L^{\infty}_{t \in I} L^2(\Omega^{*}) \mathcal{H}^{\alpha}_{\xi_0}}=0.$$
\end{proposition}

\begin{proof}
It suffices to show the claim for some $\xi_0 \in (0,\xi)$, where $\xi$ is as in \eqref{xi}. The general claim then follows. We note that, for all $N \in \mathbb{N}$:
\begin{equation}
\label{theta_omega_star_difference}
\big\|\theta^{\,\omega^{*}}-\big[\widehat{B}\,\big]^{\omega^{*}}\,\mathcal{U}(t)\,\Gamma(0)+i\int_{0}^{t} \big[\widehat{B}\,\big]^{\omega^{*}}\,\mathcal{U}(t-s)\,\theta^{\,\omega^{*}}(s)\,ds \,\big\|_{L^{\infty}_{t \in I} L^2(\Omega^{*}) \mathcal{H}^{\alpha}_{\xi_0}}
\end{equation}
$$\leq \big\|\theta^{\,\omega^{*}}-\big[\widehat{B}\,\big]^{\omega^{*}} \big[\Gamma_N\big]^{\omega^{*}}\big\|_{L^{\infty}_{t \in I} L^2(\Omega^{*}) \mathcal{H}^{\alpha}_{\xi_0}} + \big\|\big[\widehat{B}\,\big]^{\omega^{*}}\,\mathcal{U}(t)\,\big(\Gamma(0)-\Gamma_N(0)\big)\big\|_{L^{\infty}_{t \in I} L^2(\Omega^{*}) \mathcal{H}^{\alpha}_{\xi_0}} + $$
$$+ \big\|\int_{0}^{t} \big[\widehat{B}\,\big]^{\omega^{*}} \, \mathcal{U}(t-s) \big(\theta^{\,\omega^{*}}(s) -\big[\widehat{B}\,\big]^{\omega^{*}} \big[\Gamma_N\big]^{\omega^{*}}(s) \big) \, ds\big\|_{L^{\infty}_{t \in I} L^2(\Omega^{*}) \mathcal{H}^{\alpha}_{\xi_0}} +$$
$$+ \big\|\big[\widehat{B}\,\big]^{\omega^{*}} \big[\Gamma_N\big]^{\omega^{*}}(t)-\big[\widehat{B}\,\big]^{\omega^{*}}\,\mathcal{U}(t) \, \Gamma_N(0) \,+\, i \int_{0}^{t} \big[\widehat{B}\,\big]^{\omega^{*}} \, \mathcal{U}(t-s)\, \big[\widehat{B}\,\big]^{\omega^{*}} \big[\Gamma_N\big]^{\omega^{*}}(s)\, ds\big\|_{L^{\infty}_{t \in I} L^2(\Omega^{*}) \mathcal{H}^{\alpha}_{\xi_0}}.$$

Let us now show that each term converges to zero as $N \rightarrow \infty.$

\textbf{1)}
The first term on the right-hand side of \eqref{theta_omega_star_difference} converges to zero as $N \rightarrow \infty$ by using Proposition \ref{limit2} and the fact that $\xi_0<\xi$.

\textbf{2)}
For the second term, let first note that Theorem \ref{RandomizedEstimate1} and the unitarity of $\mathcal{U}^{(k+1)}(t)$ imply that for all $k \in \mathbb{N}$:
$$\big\|\big[B^{(k+1)}\big]^{\omega_{k+1}}\,\mathcal{U}^{(k+1)}(t)\,\big(\Gamma^{(k+1)}(0)-\Gamma_N^{(k+1)}(t)\big)\big\|_{L^{\infty}_{t \in I}L^2(\Omega^{*}) H^{\alpha}(\Lambda^k \times \Lambda^k)}$$ 
$$\leq C_0k \cdot \big\|\Gamma^{(k+1)}(0)-\Gamma_N^{(k+1)}(0)\big\|_{H^{\alpha}(\Lambda^{k+1} \times \Lambda^{k+1})}.$$
We recall that $\Gamma(0) \in \mathcal{H}^{\alpha}_{\xi'}$, for $\xi'$ as defined in \eqref{xi_prime}.
Since $\xi_0<\xi<\xi'$, it follows that the second term on the right-hand side of \eqref{theta_omega_star_difference} 
is $$\lesssim_{\,\xi_0,\,\xi',\,\alpha} \big\|\Gamma(0)-\Gamma_N(0)\big\|_{\mathcal{H}^{\alpha}_{\xi'}}=\big\|\mathbf{P}_{>N}\Gamma(0)\big\|_{\mathcal{H}^{\alpha}_{\xi'}}.$$
Since $\Gamma(0) \in \mathcal{H}^{\alpha}_{\xi'}$ by assumption, it follows that $\big\|\mathbf{P}_{>N}\Gamma(0)\big\|_{\mathcal{H}^{\alpha}_{\xi'}} \rightarrow 0$ as $N \rightarrow \infty.$

\textbf{3)} 
By Minkowski's inequality, the third term on the right-hand side of \eqref{theta_omega_star_difference} is:
\begin{equation}
\notag
\leq \int_{0}^{T} \big\|\big[\widehat{B}\,\big]^{\omega^{*}} \, \mathcal{U}(t-s) \big(\theta^{\,\omega^{*}}(s) -\big[\widehat{B}\,\big]^{\omega^{*}} \big[\Gamma_N\big]^{\omega^{*}}(s) \big)\big\|_{L^{\infty}_{t \in I} L^2(\Omega^{*}) \mathcal{H}^{\alpha}_{\xi_0}} \,ds.
\end{equation}
Let us now analyze the integrand.
Let us fix $s \in I$ and $k \in \mathbb{N}$ and consider:
$$\big\|\Big(\big[\widehat{B}\,\big]^{\omega^{*}} \, \mathcal{U}(t-s) \big(\theta^{\,\omega^{*}}(s) -\big[\widehat{B}\,\big]^{\omega^{*}} \big[\Gamma_N\big]^{\omega^{*}}(s) \big)\Big)^{(k)}\big\|_{L^{\infty}_{t \in I} L^2(\Omega^{*}) H^{\alpha}(\Lambda^k \times \Lambda^k)}$$
$$=\big\|\big[B^{(k+1)}\big]^{\omega_{k+1}} \, \mathcal{U}^{(k+1)}(t-s) \big(\big(\theta^{\,\omega^{*}}\big)^{(k+1)}(s) -\big[B_{k+2}\big]^{\omega_{k+2}} \big(\big[\Gamma_N\big]^{\omega^{*}}\big)^{(k+1)}(s) \big)\big\|_{L^{\infty}_{t \in I} L^2(\Omega^{*}) H^{\alpha}(\Lambda^k \times \Lambda^k)}.
$$
We note that, by Proposition \ref{limit2}, there is no $\omega_{k+1}$ dependence in $\big(\theta^{\,\omega^{*}}\big)^{(k+1)}(s)$. From formula \eqref{B_GammaN_omegastar} the proof of Proposition \ref{limit2}, we may also deduce that there is no $\omega_{k+1}$ dependence in $\big[B_{k+2}\big]^{\omega_{k+2}} \big(\big[\Gamma_N\big]^{\omega^{*}}\big)^{(k+1)}(s)$. In particular, we can apply the unitarity of $\mathcal{U}^{(k+1)}(t-s)$ and estimate \eqref{RandomizedEstimate2III} from Proposition \ref{RandomizedEstimate2} in order to deduce that the third term on the right-hand side of \eqref{theta_omega_star_difference} is:
%$$\leq C_0k \cdot \big\|\big(\theta^{\omega^{*}}\big)^{(k+1)}(s)-\big[B_{k+2}\big]^{\omega_{k+2}}\big(\big[\Gamma_N\big]^{\omega^{*}}\big)^{(k+1)}(s)\big\|_{L^2(\Omega^{*}) H^{\alpha}(\Lambda^{k+1} \times \Lambda^{k+1})}.$$
%Since $\xi>\xi_1$, it follows that the third term is:
$$\lesssim_{\,\xi_0,\,\xi,\,\alpha} T \cdot  \big\|\theta^{\,\omega^{*}}(s)-\big[\widehat{B}\,\big]^{\omega^{*}}\big[\Gamma_N\big]^{\omega^{*}}(s)\big\|_{L^{\infty}_{s \in I}L^2(\Omega^{*})\mathcal{H}^{\alpha}_{\xi}}.$$
By Proposition \ref{limit2}, this quantity converges to zero as $N \rightarrow \infty$. Consequently, the third term on the right-hand side of \eqref{theta_omega_star_difference} converges to zero as $N \rightarrow \infty$.

\textbf{4)}
Finally, we note that the fourth term on the right-hand side of \eqref{theta_omega_star_difference} equals zero by the construction of $\big[\Gamma_N\big]^{\omega^{*}}$.

The proposition now follows.
\end{proof}
In addition to the convergence of the sequence $\big(\big[\widehat{B}\,\big]^{\omega^{*}\,}\big[\Gamma_N\big]^{\omega^{*}}\big)_N$ obtained in Proposition 
\ref{limit2}, we can also obtain the convergence of the sequence $\big(\big[\Gamma_N\big]^{\omega^{*}}\big)_N$. As we will see, the analysis will be similar, but it will be simpler due to the fact that we already know the convergence of the sequence $\big(\big[\widehat{B}\,\big]^{\omega^{*}\,}\big[\Gamma_N\big]^{\omega^{*}}\big)_N$. A similar method was used in the deterministic setting in \cite{CP4}. We now prove the following result:

\begin{proposition}
\label{CauchySequence2}
The sequence $\big(\big[\Gamma_N\big]^{\omega^{*}}\big)_N$ is Cauchy in $L^{\infty}_{t \in I}L^2(\Omega^{*}) \mathcal{H}^{\alpha}_{\xi}$. Moreover, there exists $\big[\Gamma\,\big]^{\omega^{*}} \in L^{\infty}_{t \in I}  L^2(\Omega^{*}) \mathcal{H}^{\alpha}_{\xi}$ such that $\big[\Gamma_N\big]^{\omega^{*}} \rightarrow \big[\Gamma\,\big]^{\omega^{*}}$ strongly in $L^{\infty}_{t \in I} L^2(\Omega^{*}) \mathcal{H}^{\alpha}_{\xi}$ as $N \rightarrow \infty$. 

The obtained $\big[\Gamma\,\big]^{\omega^{*}}$ has the property that, for all $k \geq 2$, $\Big(\big[\Gamma\,\big]^{\omega^{*}}\Big)^{(k)}$ does not  depend on the random parameters $\omega_2, \omega_3, \ldots, \omega_k$. Furthermore, $\big[\Gamma\,\big]^{\omega^{*}}$ satisfies the bound:

\begin{equation}
\label{Gamma_omega_star_bound}
\big\|\big[\Gamma\,\big]^{\omega^{*}}\big\|_{L^{\infty}_{t \in I} L^2(\Omega^{*}) \mathcal{H}^{\alpha}_{\xi}} \lesssim_{\,\xi,\,\xi',\,\alpha} \big\|\Gamma(0)\big\|_{\mathcal{H}^{\alpha}_{\xi'}}.
\end{equation} 
\end{proposition}

\begin{proof}
Let us fix $N_1,N_2 \in \mathbb{N}$ with $N_1<N_2$. By construction of $\big[\Gamma_{N_j}\big]^{\omega^{*}}$, for $j=1,2$, it follows that:
$$\big\|\big[\Gamma_{N_1}\big]^{\omega^{*}}(t)-\big[\Gamma_{N_2}\big]^{\omega^{*}}(t)\big\|_{L^{\infty}_{t \in I} L^2(\Omega^{*})\mathcal{H}^{\alpha}_{\xi}} \leq \big\|\,\mathcal{U}(t)\,\big(\Gamma_{N_1}(0)-\Gamma_{N_2}(0)\big)\big\|_{L^{\infty}_{t \in I} L^2(\Omega^{*})\mathcal{H}^{\alpha}_{\xi}}+$$
$$+\big\|\int_{0}^{t}\,\mathcal{U}(t-s)\,\big[\widehat{B}\,\big]^{\omega^{*}} \big(\big[\Gamma_{N_1}\big]^{\omega^{*}}(s)-\big[\Gamma_{N_2}\big]^{\omega^{*}}(s)\big)\,ds\,\big\|_{L^{\infty}_{t \in I} L^2(\Omega^{*})\mathcal{H}^{\alpha}_{\xi}}.$$
We use the fact that $\Gamma_{N_1}(0)-\Gamma_{N_2}(0)$ does not depend on $\omega^{*}$, unitarity,  and the fact that $\xi<\xi'$ in order to deduce that the first term is:
\begin{equation}
\label{first_termGammaN1N2}
\leq \big\|\Gamma_{N_1}(0)-\Gamma_{N_2}(0)\big\|_{\mathcal{H}^{\alpha}_{\xi'}}.
\end{equation}
We can use Minkowski's inequality and unitarity in order to deduce that the second term is:
$$\leq \Big\| \int_{0}^{t} \big\|\big[\widehat{B}\,\big]^{\omega^{*}} \big(\big[\Gamma_{N_1}\big]^{\omega^{*}}(s)-\big[\Gamma_{N_2}\big]^{\omega^{*}}(s)\big) \big\|_{L^2(\Omega^{*}) \mathcal{H}^{\alpha}_{\xi}}\,ds\, \Big\|_{L^{\infty}_{t \in I}}$$
\begin{equation}
\label{second_termGammaN1N2}
\leq T \cdot \big\|\big[\widehat{B}\,\big]^{\omega^{*}} \big[\Gamma_{N_1}\big]^{\omega^{*}}(s)-\big[\widehat{B}\,\big]^{\omega^{*}} \big[\Gamma_{N_2}\big]^{\omega^{*}}(s)\big\|_{L^{\infty}_{s \in I} L^2(\Omega^{*}) \mathcal{H}^{\alpha}_{\xi}}.
\end{equation}
We recall from Proposition \ref{CauchySequence1} that the sequence $\big(\big[\widehat{B}\,\big]^{\omega^{*}} \big[\Gamma_N\big]^{\omega^{*}}\big)_N$ is Cauchy in $L^{\infty}_{t \in I} L^2(\Omega^{*}) \mathcal{H}^{\alpha}_{\xi}$.
Consequently, $\big(\big[\Gamma_N\big]^{\omega^{*}}\big)_N$ is Cauchy in $L^{\infty}_{t \in I} L^2(\Omega^{*}) \mathcal{H}^{\alpha}_{\xi}$. Since $L^{\infty}_{t \in I} L^2(\Omega^{*}) \mathcal{H}^{\alpha}_{\xi}$ is a Banach space, it follows that there exists $\big[\Gamma\,\big]^{\omega^{*}} \in L^{\infty}_{t \in I} L^2(\Omega^{*}) \mathcal{H}^{\alpha}_{\xi}$ such that:
\begin{equation}
\notag
\big[\Gamma_N\big]^{\omega^{*}} \rightarrow \big[\Gamma \,\big]^{\omega^{*}}
\end{equation}
strongly in $L^{\infty}_{t \in I} L^2(\Omega^{*}) \mathcal{H}^{\alpha}_{\xi}$ as $N \rightarrow \infty.$

%Proposition \ref{limit2} that $\big[\widehat{B}\,\big]^{\omega^{*}} \big[\Gamma_N\big]^{\omega^{*}} \rightarrow \theta^{\,\omega^{*}}$ strongly in $L^{\infty}_{t \in I} L^2(\Omega^{*}) \mathcal{H}^{\alpha}_{\xi}$ as $N \rightarrow \infty$. 

By construction of $\big[\Gamma_N\big]^{\omega^{*}}$, we know that, for all $k \in \mathbb{N}$ and for all $t \in I$:
$$\Big(\big[\Gamma_N\big]^{\omega^{*}}\Big)^{(k)}(t)=\mathcal{U}^{(k)}(t)\, \gamma_0^{(k)} -i \int_{0}^{t} 
\,\mathcal{U}^{(k)}\,(t-t_1) \big[B^{(k+1)}\big]^{\omega_{k+1}} \,\mathcal{U}^{(k+1)}(t_1)\, \gamma_0^{(k+1)}\, dt_1+...
$$
$$\cdots+(-i)^{N-k} \int_{0}^{t} \int_{0}^{t_1} \cdots \int_{0}^{t_{N-k}}
\mathcal{U}^{(k)}(t-t_1)\,\big[B^{(k+1)}\big]^{\omega_{k+1}} \, \mathcal{U}^{(k+1)}(t_1-t_2)
\,\big[B^{(k+2)}\big]^{\omega_{k+2}}\,\cdots$$ 
$$\cdots \big[B^{(N)}\big]^{\omega_{N}} \,\mathcal{U}^{(N)}(t_{N-k+1})\,\gamma_0^{(N)}\,dt_{N-k+1} \cdots dt_2 \, dt_1.
$$
which depends only on the random parameters $\omega_{k+1},\omega_{k+2},\ldots,\omega_N$. Hence, if $k \geq 2$ then $\big(\big[\Gamma_N\big]^{\omega^{*}}\big)^{(k)}$ does not depend on the random parameters $\omega_2, \omega_3, \ldots, \omega_k$. Since $\big[\Gamma_N\big]^{\omega^{*}} \rightarrow \big[\Gamma \,\big]^{\omega^{*}}$ strongly in $L^{\infty}_{t \in I} L^2(\Omega^{*}) \mathcal{H}^{\alpha}_{\xi}$ as $N \rightarrow \infty$, it follows that, for all $k \geq 2$ and for all $t \in I$, 
$\big(\big[\Gamma\big]^{\omega^{*}}\big)^{(k)}(t)$ does not depend on $\omega_2,\omega_3,\ldots,\omega_k$.

Let us now prove \eqref{Gamma_omega_star_bound}. We fix $N_1 \in \mathbb{N}$ as before. By using the same arguments as in \eqref{first_termGammaN1N2} and \eqref{second_termGammaN1N2}, it follows that:
$$\big\|\big[\Gamma_{N_1}\big]^{\omega^{*}}\big\|_{L^{\infty}_{t \in I} L^2(\Omega^{*}) \mathcal{H}^{\alpha}_{\xi}} 
\leq \big\|\Gamma_{N_1}(0)\big\|_{\mathcal{H}^{\alpha}_{\xi}} +
T \cdot \big\|\big[\widehat{B}\,\big]^{\omega^{*}} \big[\Gamma_N\,\big]^{\omega^{*}} \big\|_{L^{\infty}_{t \in I} L^2(\Omega^{*}) \mathcal{H}^{\alpha}_{\xi}}
.$$
By \eqref{BN1bound} from Remark \ref{BN1}, it follows that this quantity is:
$$\lesssim_{\,T,\,\xi,\,\xi',\,\alpha} \big\|\Gamma_{N_1}(0)\big\|_{\mathcal{H}^{\alpha}_{\xi}} + \big\|\Gamma(0)\big\|_{\mathcal{H}^{\alpha}_{\xi'}}.$$
Since $T$ is chosen to depend on $\xi,\xi'$, and $\alpha$, it follows that:
$$\big\|\big[\Gamma_{N_1}\big]^{\omega^{*}}\big\|_{L^{\infty}_{t \in I} L^2(\Omega^{*}) \mathcal{H}^{\alpha}_{\xi}} \lesssim_{\,\xi,\xi',\alpha} \big\|\Gamma(0)\big\|_{\mathcal{H}^{\alpha}_{\xi'}}.$$
We let $N_1 \rightarrow \infty$ and we use the fact that $\big[\Gamma_{N_1}\,\big]^{\omega^{*}} \rightarrow \big[\Gamma\,\big]^{\omega^{*}}$ in $L^{\infty}_{t \in I} L^2(\Omega^{*}) \mathcal{H}^{\alpha}_{\xi}$ in order to deduce \eqref{Gamma_omega_star_bound}.
The proposition now follows.
\end{proof}

Let us now find the equation which is satisfied by the limiting object $\big[\Gamma\,\big]^{\omega^{*}}$. In other words, we would like to find an analogue of Proposition \ref{theta_omega_star_equation} for the limit of $\big(\big[\Gamma_N\big]^{\omega^{*}}\big)_N$. We will now prove:

\begin{proposition}
\label{Gamma_omega_star_equation}
For $\big[\Gamma\,\big]^{\omega^{*}}$ from Proposition \ref{CauchySequence2} and for all $\xi_0>0$, the following identity holds:
$$\big\|\big[\Gamma\,\big]^{\omega^{*}}-\mathcal{U}(t)\,\Gamma(0)+i\int_{0}^{t} \mathcal{U}(t-s)\,\theta^{\,\omega^{*}}(s)\,ds \,\big\|_{L^{\infty}_{t \in I} L^2(\Omega^{*}) \mathcal{H}^{\alpha}_{\xi_0}}=0.$$
\end{proposition}

\begin{proof}
As before, it suffices to show the claim for $\xi_0 \in (0,\xi)$, where $\xi$ is as in \eqref{xi}.
We note that:
$$\big\|\big[\Gamma\,\big]^{\omega^{*}}(t)-\mathcal{U}(t)\,\Gamma(0)+i\int_{0}^{t} \mathcal{U}(t-s)\,\theta^{\,\omega^{*}}(s)\,ds \,\big\|_{L^{\infty}_{t \in I} L^2(\Omega^{*}) \mathcal{H}^{\alpha}_{\xi_0}}$$
$$
\leq \big\|\big[\Gamma\,\big]^{\omega^{*}}(t)-\big[\Gamma_N\big]^{\omega^{*}}(t)\big\|_{L^{\infty}_{t \in I} L^2(\Omega^{*}) \mathcal{H}^{\alpha}_{\xi_0}} + \big\|\,\mathcal{U}(t)\big(\Gamma(0)-\Gamma_N(0)\big)\big\|_{L^{\infty}_{t \in I} L^2(\Omega^{*}) \mathcal{H}^{\alpha}_{\xi_0}}+
$$
$$+\big\|\int_{0}^{t}\,\mathcal{U}(t-s)\,\big(\theta^{\,\omega^{*}}(s)-\big[\widehat{B}\,\big]^{\omega^{*}}\big[\Gamma_N\big]^{\omega^{*}}(s)\big)\,ds\,\big\|_{L^{\infty}_{t \in I} L^2(\Omega^{*}) \mathcal{H}^{\alpha}_{\xi_0}}+$$
$$+\big\|\big[\Gamma_N\big]^{\omega^{*}}(t)-\mathcal{U}(t)\,\Gamma_N(0)+i\int_{0}^{t}\,\mathcal{U}(t-s)\,\big[\widehat{B}\,\big]^{\omega^{*}}\big[\Gamma_N\big]^{\omega^{*}}(s)\,ds\big\|_{L^{\infty}_{t \in I} L^2(\Omega^{*})\mathcal{H}^{\alpha}_{\xi_0}}.$$

\textbf{1)}
The first term converges to zero as $N \rightarrow \infty$ by using Proposition \ref{CauchySequence2} and the fact that $\xi_0<\xi.$

\textbf{2)}
By unitarity, the second term equals: 
$$\big\|\Gamma(0)-\Gamma_N(0)\big\|_{L^2(\Omega^{*}) \mathcal{H}^{\alpha}_{\xi_0}}= \big\|\mathbf{P}_{>N}\Gamma(0)\big\|_{L^2(\Omega^{*})\mathcal{H}^{\alpha}_{\xi_0}}=\big\|\mathbf{P}_{>N}\Gamma(0)\big\|_{\mathcal{H}^{\alpha}_{\xi_0}}$$
by the property of probability measures.
Since $\xi_0<\xi'$, this quantity is bounded from above by $\big\|\mathbf{P}_{>N}\Gamma(0)\big\|_{\mathcal{H}^{\alpha}_{\xi'}}$, which converges to zero as $N \rightarrow \infty$ since $\Gamma(0) \in \mathcal{H}^{\alpha}_{\xi'}$.

\textbf{3)}
By Minkowski's inequality, the third term is:
$$\leq \Big\| \int_{0}^{t} \big\|\theta^{\,\omega^{*}}(s)-\big[\widehat{B}\,\big]^{\omega^{*}}\big[\Gamma_N\big]^{\omega^{*}}(s)\big\|_{L^{\infty}_{t \in I} L^2(\Omega^{*})\mathcal{H}^{\alpha}_{\xi_0}}\,ds\,\Big\|_{L^{\infty}_{t \in I}}$$
$$\leq T \cdot \big\|\theta^{\,\omega^{*}}(s)-\big[\widehat{B}\,\big]^{\omega^{*}}\big[\Gamma_N\big]^{\omega^{*}}(s)\big\|_{L^{\infty}_{s \in I} L^2(\Omega^{*}) \mathcal{H}^{\alpha}_{\xi_0}},$$
which converges to zero as $N \rightarrow \infty$ by Proposition \ref{limit2}.

\textbf{4)}
The fourth term equals zero by construction of $\big[\Gamma_N\big]^{\omega^{*}}$.

The proposition now follows.
\end{proof}
We can now prove the main result of this section:
\begin{theorem}
\label{Theorem1} 
For $\big[\Gamma\,\big]^{\omega^{*}}$ constructed above and for all $\xi_0>0$, the following identity holds:
\begin{equation}
\label{Theorem1identity1}
\big\|\big[\Gamma\,\big]^{\omega^{*}}(t)-\mathcal{U}(t)\,\Gamma(0)+i \int_{0}^{t}\,\mathcal{U}(t-s)\,\big[\widehat{B}\,\big]^{\omega^{*}} \big[\Gamma\,\big]^{\omega^{*}}(s)\,ds\,\big\|_{L^{\infty}_{t \in I} L^2(\Omega^{*}) \mathcal{H}^{\alpha}_{\xi_0}}=0.
\end{equation}
In particular, for almost all $t \in I$, it is the case that for all $k \in \mathbb{N}$:
\begin{equation}
\label{Theorem1identity2}
\big\|\big(\big[\Gamma\,\big]^{\omega^{*}}(t)-\mathcal{U}(t)\,\Gamma(0)+i \int_{0}^{t}\,\mathcal{U}(t-s)\,
\big[\widehat{B}\,\big]^{\omega^{*}}\big[\Gamma\,\big]^{\omega^{*}}(s)\,ds\,\big)^{(k)}\big\|_{L^2(\Omega^{*}) H^{\alpha}(\Lambda^k \times \Lambda^k)}
=0.
\end{equation}
In other words, we obtain a solution to the independently randomized Gross-Pitaevskii hierarchy satisfying the a priori bound \eqref{Gamma_omega_star_bound}.
\end{theorem}

\begin{proof}
We note that \eqref{Theorem1identity1} implies \eqref{Theorem1identity2}. We now prove \eqref{Theorem1identity1}.
It suffices to prove the claim for $\xi_0 \in (0,\xi)$, where $\xi$ is as in \eqref{xi}. Let us recall that by Proposition \ref{limit2} and Proposition \ref{CauchySequence2}, we know that for all $k \geq 2$, $\big(\theta^{\,\omega^{*}}\big)^{(k)}$ and $(\big[\Gamma\,\big]^{\omega^{*}}\big)^{(k)}$ do not depend on $\omega_2,\omega_3, \ldots, \omega_k$. Hence, we can use estimate \eqref{RandomizedEstimate2III} from Proposition \ref{RandomizedEstimate2} in order to deduce that:
\begin{equation}
\label{Theorem1A}
\big\|\big[\widehat{B}\,\big]^{\omega^{*}} \big[\Gamma\,\big]^{\omega^{*}}(t)-\big[\widehat{B}\,\big]^{\omega^{*}}\,\mathcal{U}(t)\,\Gamma(0)+i\int_{0}^{t}\big[\widehat{B}\,\big]^{\omega^{*}}\,\mathcal{U}(t-s)\,\theta^{\,\omega^{*}}(s)\,ds\big\|_{L^{\infty}_{t \in I} L^2(\Omega^{*}) \mathcal{H}^{\alpha}_{\xi_0}}
\end{equation}
$$\lesssim_{\,\xi_0,\,\xi,\,\alpha} \big\|\big[\Gamma\,\big]^{\omega^{*}}(t)-\mathcal{U}(t)\,\Gamma(0)+i\int_{0}^{t}\,\mathcal{U}(t-s)\,\theta^{\,\omega^{*}}(s)\,ds\big\|_{L^{\infty}_{t \in I} L^2(\Omega^{*}) \mathcal{H}^{\alpha}_{\xi}}=0$$
The last equality follows from Proposition \ref{Gamma_omega_star_equation}.

By \eqref{Theorem1A} and Proposition \ref{theta_omega_star_equation}, it follows that:
\begin{equation}
\label{Theorem1B}
\big\|\big[\widehat{B}\,\big]^{\omega^{*}}\big[\Gamma\,\big]^{\,\omega^{*}}(t)-\theta^{\,\omega^{*}}(t)\big\|_{L^{\infty}_{t \in I} L^2(\Omega^{*}) \mathcal{H}^{\alpha}_{\xi_0}}=0.
\end{equation}
Moreover, we observe that, by Minkowski's inequality and unitarity:

\begin{equation}
\label{Theorem1C}
\big\|\int_{0}^{t}\,\mathcal{U}(t-s)\,\big(\theta^{\,\omega^{*}}(s)-\big[\widehat{B}\,\big]^{\omega^{*}} \big[\Gamma\,\big]^{\omega^{*}}(s)\big)\,ds\,\big\|_{L^{\infty}_{t \in I} L^2(\Omega^{*}) \mathcal{H}^{\alpha}_{\xi_0}}
\end{equation}
$$\leq T \cdot \big\|\theta^{\,\omega^{*}}-\big[\widehat{B}\,\big]^{\omega^{*}}\big[\Gamma\,\big]^{\omega^{*}}\big\|_{L^{\infty}_{t \in I} L^2(\Omega^{*}) \mathcal{H}^{\alpha}_{\xi_0}}$$

This quantity equals zero by \eqref{Theorem1B}.
Finally, we can combine Proposition \ref{Gamma_omega_star_equation} and \eqref{Theorem1C} in order to deduce that:
$$\big\|\big[\Gamma\,\big]^{\omega^{*}}(t)-\mathcal{U}(t)\,\Gamma(0)+i\int_{0}^{t} \,\mathcal{U}(t-s)\,\big[\widehat{B}\,\big]^{\omega^{*}} \big[\Gamma\,\big]^{\omega^{*}}(s)\,ds\,\big\|_{L^{\infty}_{t \in I}L^2(\Omega^{*})\mathcal{H}^{\alpha}_{\xi_0}}=0.$$
Theorem \ref{Theorem1} now follows.
\end{proof}

\begin{remark}
\label{BhatGamma_omega_starA}
By using the identity \eqref{Theorem1B} and the estimate \eqref{theta_omega_star_bound1} from Remark \ref{theta_omega_star_boundA}, we can deduce that:

\begin{equation}
\label{BhatGamma_omega_star1}
\big\|\big[\widehat{B}\,\big]^{\omega^{*}}\big[\Gamma\,\big]^{\,\omega^{*}}\big\|_{L^{\infty}_{t \in I} L^2(\Omega^{*}) \mathcal{H}^{\alpha}_{\xi}} \lesssim_{\,T,\,\xi,\,\xi',\,\alpha} \big\|\Gamma(0)\big\|_{\mathcal{H}^{\alpha}_{\xi'}}.
\end{equation}
Hence, we obtain a similar a priori estimate as in \eqref{Gamma_omega_star_bound}. As before, the $T$ dependence in the implied constant can be omitted since $T$ is chosen to depend on $\xi,\xi'$ and $\alpha$.
We note that \eqref{BhatGamma_omega_star1} is a probabilistic version of the a priori bound that was assumed in the work of Klainerman and Machedon \cite{KM}. Analogous estimates were obtained in the deterministic setting in \cite{CP4}.
\end{remark}

\subsection{The case of additional regularity in the initial data $\Gamma(0)$}
\label{additional_regularity}

\medskip

In this subsection, we will show that the identity \eqref{Theorem1identity2} in Theorem \ref{Theorem1} can be improved to hold for \emph{all times $t \in I$} if one is willing to take slightly more regular initial data $\Gamma(0)$. Furthermore, in this case, the solution $\big[\Gamma\,\big]^{\omega^{*}}$ is continuous in time as opposed to being only bounded and measurable in time. More precisely, we will prove:

\begin{theorem}
\label{Theorem2}
Let $\alpha_0>\alpha>\frac{3}{4}$ be given.
There exists a constant $\widetilde{C_1}=\widetilde{C_1}(\alpha,\alpha_0)>0$ and a universal constant $\widetilde{C_2}>0$ such that, whenever $\xi,\xi'>0$ with $\xi \in (0,\xi')$ and $T>0$ satisfy the following assumptions:
\begin{itemize}
\item[i)] $\frac{\widetilde{C_1}T}{\xi'}<1$
\item[ii)] $\frac{\widetilde{C_2} \,\xi}{\xi'}<1,$
\end{itemize}
and whenever $\Gamma(0) \in \mathcal{H}^{\alpha_0}_{\xi'}$, then there exists $\big[\Gamma\,\big]^{\omega^{*}} \in C_{t \in I} L^2(\Omega^{*}) \mathcal{H}^{\alpha}_{\xi}$ such that, for all  $t \in I$ and for all $\xi_0>0:$
\begin{equation}
\label{Theorem2identity1}
\big\|\big[\Gamma\,\big]^{\omega^{*}}(t)-\mathcal{U}(t)\,\Gamma(0)+i \int_{0}^{t}\,\mathcal{U}(t-s)\,\big[\widehat{B}\,\big]^{\omega^{*}} \big[\Gamma\,\big]^{\omega^{*}}(s)\,ds\,\big\|_{L^2(\Omega^{*}) \mathcal{H}^{\alpha}_{\xi_0}}=0.
\end{equation}
In particular, for all $t \in I$, it is the case that for all $k \in \mathbb{N}$:
\begin{equation}
\label{Theorem2identity2}
\big\|\big(\big[\Gamma\,\big]^{\omega^{*}}(t)-\mathcal{U}(t)\,\Gamma(0)+i \int_{0}^{t}\,\mathcal{U}(t-s)\,
\big[\widehat{B}\,\big]^{\omega^{*}}\big[\Gamma\,\big]^{\omega^{*}}(s)\,ds\,\big)^{(k)}\big\|_{L^2(\Omega^{*}) H^{\alpha}(\Lambda^k \times \Lambda^k)}
\end{equation}
\begin{equation}
\notag
=0.
\end{equation}
Moreover, $\big[\Gamma\,\big]^{\omega^{*}}$ satisfies the a priori bound:
\begin{equation}
\label{Theorem2_a_priori_bound}
\big\|\big[\Gamma\,\big]^{\omega^{*}}\big\|_{L^{\infty}_{t \in I} L^2(\Omega^{*}) \mathcal{H}^{\alpha}_{\xi}} \lesssim_{\,\xi,\,\xi',\alpha,\alpha_0} \big\|\Gamma(0)\big\|_{\mathcal{H}^{\alpha}_{\xi'}}.
\end{equation}
\end{theorem}

\begin{remark}
\label{strong solution}
In the terminology used in the study of nonlinear dispersive equations \cite{Tao}, the solution $\big[\Gamma\,\big]^{\omega^{*}}$ in Theorem \ref{Theorem2} can be thought of as a \textbf{strong solution} to the independently randomized GP hierarchy, whereas the solution in Theorem \ref{Theorem1} can be thought of as a \textbf{weak solution}. In other words, by adding additional regularity in the initial data $\Gamma(0)$, it is possible to upgrade the weak solution to a strong solution. 
\end{remark}

The following definition will be useful in the discussion that follows:

\begin{definition}
\label{equicontinuousOmegastar}
A sequence $(F_N)_{N \in \mathbb{N}}$ in $C_{t \in I}L^2(\Omega^{*})\mathcal{H}^{\alpha}_{\xi}$ is said to be \textbf{equicontinuous in} $L^2(\Omega^{*})\mathcal{H}^{\alpha}_{\xi}$ if for all $\epsilon>0$, there exists $\delta>0$, such that for all $N \in \mathbb{N}$, and for all $t_1, t_2 \in I$,  $|t_1-t_2| \leq \delta$ implies that $\big\|F_N(t_1)-F_N(t_2)\big\|_{L^2(\Omega^{*}) \mathcal{H}^{\alpha}_{\xi}} \leq \epsilon.$
\end{definition}

In the proof of Theorem \ref{Theorem2}, we will use the following result:

\begin{lemma}
\label{DifferenceBound}
Let us fix $\beta_0>\beta>0$. Suppose that $k \in \mathbb{N}$ and suppose that $\sigma_0^{(k)} \in H^{\beta_0}(\Lambda^k \times \Lambda^k)$. Then, there exist constants $C=C(\beta,\beta_0)>0$ and $r=r(\beta,\beta_0)>0$, independent of $k$ such that for all $t \in \mathbb{R}$ and  $\delta>0$, the following bound holds:
$$\big\|\big(\,\mathcal{U}^{(k)}(t+\delta)-\mathcal{U}^{(k)}(t)\big) \sigma_0^{(k)}\big\|_{H^{\beta}(\Lambda^k \times \Lambda^k)} \leq C \delta^{r} \, \cdot \, \big\|\sigma_0^{(k)}\big\|_{H^{\beta_0}(\Lambda^k \times \Lambda^k)}.$$

\end{lemma}

\begin{proof}(of Lemma \ref{DifferenceBound})

We first observe that, by the semigroup property:
$$\mathcal{U}^{(k)}(t+\delta)=\,\mathcal{U}^{(k)}(t) \, \mathcal{U}^{(k)}(\delta).$$
Consequently, by unitarity of $\mathcal{U}^{(k)}(t)$:
$$\big\|\big(\,\mathcal{U}^{(k)}(t+\delta)-\mathcal{U}^{(k)}(t)\big) \sigma_0^{(k)}\big\|_{H^{\beta}(\Lambda^k \times \Lambda^k)}=\big\|\big(\,\mathcal{U}^{(k)}(\delta)-I^{(k)}\big) \sigma_0^{(k)}\big\|_{H^{\beta}(\Lambda^k \times \Lambda^k)}.$$
Here, by $I^{(k)}$, we denote the identity operator on density matrices of order $k$.
We observe that:
\begin{equation}
\label{difference12}
\Big(\big(\,\mathcal{U}^{(k)}(\delta)-I^{(k)}\big) \sigma_0^{(k)} \Big)\,\widehat{\,}\,\,(\vec{\xi}_k; \vec{\xi}'_k)=\big(e^{-i\delta(|\vec{\xi}_k|^2-|\vec{\xi}'_k|^2)}-1\big) \cdot \widehat{\sigma}_0^{(k)}(\vec{\xi}_k;\vec{\xi}'_k).
\end{equation}
Let us note that:
\begin{equation}
\label{difference1}
\big|e^{-i\delta(|\vec{\xi}_k|^2-|\vec{\xi}'_k|^2)}-1\big| \leq 2.
\end{equation}
Moreover, we can use the Mean Value Theorem to deduce that:
\begin{equation}
\label{difference2}
e^{-i\delta(|\vec{\xi}_k|^2-|\vec{\xi}'_k|^2)}-1= \Big(\cos\big(\delta(|\vec{\xi}_k|^2-|\vec{\xi}'_k|^2)\big)-1\Big) - i \sin \big(\delta (|\vec{\xi}_k|^2-|\vec{\xi}'_k|^2)\big)
\end{equation}
$$=\mathcal{O}\Big(\delta (|\vec{\xi}_k|^2-|\vec{\xi}'_k|^2)\Big)=\mathcal{O}\Big(\delta (|\vec{\xi}_k|^2+|\vec{\xi}'_k|^2)\Big)=$$
$$=\mathcal{O}\Big(\delta \cdot \langle \xi_1 \rangle^{2} \cdots \langle \xi_k \rangle^{2} \cdot \langle \xi'_1 \rangle^{2} \cdots \langle \xi'_k \rangle^{2} \Big).
$$
Here, the implied constant is independent of $k$.
The latter estimate follows directly from the fact that:
$$|\vec{\xi}_k|^2+|\vec{\xi}'_k|^2 \leq \langle \xi_1 \rangle^2 \cdots \langle \xi_k \rangle^2 \cdot \langle \xi'_1 \rangle^2 \cdots \langle \xi'_k \rangle^2,$$
with no implied constant in the estimate.

We now interpolate between the estimates in \eqref{difference1} and \eqref{difference2}, and use Plancherel's theorem in \eqref{difference12} in order to conclude that, for all $r \in [0,1]$:
$$\big\|\big(\,\mathcal{U}^{(k)}(\delta)-I^{(k)}\big) \sigma_0^{(k)}\big\|_{H^{\beta}(\Lambda^k \times \Lambda^k)} \leq C(r) \cdot \delta^{r} \cdot \big\|\sigma_0^{(k)}\big\|_{H^{\beta+2r}(\Lambda^k \times \Lambda^k)}.$$
Here, $C(r)$ is a positive constant depending on $r$.
The lemma now follows if we take $r:=\frac{\beta_0-\beta}{2}$ whenever $\beta_0 \leq \beta +2$ and $r=1$ otherwise. In each case, we take $C=C(r)$, depending on $r$. We note that both $r$ and $C$ are independent of $k$.
\end{proof}

\begin{proof}(of Theorem \ref{Theorem2})
We will show that $\big[\Gamma\,\big]^{\omega^{*}}$, constructed in Proposition \ref{CauchySequence2}, has the wanted properties. As was noted above, the time interval might possibly be smaller due to the additional dependence of $T$ on $\alpha_0$. Furthermore, $\xi'$ should be chosen to be smaller in terms of $\xi$ due to the additional $\alpha_0$ dependence. The a priori bound \eqref{Theorem2_a_priori_bound} follows from the proof of \eqref{Gamma_omega_star_bound}. The $\alpha_0$ dependence in the implied constant comes from the fact that $T$ depends on $\alpha_0$. Similarly, we might have to choose $\xi$ to be smaller in terms of $\xi'$. For simplicity of exposition, we will not explicitly mention these distinctions in the discussion below.

As we will see, the continuity in time will follow as a result of the assumption of additional regularity. We note that \eqref{Theorem2identity1} implies \eqref{Theorem2identity2}. Moreover, it suffices to prove \eqref{Theorem2identity1} for $\xi_0 \in (0,\xi)$, where $\xi$ is as in the statement of the theorem. Let us henceforth in the proof consider such a $\xi_0$.

\medskip

The proof of the theorem will be divided into several steps.

\medskip

\textbf{Step 1:} The sequence $\big(\big[\Gamma_N\,\big]^{\omega^{*}}\big)_N$, constructed as in \eqref{GammaN}, is equicontinuous in $L^2(\Omega^{*}) \mathcal{H}^{\alpha}_{\xi}$.

Let us recall the construction of $\big(\big[\Gamma_N\,\big]^{\omega^{*}}\big)_N$ from \eqref{GammaN}. This is done on the time interval $[0,T]$, for $T$ satisfying \eqref{xi}, where $\xi$ and $\xi'$ are assumed to satisfy \eqref{xi_prime}. This construction applies if we take $T$ to be smaller, due to the additional $\alpha_0$ dependence. We can take $T$ to be a function of $\xi,\xi',\alpha,\alpha_0$ by the assumptions $i)$ and $ii)$ of Theorem \ref{Theorem2} .

\medskip

Let us fix $N \in \mathbb{N}$ and let us consider $\big[\Gamma_N\,\big]^{\omega^{*}}$. Let us take $k \in \mathbb{N}$, $t \in I$ and $\delta>0$ small. We observe that:
\begin{equation}
\label{GammaNomega_delta}
\big\|\big(\big[\Gamma_N\,\big]^{\omega^{*}}\big)^{(k)}(t+\delta)-\big(\big[\Gamma_N\,\big]^{\omega^{*}}\big)^{(k)}(t)\big\|_{L^2(\Omega^{*}) H^{\alpha}(\Lambda^k \times \Lambda^k)}
\end{equation}
$$\leq \big\|\sum_{j=0}^{N-k} \big\{Duh_{\,j}^{\,\omega^{*}}(\Gamma_N(0))^{(k)}(t+\delta)- Duh_{\,j}^{\,\omega^{*}}(\Gamma_N(0))^{(k)}(t) \big\}\big\|_{L^2(\Omega^{*}) H^{\alpha}(\Lambda^k \times \Lambda^k)}$$
$$\leq \sum_{j=0}^{N-k} \big\|Duh_{\,j}^{\,\omega^{*}}(\Gamma_N(0))^{(k)}(t+\delta)- Duh_{\,j}^{\,\omega^{*}}(\Gamma_N(0))^{(k)}(t)\big\|_{L^2(\Omega^{*}) H^{\alpha}(\Lambda^k \times \Lambda^k)}.$$
By construction:
\begin{equation}
\notag
%\label{difference_delta}
Duh_{\,j}^{\,\omega^{*}}(\Gamma_N(0))^{(k)}(t+\delta)- Duh_{\,j}^{\,\omega^{*}}(\Gamma_N(0))^{(k)}(t)=
\end{equation}
$$=(-i)^j \int_0^t \int_0^{t_1} \cdots \int_0^{t_{j-1}} \big(\,\mathcal{U}^{(k)}(t+\delta-t_1)-\,\mathcal{U}^{(k)}(t-t_1)\big) \big[B^{(k+1)}\big]^{\omega_{k+1}} \, \mathcal{U}^{(k+1)}(t_1-t_2) \, \big[B^{(k+2)}\big]^{\omega_{k+2}} \cdots $$
$$\cdots \,\mathcal{U}^{(k+j-1)}(t_{j-1}-t_j) \, \big[B^{(k+j)}\big]^{\omega_{k+j}} \, \mathcal{U}^{(k+j)}(t_j) \gamma_0^{(k+j)}\, dt_j \, dt_{j-1} \,\cdots\, dt_2 \, dt_1$$
$$+(-i)^j \int_{t}^{t+\delta} \int_0^{t_1} \cdots \int_0^{t_{j-1}} \,\mathcal{U}^{(k)}(t+\delta-t_1) \, \big[B^{(k+1)}\big]^{\omega_{k+1}} \, \mathcal{U}^{(k+1)}(t_1-t_2) \, \big[B^{(k+2)}\big]^{\omega_{k+2}} \cdots $$
$$\cdots \,\mathcal{U}^{(k+j-1)}(t_{j-1}-t_j) \, \big[B^{(k+j)}\big]^{\omega_{k+j}} \, \mathcal{U}^{(k+j)}(t_j) \gamma_0^{(k+j)}\, dt_j \, dt_{j-1} \,\cdots\, dt_2 \, dt_1$$
$$=: \mathcal{I} + \mathcal{II}.$$
Let us now estimate the quantities $\mathcal{I}$ and $\mathcal{II}$ in the norm $\big\|\cdot\big\|_{L^2(\Omega^{*})H^{\alpha}(\Lambda^k \times \Lambda^k)}.$

\medskip

By Minkowski's inequality:

$$\big\|\,\mathcal{I}\,\big\|_{L^2(\Omega^{*}) H^{\alpha}(\Lambda^k \times \Lambda^k)} \leq \int_0^t \int_0^{t_1} \cdots \int_0^{t_{j-1}} \big\| \big(\,\mathcal{U}^{(k)}(t+\delta-t_1)-\,\mathcal{U}^{(k)}(t-t_1)\big) \, \big[B^{(k+1)}\big]^{\omega_{k+1}} \, \mathcal{U}^{(k+1)}(t_1-t_2)$$
$$\big[B^{(k+2)}\big]^{\omega_{k+2}} \cdots \,\mathcal{U}^{(k+j-1)}(t_{j-1}-t_j) \, \big[B^{(k+j)}\big]^{\omega_{k+j}} \, \mathcal{U}^{(k+j)}(t_j) \gamma_0^{(k+j)}\big\|_{L^2(\Omega^{*}) H^{\alpha}(\Lambda^k \times \Lambda^k)}\, dt_j \, dt_{j-1} \,\cdots\, dt_2 \, dt_1,$$
which, by Lemma \ref{DifferenceBound} is:
$$\lesssim_{\,\alpha,\,\alpha_0} \delta^{r} \cdot \int_0^t \int_0^{t_1} \cdots \int_0^{t_{j-1}} \big\| \big[B^{(k+1)}\big]^{\omega_{k+1}} \, \mathcal{U}^{(k+1)}(t_1-t_2) \, \big[B^{(k+2)}\big]^{\omega_{k+2}} \cdots $$
$$\cdots \,\mathcal{U}^{(k+j-1)}(t_{j-1}-t_j) \, \big[B^{(k+j)}\big]^{\omega_{k+j}} \, \mathcal{U}^{(k+j)}(t_j) \gamma_0^{(k+j)}\big\|_{L^2(\Omega^{*}) H^{\alpha_0}(\Lambda^k \times \Lambda^k)}\, dt_j \, dt_{j-1} \,\cdots\, dt_2 \, dt_1$$
for some $r=r(\alpha,\alpha_0) \in (0,1)$.
By iteratively applying Theorem \ref{RandomizedEstimate1}, as in the proof of Proposition \ref{DuhamelEstimate1}, and by applying the identity \eqref{Integral_Identity}, it follows that this quantity is:
$$\leq \delta^{r} \cdot \frac{T^j}{j!}\, \cdot \, (C_0')^j \, \cdot \, k \, \cdot \, (k+1) \,  \cdots  \, (k+j-1) \, \cdot \big\|\gamma_0^{(k+j)}\big\|_{H^{\alpha_0}(\Lambda^{k+j} \times \Lambda^{k+j})}.$$
Here, $C_0'=C_0'(\alpha_0)>0$ is the implied constant which we obtain for the regularity $\alpha_0$ in Theorem \ref{RandomizedEstimate1}. The estimates in the proof of Proposition \ref{DuhamelEstimate1} further imply that this quantity is:
\begin{equation}
\label{I}
\leq \delta^r \cdot (C_1' T)^j \cdot \, ({C'_2})^k \cdot \big\| \gamma_0^{(k+j)}\big\|_{H^{\alpha_0}(\Lambda^{k+j} \times \Lambda^{k+j})}
\end{equation}
for some constant $C_1'=C_1'(\alpha_0)>0$ and for some universal constant $C'_2>0$ (which can be taken to be the universal constant $C_2$ from Proposition \ref{DuhamelEstimate1}). 

\medskip

Let us fix $t \in I$ and $\delta>0$ such that $t+\delta \in I$.
We note that:

\begin{equation}
\label{integral_bound}
\int_{0}^{t+\delta} \int_{0}^{t_1} \cdots \int_{0}^{t_{j-1}} \,dt_j\,dt_{j-1} \cdots \,dt_2 \, dt_1 = \int_{t}^{t+\delta} \frac{t_1^{j-1}}{(j-1)!} \, dt_1 \leq \delta \cdot \frac{T^{j-1}}{(j-1)!}.
\end{equation}
%$$=\frac{(t+\delta)^j}{j!}-\frac{t^j}{j!}.$$
By the Mean Value Theorem, it follows that $(t+\delta)^j - t^j= \delta \cdot j \cdot t_0^{j-1}$ for some $t_0 \in (t,t+\delta).$

By Minkowski's inequality:
$$\big\|\,\mathcal{II}\,\big\|_{L^2(\Omega^{*}) H^{\alpha}(\Lambda^k \times \Lambda^k)} \leq 
\int_{t}^{t+\delta} \int_0^{t_1} \cdots \int_0^{t_{j-1}} \big\|\,\mathcal{U}^{(k)}(t+\delta-t_1) \, \big[B^{(k+1)}\big]^{\omega_{k+1}} \, \mathcal{U}^{(k+1)}(t_1-t_2) \, \big[B^{(k+2)}\big]^{\omega_{k+2}} \cdots $$
$$\cdots \,\mathcal{U}^{(k+j-1)}(t_{j-1}-t_j) \, \big[B^{(k+j)}\big]^{\omega_{k+j}} \, \mathcal{U}^{(k+j)}(t_j) \gamma_0^{(k+j)}\,\big\|_{L^2(\Omega^{*}) H^{\alpha}(\Lambda^k \times \Lambda^k)}\, dt_j \, dt_{j-1} \,\cdots\, dt_2 \, dt_1
$$
which by iteratively applying Theorem \ref{RandomizedEstimate1} is:
$$\leq \int_{t}^{t+\delta} \int_{0}^{t_1} \cdots \int_{0}^{t_{j-1}} C_0^j \, \cdot \, k \, \cdot \, (k+1) \cdots (k+j-1) \, \cdot \, \big\|\gamma_0^{(k+j)}\big\|_{H^{\alpha}(\Lambda^{k+j} \times \Lambda^{k+j})} \,dt_j \, dt_{j-1} \cdots \,dt_2 \, dt_1$$
$$\leq \int_{t}^{t+\delta} \int_{0}^{t_1} \cdots \int_{0}^{t_{j-1}} C_0^j \,\cdot\, k \, \cdot \, (k+j-1)^{j-1}\, \cdot \, \big\|\gamma_0^{(k+j)}\big\|_{H^{\alpha}(\Lambda^{k+j} \times \Lambda^{k+j})} \,dt_j \, dt_{j-1} \cdots \,dt_2 \, dt_1$$
for the implied constant $C_0=C_0(\alpha)$ obtained in Theorem \ref{RandomizedEstimate1} corresponding to the regularity $\alpha$.
By \eqref{integral_bound}, this quantity is:
$$\leq \delta \cdot \frac{T^{j-1}}{(j-1)!} \, \cdot  \,  C_0^j \,\cdot\, k \, \cdot \, (k+j-1)^{j-1} \, \cdot 
\big\|\gamma_0^{(k+j)}\big\|_{H^{\alpha}(\Lambda^{k+j} \times \Lambda^{k+j})}.$$
By arguing as in the estimate for $\big\|\,\mathcal{II}\,\big\|_{L^2(\Omega^{*})H^{\alpha}(\Lambda^k \times \Lambda^k)}$, it follows that, for the same universal constant $C'_2>0$, and for some $C_1=C_1(\alpha)$ (which is the same as the constant $C_1$ in Proposition \ref{DuhamelEstimate1}), we obtain:
$$\big\|\,\mathcal{II}\,\big\|_{L^2(\Omega^{*})H^{\alpha}(\Lambda^k \times \Lambda^k)} \leq \delta \cdot T^{j-1} \cdot C_0 \cdot C_1^{j-1} \cdot \,k \, \cdot \,  (C'_2)^k \cdot \big\|\gamma_0^{(k+j)}\big\|_{H^{\alpha}(\Lambda^{k+j} \times \Lambda^{k+j})}$$ 
\begin{equation}
\label{II}
\leq \frac{C_0 \delta}{C_1 T} \cdot (C_1T)^j \cdot \,(2C'_2)^k \cdot \big\|\gamma_0^{(k+j)}\big\|_{H^{\alpha}(\Lambda^{k+j} \times \Lambda^{k+j})}.
\end{equation}
We now define:
\begin{equation}
\label{C1tildeC2tilde}
\widetilde{C_1}:=\,\max\{C_1,C_1'\},\widetilde{C_2}:=2C'_2.
\end{equation}
We note that then $\widetilde{C_1}$ depends on $\alpha$ and on $\alpha_0$, and is independent of $j$ and $k$, whereas $\widetilde{C_2}$ is a universal constant. 

\medskip

Combining \eqref{I} and \eqref{II}, it follows that:
$$\big\| Duh_{\,j}^{\,\omega^{*}}(\Gamma_N(0))^{(k)}(t+\delta)- Duh_{\,j}^{\,\omega^{*}}(\Gamma_N(0))^{(k)}(t)\big\|_{L^2(\Omega^{*})H^{\alpha}(\Lambda^k \times \Lambda^k)}$$
\begin{equation}
\label{Difference_Estimate}
\lesssim_{\,\alpha,\,\alpha_0\,,\,T} \delta^{r} \cdot (\widetilde{C_1}T)^j \cdot \widetilde{C_2}^k \cdot \big\|\gamma_0^{(k+j)}\big\|_{H^{\alpha_0}(\Lambda^{k+j} \times \Lambda^{k+j})}.
\end{equation}
Here, we used the fact that that, by construction, $r \leq 1$.

From \eqref{GammaNomega_delta} and \eqref{Difference_Estimate}, it follows that:
\begin{equation}
\label{Difference_Estimate_42b}
\big\|\big[\Gamma_N\,\big]^{\omega^{*}}(t+\delta)-\big[\Gamma_N\,\big]^{\omega^{*}}(t)\big\|_{L^2(\Omega^{*}) \mathcal{H}^{\alpha}_{\xi}}=
\end{equation}
$$=\sum_{k=0}^{\infty} \,\xi^k \cdot  \big\|\big(\big[\Gamma_N\,\big]^{\omega^{*}}\big)^{(k)}(t+\delta)-\big(\big[\Gamma_N\,\big]^{\omega^{*}}\big)^{(k)}(t)\big\|_{L^2(\Omega^{*}) H^{\alpha}(\Lambda^k \times \Lambda^k)}$$
$$\lesssim_{\,\alpha,\,\alpha_0,\,T} \delta^{r} \cdot \sum_{k=0}^{\infty} \,\xi^k \cdot \sum_{j=0}^{N-k} (\widetilde{C_1}T)^j \cdot \widetilde{C_2}^k \cdot \big\|\gamma_0^{(k+j)}\big\|_{H^{\alpha_0}(\Lambda^{k+j} \times \Lambda^{k+j})}$$
$$= \delta^{r} \cdot \sum_{k=0}^{\infty} \sum_{j=0}^{N-k} \Big(\frac{\widetilde{C_2}\,\xi}{\xi'}\Big)^k \cdot \Big(\frac{\widetilde{C_1} \,T}{\xi'}\Big)^j \cdot \, (\xi')^{k+j} \, \cdot \big\|\gamma_0^{(k+j)}\big\|_{H^{\alpha_0}(\Lambda^{k+j} \times \Lambda^{k+j})}$$
$$\leq \delta^{r} \cdot \sum_{k=0}^{\infty} \sum_{j=0}^{N-k} \Big(\frac{\widetilde{C_2}\,\xi}{\xi'}\Big)^k \cdot \Big(\frac{\widetilde{C_1} \,T}{\xi'}\Big)^j \cdot \big\|\Gamma(0)\big\|_{\mathcal{H}^{\alpha_0}_{\xi'}}$$
$$\leq \delta^{r} \cdot \sum_{j=0}^{\infty} \Big(\frac{\widetilde{C_1} \,T}{\xi'}\Big)^j \cdot \sum_{k=0}^{\infty} \Big(\frac{\widetilde{C_2}\,\xi}{\xi'}\Big)^k \cdot \big\|\Gamma(0)\big\|_{\mathcal{H}^{\alpha_0}_{\xi'}}.$$
If we take $\xi,\xi'$, and $T$ to satisfy: $\frac{\widetilde{C_1} \,T}{\xi'}<1$ and $\frac{\widetilde{C_2}\,\xi}{\xi'}<1$, it follows that the above sum is:
$$\lesssim_{\,\xi,\,\xi',\,T,\,\alpha,\,\alpha_0} \delta^{r} \, \cdot \, \big\|\Gamma(0)\big\|_{\mathcal{H}^{\alpha_0}_{\xi'}}.$$
In particular, if the assumptions $i)$ and $ii)$ are satisfied with $\widetilde{C_1}$ and $\widetilde{C_2}$ as defined in \eqref{C1tildeC2tilde}, it follows that:
\begin{equation}
\label{equicontinuity}
\big\|\big[\Gamma_N\,\big]^{\omega^{*}}(t+\delta)-\big[\Gamma_N\,\big]^{\omega^{*}}(t)\big\|_{L^2(\Omega^{*})\mathcal{H}^{\alpha}_{\xi}} \lesssim_{\,\xi,\,\xi',\,\alpha,\,\alpha_0} \delta^{r} \, \cdot \, \big\|\Gamma(0)\big\|_{\mathcal{H}^{\alpha_0}_{\xi'}}.
\end{equation}
Here, we use the fact that $T$ can be chosen to be a function of $\xi,\xi',\alpha,\alpha_0$.

Since $\Gamma(0) \in \mathcal{H}^{\alpha_0}_{\xi'}$, the equicontinuity of $\big(\big[\Gamma_N\,\big]^{\omega^{*}}\big)_N$ in $L^2(\Omega^{*}) \mathcal{H}^{\alpha}_{\xi}$ now follows from \eqref{equicontinuity}.

\medskip

\textbf{Step 2:} The limit $\big[\Gamma\,\big]^{\omega^{*}}$ of $\big(\big[\Gamma_N\,\big]^{\omega^{*}}\big)_N$ belongs to $C_{t \in I} L^2(\Omega^{*}) \mathcal{H}^{\alpha}_{\xi}.$
\medskip

We recall that, from Proposition \ref{CauchySequence2}:

$$\big[\Gamma_N\,\big]^{\omega^{*}} \rightarrow \big[\Gamma\,\big]^{\omega^{*}}$$
as $N \rightarrow \infty$ in $L^{\infty}_{t \in I} L^2(\Omega^{*}) \mathcal{H}^{\alpha}_{\xi}.$
The fact that $\big[\Gamma\,\big]^{\omega^{*}} \in C_{t \in I} L^2(\Omega^{*}) \mathcal{H}^{\alpha}_{\xi}$ now follows from the equicontinuity result in Step 1.
In particular, the following bound holds by letting $N \rightarrow \infty$ in \eqref{equicontinuity}:

\begin{equation}
\label{equicontinuity2}
\big\|\big[\Gamma\,\big]^{\omega^{*}}(t+\delta)-\big[\Gamma\,\big]^{\omega^{*}}(t)\big\|_{L^2(\Omega^{*})\mathcal{H}^{\alpha}_{\xi}} \lesssim_{\,\xi,\,\xi',\,\alpha,\,\alpha_0} \delta^{r} \, \cdot \, \big\|\Gamma(0)\big\|_{\mathcal{H}^{\alpha_0}_{\xi'}}.
\end{equation}

\medskip

\textbf{Step 3:} $\mathcal{U}(t) \, \Gamma(0)$ belongs to $C_{t \in I} L^2(\Omega^{*}) \mathcal{H}^{\alpha}_{\xi_0}$.

\medskip

From Lemma \ref{DifferenceBound}, it follows that:

\begin{equation}
\notag
\big\|\,\mathcal{U}(t+\delta)\,\Gamma(0)-\mathcal{U}(t)\,\Gamma(0)\big\|_{L^2(\Omega^{*}) \mathcal{H}^{\alpha}_{\xi_0}} \lesssim_{\,\alpha,\,\alpha_0} \delta^r \, \cdot \,  \big\|\Gamma(0)\big\|_{\mathcal{H}^{\alpha_0}_{\xi_0}}.
\end{equation}
Here $r=r(\alpha,\alpha_0)>0$ is as above. Hence, $\big\|\,\mathcal{U}(t+\delta)\,\Gamma(0)-\mathcal{U}(t)\,\Gamma(0)\big\|_{L^2(\Omega^{*}) \mathcal{H}^{\alpha}_{\xi_0}}$ converges to zero as $\delta$ converges to zero since $\Gamma(0) \in \mathcal{H}^{\alpha_0}_{\xi_0} \subseteq \mathcal{H}^{\alpha_0}_{\xi'}$. It follows that $\mathcal{U}(t)\,\Gamma(0)$ belongs to $C_{t \in I} L^2(\Omega^{*}) \mathcal{H}^{\alpha}_{\xi_0}$.

\medskip

\textbf{Step 4:} $\big[\widehat{B}\,\big]^{\omega^{*}} \big[\Gamma\,\big]^{\omega^{*}}$ belongs to $C_{t \in I} L^2(\Omega^{*}) \mathcal{H}^{\alpha}_{\xi_0}$.

\medskip

We know from Proposition \ref{CauchySequence2} that $\big(\big[\Gamma\,\big]^{\omega^{*}}\big)^{(k)}$ does not depend on $\omega_2,\omega_3, \ldots, \omega_k$, whenever $k \geq 2$. In particular, we can use estimate \eqref{RandomizedEstimate2III} from Proposition \ref{RandomizedEstimate2} in order to deduce that: 

%for all $N \in \mathbb{N}$ and for all $k \in \mathbb{N}$:
%$$\big\| \big(\big[\widehat{B}\,\big]^{\omega^{*}} \big[\Gamma_N\,\big]^{\omega^{*}}\big)^{(k)}(t_1)-\big[\widehat{B}\,\big]^{\omega^{*}} \big[\Gamma_N\,\big]^{\omega^{*}}\big)^{(k)}(t_2)\big\|_{L^2(\Omega^{*}) H^{\alpha}(\Lambda^{k} \times \Lambda^{k})}=$$
%$$=\big\|\big[B_{k+1}\,\big]^{\omega_{k+1}}\Big( \big(\big[\Gamma_N\,\big]^{\omega^{*}}\big)^{(k+1)}(t_1)-\big(\big[\Gamma_N\,\big]^{\omega^{*}}\big)^{(k+1)}(t_2)\Big)\big\|_{L^2(\Omega^{*}) H^{\alpha}(\Lambda^{k} \times \Lambda^{k})}.$$
$$\big\|\big[\widehat{B}\,\big]^{\omega^{*}} \big[\Gamma\,\big]^{\omega^{*}}(t_1)-\big[\widehat{B}\,\big]^{\omega^{*}} \big[\Gamma\,\big]^{\omega^{*}}(t_2)\big\|_{L^2(\Omega^{*}) \mathcal{H}^{\alpha}_{\xi_0}} \lesssim_{\,\xi,\,\xi_0,\,\alpha} \big\|\big[\Gamma\,\big]^{\omega^{*}}(t_1)-\big[\Gamma\,\big]^{\omega^{*}}(t_2)\big\|_{L^2(\Omega^{*}) \mathcal{H}^{\alpha}_{\xi}}.$$
Since $\big[\Gamma\,\big]^{\omega^{*}}$ belongs to $C_{t \in I} L^2(\Omega^{*}) \mathcal{H}^{\alpha}_{\xi}$, it follows that $\big[\widehat{B}\,\big]^{\omega^{*}} \big[\Gamma\,\big]^{\omega^{*}}$ belongs to $C_{t \in I} L^2(\Omega^{*}) \mathcal{H}^{\alpha}_{\xi_0}$.

\medskip

\textbf{Step 5:} $\int_{0}^{t} \, \mathcal{U}(t-s) \, \big[\widehat{B}\,\big]^{\omega^{*}} \big[\Gamma\,\big]^{\omega^{*}}(s)\, ds$ belongs to $C_{t \in I} L^2(\Omega^{*}) \mathcal{H}^{\alpha}_{\xi_0}$.

\medskip

We note that:

$$\big\|\int_{0}^{t+\delta} \,\mathcal{U}(t+\delta-s) \, \big[\widehat{B}\,\big]^{\omega^{*}} \big[\Gamma\,\big]^{\omega^{*}}(s) \,ds - \int_{0}^{t} \,\mathcal{U}(t-s) \, \big[\widehat{B}\,\big]^{\omega^{*}} \big[\Gamma\,\big]^{\omega^{*}}(s) \,ds \,\big\|_{L^2(\Omega^{*}) \mathcal{H}^{\alpha}_{\xi}}$$
$$=\big\|\int_{-\delta}^{t} \,\mathcal{U}(t-s) \, \big[\widehat{B}\,\big]^{\omega^{*}} \big[\Gamma\,\big]^{\omega^{*}}(s+\delta) \,ds - \int_{0}^{t} \,\mathcal{U}(t-s) \, \big[\widehat{B}\,\big]^{\omega^{*}} \big[\Gamma\,\big]^{\omega^{*}}(s) \,ds \,\big\|_{L^2(\Omega^{*}) \mathcal{H}^{\alpha}_{\xi}}$$
$$\leq \big\|\int_{-\delta}^{0} \,\mathcal{U}(t-s)\,\big[\widehat{B}\,\big]^{\omega^{*}}\big[\Gamma\,\big]^{\omega^{*}}(s+\delta)\, ds\, \big\|_{L^2(\Omega^{*}) \mathcal{H}^{\alpha}_{\xi_0}}+$$
$$\big\| \int_{0}^{t} \, \mathcal{U}(t-s)\, \big(\big[\widehat{B}\,\big]^{\omega^{*}}\big[\Gamma\,\big]^{\omega^{*}}(s+\delta)-\big[\widehat{B}\,\big]^{\omega^{*}}\big[\Gamma\,\big]^{\omega^{*}}(s)\big)\big\|_{L^2(\Omega^{*}) \mathcal{H}^{\alpha}_{\xi_0}}.$$
By Minkowski's inequality and unitarity, this expression is:
$$\leq \delta \cdot \big\|\big[\widehat{B}\,\big]^{\omega^{*}} \big[\Gamma\,\big]^{\omega^{*}} \big\|_{L^{\infty}_{t \in I} L^2(\Omega^{*}) \mathcal{H}^{\alpha}_{\xi_0}} + \int_{0}^{t} \big\|\big[\widehat{B}\,\big]^{\omega^{*}} \big(\big[\Gamma\,\big]^{\omega^{*}}(s+\delta) - \big[\Gamma\,\big]^{\omega^{*}}(s)\big)\big\|_{L^2(\Omega^{*}) \mathcal{H}^{\alpha}_{\xi_0}}\, ds.$$
We can again use Proposition \ref{CauchySequence2} and estimate \eqref{RandomizedEstimate2III}
from Proposition \ref{RandomizedEstimate2} in order to see that this is:
$$\lesssim_{\,\xi,\,\xi_0,\,\alpha} \delta \, \cdot \, \big\|\big[\Gamma\,\big]^{\omega^{*}} \big\|_{L^{\infty}_{t \in I} L^2(\Omega^{*}) \mathcal{H}^{\alpha}_{\xi}} + \int_{0}^{t} \big\|\big[\Gamma\,\big]^{\omega^{*}}(s+\delta) - \big[\Gamma\,\big]^{\omega^{*}}(s)\big\|_{L^2(\Omega^{*}) \mathcal{H}^{\alpha}_{\xi}}\, ds.$$
By \eqref{equicontinuity2}, this expression is:
\begin{equation}
\notag
%\label{Step5Theorem2}
\lesssim_{\,\xi,\,\xi',\,\alpha,\,\alpha_0} \delta \, \cdot \, \big\|\big[\Gamma\,\big]^{\omega^{*}} \big\|_{L^{\infty}_{t \in I} L^2(\Omega^{*}) \mathcal{H}^{\alpha}_{\xi}} + \delta^{r} \, \cdot \, T \, \cdot \, \big\|\Gamma(0)\big\|_{\mathcal{H}^{\alpha_0}_{\xi'}}.
\end{equation}
This quantity converges to zero as $\delta$ converges to zero. Hence, it follows that the Duhamel term $\int_{0}^{t}\,\mathcal{U}(t-s)\,\big[\widehat{B}\,\big]^{\omega^{*}} \big[\Gamma\,\big]^{\omega^{*}}(s)\,ds$ belongs to $C_{t \in I} L^2(\Omega^{*}) \mathcal{H}^{\alpha}_{\xi_0}$.

\medskip

\textbf{Step 6:} Conclusion of the proof.

\medskip

We recall from Theorem \ref{Theorem1} that $\big[\Gamma\,\big]^{\omega^{*}}$ satisfies:

\begin{equation}
\notag
\big\|\big[\Gamma\,\big]^{\omega^{*}}(t)-\mathcal{U}(t)\,\Gamma(0)+i \int_{0}^{t}\,\mathcal{U}(t-s)\,\big[\widehat{B}\,\big]^{\omega^{*}} \big[\Gamma\,\big]^{\omega^{*}}(s)\,ds\,\big\|_{L^{\infty}_{t \in I} L^2(\Omega^{*}) \mathcal{H}^{\alpha}_{\xi_0}}=0.
\end{equation}
If we combine the results of Step 2, Step 3 and Step 5, it follows that:
\begin{equation}
\notag
\big[\Gamma\,\big]^{\omega^{*}}(t)-\mathcal{U}(t)\,\Gamma(0)+i \int_{0}^{t}\,\mathcal{U}(t-s)\,\big[\widehat{B}\,\big]^{\omega^{*}} \big[\Gamma\,\big]^{\omega^{*}}(s)\,ds\, \in C_{t \in I} L^2(\Omega^{*}) \mathcal{H}^{\alpha}_{\xi_0}.
\end{equation}
We may hence conclude that, \emph{for all $t \in I$}:
\begin{equation}
\notag
\big\|\big[\Gamma\,\big]^{\omega^{*}}(t)-\mathcal{U}(t)\,\Gamma(0)+i \int_{0}^{t}\,\mathcal{U}(t-s)\,\big[\widehat{B}\,\big]^{\omega^{*}} \big[\Gamma\,\big]^{\omega^{*}}(s)\,ds\,\big\|_{L^2(\Omega^{*}) \mathcal{H}^{\alpha}_{\xi_0}}=0.
\end{equation}
Theorem \ref{Theorem2} now follows.
\end{proof}

\section{The dependently randomized Gross-Pitaevskii hierarchy}
\label{The dependently randomized Gross-Pitaevskii hierarchy}
In this section, we will apply the above techniques in the context of the \emph{dependently randomized Gross-Pitaevskii hierarchy}:

\begin{equation}
\label{RandomizedGP2}
\begin{cases}
i \partial_t \gamma^{(k)} + (\Delta_{\vec{x}_k}-\Delta_{\vec{x}'_k}) \gamma^{(k)}=\sum_{j=1}^{k}[B_{j,k+1}]^{\omega}(\gamma^{(k+1)})\\
\gamma^{(k)}\big|_{t=0}=\gamma^{(k)}_0.
\end{cases}
\end{equation}
Here, $\omega \in \Omega$ is a fixed element of the probability space.
The spatial domain is again given by $\Lambda=\mathbb{T}^3$, unless it is specified otherwise. 
We will now redo the above analysis in the context of \eqref{RandomizedGP2}. As we will see, a lot of the arguments will be quite similar, but there will be some notable differences. 
As in \cite{SoSt}, a big difference in dependently randomized setting is the fact that one has to work in an appropriate \emph{non-resonant class} of density matrices. The precise definition is given in Definition \ref{Non-resonant} below. As was noted in \cite{SoSt}, in this way, one can obtain good estimates on the associated Duhamel terms. These estimates are summarized in Proposition \ref{non-resonant}. We note that the range for the regularity exponent $\alpha$ has now been changed to $\alpha \geq 0$.

A difficulty which arises in the dependently randomized context is the fact that we can not apply the result of Proposition \ref{RandomizedEstimate2}, which we used in the independently randomized context in order to estimate terms that contain different random parameters. In the dependently randomized context, this type of argument is not possible since there is only one random parameter. We see this difficulty in the proof of Proposition \ref{theta_omega_equation}, which is an analogue of Proposition \ref{theta_omega_star_equation} in the dependently randomized setting. In other words, in Proposition \ref{theta_omega_equation}, we show that the limit $\theta^{\,\omega}$ of $\big(\big[\widehat{B}\,\big]^{\omega} \big[\Gamma_N\big]^{\omega}\big)_N$, obtained from the solutions $\big[\Gamma_N\big]^{\omega}$ of the associated truncated hierarchies, satisfies the equation:
\begin{equation}
\label{theta_omega_equation_Introduction}
\big\|\theta^{\,\omega}-\big[\widehat{B}\,\big]^{\omega}\,\mathcal{U}(t)\,\Gamma(0)+i\int_{0}^{t} \big[\widehat{B}\,\big]^{\omega}\,\mathcal{U}(t-s)\,\theta^{\,\omega}(s)\,ds \,\big\|_{L^{\infty}_{t \in I} L^2(\Omega) \mathcal{H}^{\alpha}_{\xi_0}}=0,
\end{equation}
for all $\xi_0>0$.
The precise definition of $\big[\Gamma_N\big]^{\omega}$ is given in \eqref{GammaNGP2} below. We recall the definition of the operator $\big[\widehat{B}\,\big]^{\omega}$ from \eqref{B_hat_omega}. 

In the proof of \eqref{theta_omega_equation_Introduction}, it is useful to show that:
\begin{equation}
\notag
%\label{theta_omega_equation_Introduction2}
\big\|\int_{0}^{t} \big[\widehat{B}\,\big]^{\omega} \, \mathcal{U}(t-s) \big(\theta^{\,\omega}(s) -\big[\widehat{B}\,\big]^{\omega} \big[\Gamma_N\big]^{\omega}(s) \big) \, ds\big\|_{L^{\infty}_{t \in I} L^2(\Omega) \mathcal{H}^{\alpha}_{\xi_0}}=0,
\end{equation}
for all $\xi_0>0$. This is analogous to the estimate of the third term on the right-hand side of \eqref{theta_omega_star_difference} in the proof of Proposition \ref{theta_omega_star_equation}. In the independently randomized setting, it was possible to prove the claim by applying Proposition \ref{RandomizedEstimate2}, whereas in the dependently randomized setting
we can not apply this method and we have to argue directly by writing out the terms from the initial data. This is done in the third step in the proof of Proposition \ref{theta_omega_equation} below.

Another place where we have to explicitly write out the terms is in estimating the difference of two Duhamel expansions. Let us note that, in the independently randomized setting, the analogous estimate was done in Step 1 of the proof of Theorem \ref{Theorem2}. The estimate on the difference of two Duhamel expansions in the dependently randomized setting is given in Lemma \ref{Duhamel_Expansion_Difference} and it relies on the explicit expansion from \cite{SoSt}, which is recalled in Lemma \ref{Duhamel_Expansion}. A concrete example of the expansion in Lemma \ref{Duhamel_Expansion_Difference} is given in Example \ref{Example} below. We use Lemma \ref{Duhamel_Expansion_Difference} in the proof of the equicontinuity properties of $\big(\big[\widehat{B}\,\big]^{\omega} \big[\Gamma_N\,\big]^{\omega}\big)_N$ in Step 4 of the proof of Theorem \ref{Theorem2GP2}.

All throughout this section, the implied constants $C_0,C_1,$ etc. will in general denote different constants than the ones used in Section \ref{The truncated Gross-Pitaevskii hierarchy corresponding to RandomizedGP1}. For each constant, we will explicitly note on which parameters it depends. 
Let us now define the new notation which we will use in the context of the hierarchy \eqref{RandomizedGP2}.

We first define Duhamel iteration terms for \eqref{RandomizedGP2}. Similarly as in \eqref{Duhamel_0} and \eqref{Duhamel_j}, we define for fixed $\omega \in \Omega$ and $\Gamma(0)=(\gamma_0^{(k)})_k$ the following quantities:

If $j=0$:
\begin{equation}
\label{Duhamel_02}
Duh_{\,0}^{\,\omega}(\Gamma(0))^{(k)}(t):=\mathcal{U}^{(k)}(t)\,\gamma_0^{(k)}.    
\end{equation}

If $j \geq 1$:

\begin{equation}
\label{Duhamel_j2}
Duh_{\,j}^{\,\omega}(\Gamma(0))^{(k)}(t):=
\end{equation}
$$(-i)^j \int_0^t \int_0^{t_1} \cdots \int_0^{t_{j-1}} \,\mathcal{U}^{(k)}(t-t_1) \,[B^{(k+1)}]^{\omega} \, \mathcal{U}^{(k+1)}(t_1-t_2)\,[B^{(k+2)}]^{\omega}\, \cdots
$$ 
\begin{equation}
\notag
\cdots \,\mathcal{U}^{(k+j-1)}(t_{j-1}-t_j)\,[B^{(k+j)}]^{\omega}\,\mathcal{U}^{(k+j)}(t_j)\,\gamma_0^{(k+j)} \, dt_j\,dt_{j-1}\,\cdots\,dt_2\,dt_1.
\end{equation}

The quantity $Duh_{\,j}^{\,\omega}(\Gamma(0))$ corresponds to an iterated Duhamel expansion in the hierarchy \eqref{RandomizedGP2}.
As in Proposition \ref{DuhamelEstimate1}, we would like to obtain an estimate for these Duhamel terms. Due to the dependent randomization, and in light of the counterexamples from Subsection 6.1 of \cite{SoSt}, it is not possible to iterate Theorem \ref{RandomizedEstimate1} as in the context of the independently randomized GP hierarchy \eqref{RandomizedGP1}. Instead, we need to work in a more restrictive class of density matrices. We need the following definition, for fixed $\alpha \geq 0$:

\begin{definition}
\label{Non-resonant}
Suppose that  $\Sigma=(\sigma^{(m)})_m$ is a sequence, where each $\sigma^{(m)}$ is a density matrix of order $m$ on $\mathbb{T}^3$. We say that $\Sigma$ belongs to the non-resonant class $\mathcal{N}$ if:
\begin{itemize}
\item[i)]
for all $m \in \mathbb{N}$, and for all frequencies $(\xi_1,\ldots,\xi_m,\xi'_1,\ldots,\xi'_m) \in (\mathbb{Z}^3)^m \times (\mathbb{Z}^3)^m$, the Fourier coefficient $\widehat{\sigma}^{(m)}(\xi_1,\ldots,\xi_m;\xi'_1,\ldots,\xi'_m)$ equals zero unless
$$|\xi_1|>|\xi_2|>\cdots>|\xi_m|>|\xi'_1|>|\xi'_2|>\cdots > |\xi'_m|.$$
\item[ii)] There exists $C_1>0$ such that:
$$\|S^{(m,\alpha)} \sigma^{(m)}\|_{L^2(\Lambda^m \times \Lambda^m)} \leq C_1^m.$$
\end{itemize}
\end{definition}

\begin{remark}
\label{Frequency order}
We note that, when defining the non-resonant class in Definition \ref{Non-resonant}, it is possible to take any fixed order of the frequencies. In fact, the order can vary with $m$. We take the order $|\xi_1|>|\xi_2|>\cdots>|\xi_m|>|\xi'_1|>|\xi'_2|>\cdots>|\xi'_m|$ for concreteness. For more details on this point, we refer the reader to Subsection 6.3 of \cite{SoSt}.
\end{remark}
\begin{remark}
\label{Non_resonant_Td}
The class $\mathcal{N}$ can be analogously defined for density matrices on $\mathbb{T}^d$. The only change is that, in the general case, the frequencies belong to $\mathbb{Z}^d$, and so $\mathbb{Z}^3$ gets replaced by $\mathbb{Z}^d$ in the definition. For a further discussion about this class on $\mathbb{T}^d$, we refer the reader to Remark \ref{Remark_higher_dimensions2} below.
\end{remark}

\begin{remark}
\label{Non_resonant_class_comparison}
Let us note that $ii)$ in Definition \ref{Non-resonant} follows immediately if we know that $\Sigma \in \mathcal{H}^{\alpha}_{\xi}$, for some $\xi>0$.
\end{remark}

Let us recall the following result, which can be immediately deduced from Proposition 6.2 in \cite{SoSt}:

\begin{proposition}
\label{non-resonant}
Suppose that $\alpha \geq 0$. Then, there exists $C_0>0$, depending only on $\alpha$ such that for all $(\sigma^{(m)})_m \in \mathcal{N}$, for all $k,j \in \mathbb{N},\ell_1,\ldots \ell_{j}, n_1, \ldots, n_{j} \in \mathbb{N}$, with ${\ell}_1<n_1 \leq k+j, \ldots, {\ell}_{j}<n_{j} \leq k+j$ and for all $t_1,t_2, \ldots, t_{j+1} \in \mathbb{R}$, the following bound holds:
$$\big\|S^{(k,\alpha)}\,\mathcal{U}^{(k)}(t_1-t_2)\,[B^{\pm}_{{\ell}_1,n_1}]^{\omega}\,\mathcal{U}^{(k+1)}(t_2-t_3)\,[B^{\pm}_{{\ell}_2,n_2}]^{\omega} \cdots$$
$$\cdots \mathcal{U}^{(k+j-1)}(t_{j}-t_{j+1}) \,[B^{\pm}_{{\ell}_{j},n_{j}}]^{\omega} \,\sigma^{(k+j)}\big\|_{L^2(\Omega \times \Lambda^k \times \Lambda^k)} \leq C_0^{\,k+j} \cdot \|S^{(k+j,\alpha)} \sigma^{(k+j)}\|_{L^2(\Lambda^{k+j} \times \Lambda^{k+j})}.$$
\end{proposition}

%From Proposition \ref{non-resonant}, we can prove an estimate on the Duhamel terms $Duh_{\,j}^{\,\omega}(\Gamma(0))^{(k)}(t)$. In particular, we can argue as in the proof of Proposition \ref{DuhamelEstimate1}. The only difference is that the step in which we iteratively apply Proposition \ref{RandomizedEstimate1} gets replaced with one application of Proposition \ref{non-resonant}. The result that then follows is:

%\begin{proposition}
%\label{DuhamelEstimate2}
%Suppose that $T>0$ and let $I:=[0,T]$. Let us fix $\alpha \geq 0$ and let us fix $\Gamma(0)=(\gamma_0^{(k)})_k$ to be non-resonant, in the sense defined above. There exists $C_1>0$, depending only on $\alpha$ and a universal constant $C_2>0$ such that for all $j,k \in \mathbb{N}$:
%\begin{equation}
%\notag
%\big\|Duh_{\,j}^{\,\omega} (\Gamma(0))^{(k)}(t) \big\|_{L^{\infty}_{t \in I} L^2(\Omega) H^{\alpha}(\Lambda^k \times \Lambda^k)} \leq (C_1T)^j \, \cdot \, C_2^k \, \cdot \, \big\|\gamma_0^{(k+j)}\big\|_{H^{\alpha}(\Lambda^{k+j} \times \Lambda^{k+j})}.
%\end{equation} 
%\end{proposition}

It is possible to prove an analogue of Proposition \ref{DuhamelEstimate1} in the dependently randomized setting.
This can be done by replacing the step in which we iteratively apply Theorem \ref{RandomizedEstimate1} with just one application of Proposition \ref{non-resonant}. In what follows, we will not use this this bound, but we will use a similar argument outlined as above to deduce the following fact:

\begin{proposition} 
\label{DuhamelEstimate2}
Suppose that $\alpha \geq 0$. Then, there exist constants $C_1,C_2>0$, depending on $\alpha$ such that for all  $(\sigma^{(m)})_m \in \mathcal{N}$, $j,k \in \mathbb{N}$, and $t \in I$, the following estimate holds:
\begin{equation}
\notag
\big\|\int_{0}^{t} \int_{0}^{t_1} \cdots \int_{0}^{t_{j-1}}
\big[B^{(k+1)}\big]^{\omega}\,\mathcal{U}^{(k+1)}(t-t_1)\,\big[B^{(k+2)}\big]^{\omega} \, \mathcal{U}^{(k+2)}(t_1-t_2)
\,\big[B^{(k+3)}\big]^{\omega}\,\cdots$$ 
$$\cdots \big[B^{(k+j+1)}\big]^{\omega} \, \mathcal{U}^{(k+j+1)}(t_{j}) \,\sigma^{(k+j+1)}\,dt_j \cdots dt_2 \, dt_1 \big\|_{L^2(\Omega) H^{\alpha}(\Lambda^k \times \Lambda^k)}
\end{equation}
\begin{equation}
\notag
%\label{Termj_bound}
\leq (C_1T)^j \cdot C_2^k \cdot \big\|S^{(k+j+1,\alpha)}\sigma^{(k+j+1)}\big\|_{L^2(\Lambda^{k+j+1} \times \Lambda^{k+j+1})}
\end{equation}
\end{proposition}

\begin{proof}(of Proposition \ref{DuhamelEstimate2}) Suppose that $\alpha$ and $(\sigma_m)_m$ are as above.
We use Proposition \ref{non-resonant}, with $t_1=t_2$, as well as the integral identity given by \eqref{Integral_Identity} in order to deduce that the expression on the left-hand side in the statement is:

\begin{equation}
\label{Proposition4.3_auxiliary_bound1}
\leq \frac{T^j}{j!} \cdot 2^{j+1} \cdot  C_0^{k+j+1} \cdot k \cdot (k+1) \cdots (k+j) \cdot \big\|S^{(k+j+1,\alpha)} \sigma^{(k+j+1)}\big\|_{L^2(\Lambda^{k+j+1} \times \Lambda^{k+j+1})}.
\end{equation}
Here, $C_0$ is the implied constant from Proposition \ref{non-resonant}, which we recall depends only on $\alpha$. The factor of $k \cdot (k+1) \cdots (k+j)$ is obtained when we use 
\eqref{Bomegak+1}, for $\ell=k,\ldots,k+j$ and we write each collision operator $\big[B^{(\ell+1)}\big]^{\omega}$ as the sum of $\ell$ terms $\sum_{r=1}^{\ell} \big[B_{r,\ell+1}\big]^{\omega}$. The factor of $2^{j+1}$ is obtained when we separately analyze each term in the above expansion. In particular, we fix $r_{\ell}$ for $\ell=k,k+1,\ldots,k+j$ such that each $r_{\ell} \in \{1,2,\ldots,\ell\}$. We write each of the obtained collision operators $\big[B_{r_{\ell},\ell+1}\big]^{\omega}$ in this configuration as $\big[B_{r_{\ell},\ell+1}\big]^{\omega}=\big[B_{r_{\ell},\ell+1}^{+}\big]^{\omega} - \big[B_{r_{\ell},\ell+1}^{-}\big]^{\omega}$ according to \eqref{Bjkrandomized}. In this way, we obtain $2^{j+1}$ terms, each of which is of the form that can be estimated by Proposition \ref{non-resonant}.

We can bound the expression in \eqref{Proposition4.3_auxiliary_bound1} by:
\begin{equation}
\label{Proposition4.3_auxiliary_bound}
\leq \frac{T^j}{j!} \cdot  2^{j+1} \cdot C_0^{k+j+1} \cdot k \cdot (k+j)^j \cdot \big\|S^{(k+j+1,\alpha)} \sigma^{(k+j+1)}\big\|_{L^2(\Lambda^{k+j+1} \times \Lambda^{k+j+1})}.
\end{equation}
We again use Stirling's Formula \eqref{Stirling's formula} as in earlier arguments and we deduce that this quantity is:
\begin{equation}
\label{Termj_bound_Proposition4.3}
\leq (C_1T)^j \cdot C_2^k \cdot \big\|S^{(k+j+1,\alpha)}\sigma^{(k+j+1)}\big\|_{L^2(\Lambda^{k+j+1} \times \Lambda^{k+j+1})}
\end{equation}
for some constants $C_1,C_2>0$, which both depend on $\alpha$. We note that the reason why both $C_1$ and $C_2$ depend on $\alpha$ is because $C_0$ depends on $\alpha$ and it appears with a power of $k+j+1$.
\end{proof}

\begin{remark}
\label{Remark_higher_dimensions2}
Arguing similarly as in Remark \ref{Remark_higher_dimensions1}, by
the discussions in Subsection \ref{Setup of the problem} and Subsection \ref{Statement of the problem. Main Results}, we note that the analysis that we will present in this section can be generalized to the case $\Lambda=\mathbb{T}^d$. The condition $\alpha \geq 0$ does not change. The only difference is that we have to define the non-resonant class $\mathcal{N}$ as in Remark \ref{Non_resonant_Td}.
\end{remark}

\subsection{The truncated Gross-Pitaevskii hierarchy corresponding to \eqref{RandomizedGP2}}
\label{The truncated Gross-Pitaevskii hierarchy corresponding to RandomizedGP2}

Similarly as in Subsection \ref{The truncated Gross-Pitaevskii hierarchy corresponding to RandomizedGP1}, we will now consider the \emph{truncated Gross-Pitaevskii hierarchy corresponding to \eqref{RandomizedGP2}}, evolving from \emph{non-resonant initial data}. More precisely, we fix $N \in \mathbb{N}$ and $\omega \in \Omega$ and we look for a sequence $\big(\big[\gamma_N^{(k)}\big]^{\omega}\big)_k$ of density matrices, all of which are parametrized by $\omega \in \Omega$ and which solve, for all $k \in \mathbb{N}$ with $k \leq N$:

\begin{equation}
\label{TruncatedRandomizedGP2}
\begin{cases}
i \partial_t \, \big[\gamma_N^{(k)}\big]^{\omega}+(\Delta_{\vec{x}_k}-\Delta_{\vec{x}'_k}) \, \big[\gamma_N^{(k)}\big]^{\omega}=\sum_{j=1}^{k} [B_{j,k+1}]^{\omega}\big(\big[\gamma_N^{(k+1)}\big]^{\omega}\big)\\
\big[\gamma_N^{(k)}\big]^{\omega} \big|_{t=0} = \gamma_0^{(k)}.
\end{cases}
\end{equation}
Furthermore, we impose the condition that, for $k>N$:
\begin{equation}
\notag
%\label{TruncatedRandomizedGP2assumption}
\big[\gamma_N^{(k)}\big]^{\omega} \equiv 0.
\end{equation}
The initial data is taken to be \emph{non-resonant} in the sense that $\Gamma(0)=(\gamma_0^{(1)},\gamma_0^{(2)},\ldots)=(\gamma_0^{(k)})_{k} \in \mathcal{N}$.  As in \eqref{GammaN0}, we consider:
\begin{equation}
\label{GammaN0GP2}
\Gamma_N(0):=\mathbf{P}_{\leq N} \Gamma(0)=\big(\gamma_0^{(1)},\gamma_0^{(2)}, \ldots, \gamma_0^{(N)},0,0,\ldots\big).
\end{equation}
$\Gamma_N(0)$ then also belongs to the non-resonant class $\mathcal{N}$. Here, we used the definition of the projection operator $\mathbf{P}_{\leq N}$ given in \eqref{P<=N}.
Arguing as for \eqref{gammaNkomega}, we can see that an explicit solution to \eqref{TruncatedRandomizedGP2} is given by:

\begin{equation}
\label{gammaNkomega2}
\big[\gamma_N^{(k)}\big]^{\omega}(t)=\sum_{j=0}^{N-k} Duh_{\,j}^{\,\omega} (\Gamma(0))^{(k)}(t).
\end{equation}
We now define $\big[\Gamma_N\,\big]^{\omega}=\Big(([\Gamma_N]^{\omega})^{(k)}\Big)_k$ by:
\begin{equation}
\label{GammaNGP2}
\big(\big[\Gamma_N\,\big]^{\omega}\big)^{(k)}:=\big[\gamma_N^{(k)}\big]^{\omega}.
\end{equation}

By construction, $\big[\Gamma_N\big]^{\omega}$ solves:

\begin{equation}
\notag
%\label{GammaNomegastar2}
\begin{cases}
i \partial_t \, \big[\Gamma_N \big]^{\omega} + \widehat{\Delta}_{\pm} \, \big[\Gamma_N \big]^{\omega} = \big[\widehat{B} \,\big]^{\omega} \big[\Gamma_N \big]^{\omega}\\
\big[\Gamma_N \big]^{\omega} \big|_{t=0}=\Gamma_N(0).
\end{cases}
\end{equation}
Here, $\Gamma_N(0)$ is defined as in \eqref{GammaN0GP2}.
We would now like to take the limit as $N \rightarrow \infty$ in $\big[\Gamma_N\,\big]^{\omega}$.

Let us fix some more notation. For the rest of this section, unless we specify otherwise, we suppose that $\alpha \geq 0$. As before, we consider a time $T>0$ and the time interval $I:=[0,T]$ and for this $T$, we first choose $\xi'>0$ sufficiently large such that: 
\begin{equation}
\label{xi_prime2}
\frac{C_1T}{\xi'}<1.
\end{equation} 
Having chosen $\xi'$, we choose $\xi \in (0,\xi')$ sufficiently small such that: 
\begin{equation}
\label{xi2}
\frac{C_2 \,\xi}{\xi'}<1.
\end{equation}
$C_1$ and $C_2$ are now the constants from Proposition \ref{DuhamelEstimate2} (and not the constants from Proposition \ref{DuhamelEstimate1} as in the independently randomized setting in Section \ref{The independently randomized Gross-Pitaevskii hierarchy}). Conversely, given $\xi,\xi'>0$ with $\xi \in (0,\xi')$, which satisfy \eqref{xi2}, we can find $T>0$ which satisfies \eqref{xi_prime2}. Unless it is otherwise noted, we will assume that $T,\xi,\xi'$ satisfy the above assumptions.

\subsection{A local-in-time result for non-resonant initial data $\Gamma(0)$ of regularity $\alpha$}
\label{A local-in-time result for non-resonant initial data of regularity alpha}
Throughout this subsection, in addition to the other conventions, we fix the initial data $\Gamma(0)$ to belong to $\mathcal{H}^{\alpha}_{\xi'}$ and to be non-resonant. In other words, 
\begin{equation}
\label{Gamma0_nonresonant}
\Gamma(0) \in \mathcal{H}^{\alpha}_{\xi'} \cap \mathcal{N}.
\end{equation}

Keeping in mind \eqref{xi_prime2}, \eqref{xi2} and \eqref{Gamma0_nonresonant}, we now prove the following analogue of Proposition \ref{CauchySequence1}:

\begin{proposition}
\label{CauchySequence1GP2}
The sequence $\big(\big[\widehat{B}\,\big]^{\omega} \big[\Gamma_N\big]^{\omega}\big)_N$ is Cauchy in $L^{\infty}_{t \in I} L^2(\Omega) \mathcal{H}^{\alpha}_{\xi}$. 
\end{proposition}

\begin{proof} Let us fix $t \in I, k \in \mathbb{N}$ and $N_1,N_2 \in \mathbb{N}$, with $N_1<N_2$. Similarly as in \eqref{BhatGammaNdifference}, we can deduce that the following identity holds:

\begin{equation}
\notag
\Big(\big[\widehat{B}\,\big]^{\omega} \big(\big[\Gamma_{N_1}\big]^{\omega}-\big[\Gamma_{N_2}\big]^{\omega}\big)\Big)^{(k)}(t)=
\end{equation}
\begin{equation}
\notag
=-\sum_{j=N_1-k}^{N_2-k-1}\big[B^{(k+1)}\big]^{\omega} Duh_{\,j}^{\,\omega}(\Gamma_{N_2}(0))^{(k+1)}(t).
\end{equation}
Hence, given $j \in \{N_1-k,N_1-k+1,\ldots,N_2-k-1\}$, we need to estimate:
\begin{equation}
\notag
%\label{BhatGammaNdifference_Termj}
\big\|\big[B^{(k+1)}\big]^{\omega} Duh_{\,j}^{\,\omega}(\Gamma_{N_2}(0))^{(k+1)}(t)\big\|_{L^{\infty}_{t \in I} L^2(\Omega) H^{\alpha}(\Lambda^k \times \Lambda^k)}=
\end{equation}
\begin{equation}
\notag
=\big\|\int_{0}^{t} \int_{0}^{t_1} \cdots \int_{0}^{t_{j-1}}
\big[B^{(k+1)}\big]^{\omega}\,\mathcal{U}^{(k+1)}(t-t_1)\,\big[B^{(k+2)}\big]^{\omega} \, \mathcal{U}^{(k+2)}(t_1-t_2)
\,\big[B^{(k+3)}\big]^{\omega}\,\cdots$$ 
$$\cdots \big[B^{(k+j+1)}\big]^{\omega} \, \mathcal{U}^{(k+j+1)}(t_{j}) \,\gamma_0^{(k+j+1)}\,dt_j \cdots dt_2 \, dt_1 \big\|_{L^{\infty}_{t \in I} L^2(\Omega) H^{\alpha}(\Lambda^k \times \Lambda^k)},
\end{equation}
which by Proposition \ref{DuhamelEstimate2} is:
\begin{equation}
\notag
%\label{Termj_bound}
\leq (C_1T)^j \cdot C_2^k \cdot \big\|S^{(k+j+1,\alpha)}\gamma_0^{(k+j+1)}\big\|_{L^2(\Lambda^{k+j+1} \times \Lambda^{k+j+1})}.
\end{equation}
We note that here, the constants $C_1$ and $C_2$ are as in Proposition \ref{DuhamelEstimate2}. 

As a result, we may deduce that:
$$\sum_{k=1}^{\infty} \xi^k \cdot \big\| \big(\big[\widehat{B}\,\big]^{\omega} \big(\big[\Gamma_{N_1}\big]^{\omega}-\big[\Gamma_{N_2}\big]^{\omega}\big)\big)^{(k)}(t)\big\|_{L^2(\Omega)H^{\alpha}(\Lambda^k \times \Lambda^k)}$$
$$\leq \frac{1}{\xi'} \cdot \sum_{k=1}^{\infty} \sum_{j=N_1-k}^{N_2-k-1} \Big(\frac{C_1T}{\xi'}\Big)^j \cdot \Big(\frac{C_2 \, \xi}{\xi'}\Big)^k \cdot (\xi')^{k+j+1} \cdot \big\|\gamma_0^{(k+j+1)}\big\|_{H^{\alpha}(\Lambda^{k+j+1} \times \Lambda^{(k+j+1)})}.
$$
Since in the above sum, it is the case that $j+k+1>N_1$, it follows that:
$$\big\|\big[\widehat{B}\,\big]^{\omega} \big(\big[\Gamma_{N_1}\big]^{\omega}-\big[\Gamma_{N_2}\big]^{\omega}\big)(t)\big\|_{L^2(\Omega) \mathcal{H}^{\alpha}_{\xi}}$$
$$\leq \frac{1}{\xi'} \cdot \sum_{j=N_1-k}^{N_2-k-1} \Big(\frac{C_1T}{\xi'}\Big)^j \cdot \sum_{k=1}^{\infty} \Big(\frac{C_2 \, \xi}{\xi'}\Big)^k \cdot \big\|\mathbf{P}_{>N_1} \Gamma(0)\big\|_{\mathcal{H}^{\alpha}_{\xi'}},$$
which by \eqref{xi_prime2} and \eqref{xi2} is:
\begin{equation}
\label{CauchySequence1GP2_bound}
\lesssim_{\,T,\,\xi,\,\xi',\,\alpha}  \big\|\mathbf{P}_{>N_1} \Gamma(0)\big\|_{\mathcal{H}^{\alpha}_{\xi'}}.
\end{equation}
This quantity converges to zero as $N_1 \rightarrow \infty$ since $\Gamma(0) \in \mathcal{H}^{\alpha}_{\xi'}$. The claim now follows.
\end{proof}

\begin{remark}
\label{BN1GP2}
The same methods used to prove \eqref{CauchySequence1GP2_bound} allow us to deduce:
\begin{equation}
\label{BN1GP2bound}
\big\|\big[\widehat{B}\,\big]^{\omega} \big[\Gamma_{N_1}\big]^{\omega}(t)\big\|_{L^{\infty}_{t \in I}L^2(\Omega) \mathcal{H}^{\alpha}_{\xi}} \lesssim_{\,T,\,\xi,\,\xi',\,\alpha} \big\|\Gamma(0)\big\|_{\mathcal{H}^{\alpha}_{\xi'}}
\end{equation}
uniformly in $N_1$.
\end{remark}

From Proposition \ref{CauchySequence1GP2}, we can now deduce the following result:

\begin{proposition}
\label{limit2GP2}
There exists $\theta^{\,\omega} \in L^{\infty}_{t \in I} L^2(\Omega) \mathcal{H}^{\alpha}_{\xi}$ such that $\big[\widehat{B}\,\big]^{\omega} \big[\Gamma_N \big]^{\omega} \rightarrow \theta^{\,\omega}$ strongly in $L^{\infty}_{t \in I} L^2(\Omega) \mathcal{H}^{\alpha}_{\xi}$ as $N \rightarrow \infty$. 
\end{proposition}

\begin{proof}
This result immediately follows from Proposition \ref{CauchySequence1GP2} and from the fact that $L^{\infty}_{t \in I} L^2(\Omega) \mathcal{H}^{\alpha}_{\xi}$ is a Banach space.
\end{proof}

\begin{remark}
\label{theta_omega_boundA}
If we use \eqref{BN1GP2bound} from Remark \ref{BN1GP2} and Proposition \ref{limit2GP2}, and if we let $N_1 \rightarrow \infty$, it follows that:
\begin{equation}
\label{theta_omega_bound1}
\big\|\theta^{\,\omega}\big\|_{L^{\infty}_{t \in I} L^2(\Omega) \mathcal{H}^{\alpha}_{\xi}} \lesssim_{\,T,\,\xi,\,\xi',\,\alpha} \big\|\Gamma(0)\big\|_{\mathcal{H}^{\alpha}_{\xi'}}.
\end{equation}
\end{remark}

Before we prove an analogue of Proposition \ref{theta_omega_star_equation} for the dependently randomized GP hierarchy \eqref{RandomizedGP2}, we need to first prove the following lemma:

\begin{lemma}
\label{BhatSigma}
Suppose that $\Sigma^{\,\omega}=(\sigma^{(k)})_k$ is a sequence of density matrices depending on time and on the random parameter $\omega \in \Omega$ such that:
$$\big\|\Sigma^{\,\omega}\big\|_{L^{\infty}_{t \in I} L^2(\Omega) \mathcal{H}^{\alpha}_{\xi}}=0.$$ 
Then, it is the case that: 
\begin{equation}
\label{BhatSigmabound}
\big\|\big[\widehat{B}\,\big]^{\omega} \Sigma^{\,\omega} \big\|_{L^{\infty}_{t \in I} L^2(\Omega) \mathcal{H}^{\alpha}_{\xi}}=0
\end{equation}
and
\begin{equation}
\label{mathcalUSigma}
\big\|\,\mathcal{U}(s)\, \Sigma^{\,\omega}(t) \big\|_{L^{\infty}_{s \in I} L^{\infty}_{t \in I} L^2(\Omega) \mathcal{H}^{\alpha}_{\xi}}=0.
\end{equation}
\end{lemma}

We note that in the conclusion \eqref{BhatSigmabound} of the lemma, the parameter $\omega$ in $\big[\widehat{B}\,\big]^{\omega}$ and in $\Sigma^{\,\omega}$ are the same. Each term in $\Sigma$ is a density matrix $\sigma^{(k)}=\sigma^{(k)}(t,\omega,\vec{x}_k;\vec{x}'_k)$ which has order $k$. We suppress the $t$ and $\omega$ dependence for simplicity of notation.

Let us now prove the lemma:

\begin{proof}

Let us first prove \eqref{BhatSigmabound}.

Since $\big\|\Sigma^{\,\omega}\big\|_{L^{\infty}_{t \in I} L^2(\Omega) \mathcal{H}^{\alpha}_{\xi}}=0$, it follows that:
$$\big\|\Sigma^{\,\omega}\big\|_{L^{\infty}_{t \in I} L^{\infty}(\Omega) \mathcal{H}^{\alpha}_{\xi}}=0.$$
In particular, we note that, for all $k \in \mathbb{N}$:
$$\big\|S^{(k,\alpha)} \sigma^{(k)}(t,\omega,\vec{x}_k;\vec{x}'_k)\big\|_{L^{\infty}_{t \in I} L^{\infty}(\Omega) L^2(\Lambda^k \times \Lambda^k)}=0.$$
On the Fourier transform side, we hence obtain, for all $k \in \mathbb{N}$:
$$\big\|\langle \vec{\xi}_k \rangle^{\alpha} \cdot \langle \vec{\xi}'_k \rangle^{\alpha} \cdot \widehat{\sigma}^{(k)}(t,\omega,\vec{\xi}_k; \vec{\xi}'_k) \big\|_{L^{\infty}_{t \in I} L^{\infty}(\Omega) L^2(\mathbb{Z}^{3k} \times \mathbb{Z}^{3k})}=0.$$
The latter identity is equivalent to:
$$\big\|\widehat{\sigma}^{(k)}(t,\omega,\vec{\xi}_k; \vec{\xi}'_k) \big\|_{L^{\infty}_{t \in I} L^{\infty}(\Omega) L^2(\mathbb{Z}^{3k} \times \mathbb{Z}^{3k})}=0.$$
Let $\nu$ denote the measure on $I \times \Omega$ obtained by taking the product of the Lebesgue measure on $I$ and the probability measure $p$ on $\Omega$.
It follows that, for all $k \in \mathbb{N}$, and for all $(\vec{\xi}_k; \vec{\xi}'_k) \in \mathbb{Z}^{3k} \times \mathbb{Z}^{3k}$, there exists a set $\mathcal{E}_{\vec{\xi}_k;\vec{\xi}'_k}\subseteq I \times \Omega$ such that:
\begin{itemize}
\item[i)] $\nu\big( (I \times \Omega) \setminus \mathcal{E}_{\vec{\xi}_k;\vec{\xi}'_k}\big)=0.$
\item[ii)] $\widehat{\sigma}^{(k)}(t,\omega,\vec{\xi}_k;\vec{\xi}'_k)=0$ for all $(t,\omega) \in \mathcal{E}_{\vec{\xi}_k;\vec{\xi}'_k}$, 
\end{itemize}
Let us define:
\begin{equation}
\label{mathcalE}
\mathcal{E}:=\bigcup_{k \in \mathbb{N}} \bigcup_{(\vec{\xi}_k;\vec{\xi}'_k) \in \mathbb{Z}^{3k} \times \mathbb{Z}^{3k}} \mathcal{E}_{\vec{\xi}_k;\vec{\xi}'_k}.
\end{equation}
Then $\mathcal{E} \subseteq I \times \Omega$ is $\nu-$measurable and $\nu\big((I \times \Omega) \setminus \mathcal{E}\big)=0$.

We will now show that, for all $(t,\omega) \in \mathcal{E}$, it is the case that for all $k \in \mathbb{N}$ and $(\vec{\xi}_k,\vec{\xi}'_k) \in \mathbb{Z}^{3k} \times \mathbb{Z}^{3k}$:
\begin{equation}
\label{Bsigma}
\big(\big[B^{(k+1)}\big]^{\omega}\sigma^{(k+1)}\big)\,\widehat{\,}\,(t,\omega,\vec{\xi}_k;\vec{\xi}'_k)=0.
\end{equation} 
Let us note that \eqref{Bsigma} implies the claim given in \eqref{BhatSigmabound}. Namely, from \eqref{Bsigma}, we can use Plancherel's Theorem and deduce that for all $(t,\omega) \in \mathcal{E}$:
$$\big\|S^{(k,\alpha)}\big[B^{(k+1)}\big]^{\omega} \sigma^{(k+1)}(t,\omega)\big\|_{L^2(\Lambda^k \times \Lambda^k)}=0.$$ 
Since this identity holds for all $k \in \mathbb{N}$ and since $\nu \big((I \times \Omega) \setminus \mathcal{E} \big)=0$, it follows that:
$$\big\|\big[\widehat{B}\,\big]^{\omega} \Sigma^{\,\omega} \big\|_{L^{\infty}_{t \in I} L^{\infty}(\Omega) \mathcal{H}^{\alpha}_{\xi}}=0.$$
Consequently:
$$\big\|\big[\widehat{B}\,\big]^{\omega} \Sigma^{\,\omega}\big\|_{L^{\infty}_{t \in I} L^2(\Omega) \mathcal{H}^{\alpha}_{\xi}}=0,$$
as was claimed in \eqref{BhatSigmabound}.

We now need to prove \eqref{Bsigma}. Let us fix $(t,\omega) \in \mathcal{E}$ and $(\vec{\xi}_k;\vec{\xi}'_k) \in \mathbb{Z}^{3k} \times \mathbb{Z}^{3k}$.
We observe that:
$$\big(\big[B^{+}_{1,k+1}\big]^{\omega} \sigma^{(k+1)}\big)\,\widehat{\,}\,(t,\omega,\vec{\xi}_k;\vec{\xi}'_k)=$$
$$\sum_{\mu,\mu' \in \mathbb{Z}^3} \widehat{\sigma}^{(k+1)}(t,\omega,\xi_1-\mu+\mu',\xi_2,\ldots,\xi_k,\mu;\vec{\xi}'_k,\mu') \cdot h_{\xi_1}(\omega) \cdot h_{\xi_1-\mu+\mu'}(\omega) \cdot h_{\mu}(\omega) \cdot h_{\mu'}(\omega).$$
By construction, each summand equals to zero so the whole sum equals to zero.
The proof that 
$$\big(\big[B^{\pm}_{j,k+1}\big]^{\omega} \sigma^{(k+1)}\big)\,\widehat{\,}\,(t,\omega,\vec{\xi}_k;\vec{\xi}'_k)=0$$ for all $j \in \{1,2,\ldots,k\}$ is analogous. The identity \eqref{Bsigma} now follows.

We now prove \eqref{mathcalUSigma} by using a similar argument. Namely, let us take $(t,\omega) \in \mathcal{E}$, for $\mathcal{E}$ defined in \eqref{mathcalE}. We observe that  for all $k \in \mathbb{N},(\vec{\xi}_k,\vec{\xi}'_k) \in \mathbb{Z}^{3k} \times \mathbb{Z}^{3k}$ and $\tau \in I$:
$$\big(\,\mathcal{U}^{(k)}(\tau)\,\sigma^{(k)}\big)\,\widehat{\,}\,(t,\omega,\vec{\xi}_k;\vec{\xi}'_k)= e^{-i\tau(|\vec{\xi}_k|^2-|\vec{\xi}'_k|^2)} \cdot \widehat{\sigma}^{(k)}(t,\omega,\vec{\xi}_k;\vec{\xi}'_k)=0,$$
by definition of $\mathcal{E}$. The claim \eqref{mathcalUSigma} now follows by arguing as in the proof of \eqref{BhatSigmabound}.

\end{proof}
\begin{remark}
We note from the proof that the parameter $\xi$ in the statement of Lemma \ref{BhatSigma} could have been replaced by any other $\xi_0>0$.
\end{remark}

We can now prove an analogue of Proposition \ref{theta_omega_star_equation} for the hierarchy \eqref{RandomizedGP2}.

\begin{proposition}
\label{theta_omega_equation}
For $\theta^{\,\omega}$, constructed in Proposition \ref{limit2GP2}, and for all $\xi_0>0$, the following identity holds:
$$\big\|\theta^{\,\omega}-\big[\widehat{B}\,\big]^{\omega}\,\mathcal{U}(t)\,\Gamma(0)+i\int_{0}^{t} \big[\widehat{B}\,\big]^{\omega}\,\mathcal{U}(t-s)\,\theta^{\,\omega}(s)\,ds \,\big\|_{L^{\infty}_{t \in I} L^2(\Omega) \mathcal{H}^{\alpha}_{\xi_0}}=0.$$
\end{proposition}

\begin{proof}
It suffices to show the claim for $\xi_0=\xi$, where $\xi$ is as in \eqref{xi2}. Let us henceforth set $\xi_0=\xi$.
We note that, for all $N \in \mathbb{N}$:
\begin{equation}
\label{theta_omega_differenceGP2}
\big\|\theta^{\,\omega}-\big[\widehat{B}\,\big]^{\omega}\,\mathcal{U}(t)\,\Gamma(0)+i\int_{0}^{t} \big[\widehat{B}\,\big]^{\omega}\,\mathcal{U}(t-s)\,\theta^{\,\omega}(s)\,ds \,\big\|_{L^{\infty}_{t \in I} L^2(\Omega) \mathcal{H}^{\alpha}_{\xi}}
\end{equation}
$$\leq \big\|\theta^{\,\omega}-\big[\widehat{B}\,\big]^{\omega} \big[\Gamma_N\big]^{\omega}\big\|_{L^{\infty}_{t \in I} L^2(\Omega) \mathcal{H}^{\alpha}_{\xi}} + \big\|\big[\widehat{B}\,\big]^{\omega}\,\mathcal{U}(t)\,\big(\Gamma(0)-\Gamma_N(0)\big)\big\|_{L^{\infty}_{t \in I} L^2(\Omega) \mathcal{H}^{\alpha}_{\xi}} + $$
$$+ \big\|\int_{0}^{t} \big[\widehat{B}\,\big]^{\omega} \, \mathcal{U}(t-s) \big(\theta^{\,\omega}(s) -\big[\widehat{B}\,\big]^{\omega} \big[\Gamma_N\big]^{\omega}(s) \big) \, ds\big\|_{L^{\infty}_{t \in I} L^2(\Omega) \mathcal{H}^{\alpha}_{\xi}} +$$
$$+ \big\|\big[\widehat{B}\,\big]^{\omega} \big[\Gamma_N\big]^{\omega}(t)-\big[\widehat{B}\,\big]^{\omega}\,\mathcal{U}(t) \, \Gamma_N(0) \,+\, i \int_{0}^{t} \big[\widehat{B}\,\big]^{\omega} \, \mathcal{U}(t-s)\, \big[\widehat{B}\,\big]^{\omega} \big[\Gamma_N\big]^{\omega}(s)\, ds\big\|_{L^{\infty}_{t \in I} L^2(\Omega) \mathcal{H}^{\alpha}_{\xi}}.$$
We will show that each term converges to zero as $N \rightarrow \infty.$

\textbf{1)}
The first term converges to zero as $N \rightarrow \infty$ by using Proposition \ref{limit2GP2}.

\textbf{2)}
For the second term, we observe that, for all $k \in \mathbb{N}$:

$$\Big(\big[\widehat{B}\,\big]^{\omega}\,\mathcal{U}(t)\,\big(\Gamma(0)-\Gamma_N(0)\big)\Big)^{(k)}= \big[B^{(k+1)}\big]^{\omega} \,\mathcal{U}^{(k+1)}(t)\,\big(\Gamma(0)-\Gamma_N(0)\big)^{(k+1)}=$$
\begin{equation}
\notag
=
\begin{cases}
\big[B^{(k+1)}\big]^{\omega} \, \mathcal{U}^{(k+1)}(t)\,\gamma_0^{(k+1)}\,\,\mbox{if}\,\,k \geq N\\
0\,\,\mbox{otherwise}.
\end{cases}
\end{equation}
Consequently:
$$\big\|\big[\widehat{B}\,\big]^{\omega} \,\mathcal{U}(t)\,\big(\Gamma(0)-\Gamma_N(0)\big)\big\|_{L^{\infty}_{t \in I} L^2(\Omega) \mathcal{H}^{\alpha}_{\xi}} = \sum_{k \geq N} \xi^k \, \cdot \, \big\|\big[B^{(k+1)}\big]^{\omega}\,\mathcal{U}^{(k+1)}(t)\,\gamma_0^{(k+1)}\big\|_{L^2(\Omega) H^{\alpha}(\Lambda^k \times \Lambda^k)}.$$
By Theorem \ref{RandomizedEstimate1}, this expression is:
$$\lesssim_{\,\alpha} \sum_{k \geq N} k \cdot \xi^k \cdot \big\|S^{(k+1,\alpha)} \gamma_0^{(k+1)}\big\|_{L^2(\Lambda^{k+1} \times \Lambda^{k+1})} \lesssim_{\,\xi,\,\xi'} \big\|\mathbf{P}_{>N}\Gamma(0)\big\|_{\mathcal{H}^{\alpha}_{\xi'}}.$$
In the last inequality, we used the assumption that $\xi<\xi'$.
Since $\Gamma(0) \in \mathcal{H}^{\alpha}_{\xi'}$, it follows that $\big\|\mathbf{P}_{>N}\Gamma(0)\big\|_{\mathcal{H}^{\alpha}_{\xi'}} \rightarrow 0$ as $N \rightarrow \infty$. In particular, the second term on the right-hand side of \eqref{theta_omega_differenceGP2} converges to zero as $N \rightarrow \infty$.

\textbf{3)}
For the third term, let us first observe that, by Minkowski's inequality:

$$\big\|\int_{0}^{t} \big[\widehat{B}\,\big]^{\omega} \, \mathcal{U}(t-s) \big(\theta^{\,\omega}(s) -\big[\widehat{B}\,\big]^{\omega} \big[\Gamma_N\big]^{\omega}(s) \big) \, ds\big\|_{L^{\infty}_{t \in I} L^2(\Omega) \mathcal{H}^{\alpha}_{\xi}}$$
$$\leq \int_{0}^{t} \big\|\big[\widehat{B}\,\big]^{\omega} \, \mathcal{U}(t-s) \big(\theta^{\,\omega}(s) -\big[\widehat{B}\,\big]^{\omega} \big[\Gamma_N\big]^{\omega}(s) \big)\big\|_{L^{\infty}_{t \in I} L^2(\Omega) \mathcal{H}^{\alpha}_{\xi}} \, ds$$
$$ \leq T \cdot \big\|\big[\widehat{B}\,\big]^{\omega} \, \mathcal{U}(t-s) \big(\theta^{\,\omega}(s) -\big[\widehat{B}\,\big]^{\omega} \big[\Gamma_N\big]^{\omega}(s) \big)\big\|_{L^{\infty}_{s \in I} L^{\infty}_{t \in I} L^2(\Omega) \mathcal{H}^{\alpha}_{\xi}}.$$
Hence, we need to show:

\begin{equation}
\label{Third_term_boundGP2}
\big\|\big[\widehat{B}\,\big]^{\omega} \, \mathcal{U}(t-s) \big(\theta^{\,\omega}(s) -\big[\widehat{B}\,\big]^{\omega} \big[\Gamma_N\big]^{\omega}(s) \big)\big\|_{L^{\infty}_{s \in I} L^{\infty}_{t \in I} L^2(\Omega) \mathcal{H}^{\alpha}_{\xi}} \rightarrow 0,
\end{equation}
as $N \rightarrow \infty$.

The difficulty in showing \eqref{Third_term_boundGP2} lies in the fact that there is $\omega$ dependence in $\big[\widehat{B}\,\big]^{\omega}$ and in $\big(\theta^{\,\omega} -\big[\widehat{B}\,\big]^{\omega} \big[\Gamma_N\big]^{\omega}\big)$. We will have to expand out all of the terms. It will also be important to use the fact that we know explicitly how the operators $\big[\widehat{B}\,\big]^{\omega}$ and $\mathcal{U}(t-s)$ act on the Fourier domain.

Let us first note that, formally, by \eqref{gammaNkomega2} and by Proposition \ref{limit2GP2}, for all $k \in \mathbb{N}$ with $k \leq N -1$:
\begin{equation}
\label{identity}
\Big(\theta^{\,\omega}(t) -\big[\widehat{B}\,\big]^{\omega} \big[\Gamma_N\big]^{\omega}(t)\Big)^{(k+1)}=
\end{equation}
$$\sum_{j=N-k-1}^{\infty} (-1)^j \int_{0}^{t} \int_{0}^{t_1} \cdots \int_{0}^{t_{j-1}}
\big[B^{(k+2)}\big]^{\omega} \, \mathcal{U}^{(k+2)}(t-t_1) \, \big[B^{(k+3)}\big]^{\omega} \, \mathcal{U}^{(k+3)}(t_1-t_2)\,
\,\cdots \,$$ 
$$\cdots\, \mathcal{U}^{(k+j+1)}(t_{j-1}-t_j) \big[B^{(k+j+2)}\big]^{\omega} \, \mathcal{U}^{(k+j+2)}(t_{j}) \,\gamma_0^{(k+j+2)}\,dt_j \cdots dt_2 \, dt_1.$$
In the above sum, if $j=0$, the summand is taken to be $\big[B_{k+2}\big]^{\omega}\,\mathcal{U}^{(k+2)}(t)\, \gamma_0^{(k+2)}$, without any integrals in time. We use this convention in the discussion that follows.

The identity \eqref{identity} is made rigorous as follows. We fix $M > N$. Then, we know that, for all $k \in \mathbb{N}$ with $k \leq N-1$:

\begin{equation}
\label{M_N_identity}
\Big(\big[\widehat{B}\,\big]^{\omega} \big[\Gamma_M\big]^{\omega}(t)-\big[\widehat{B}\,\big]^{\omega} \big[\Gamma_N\big]^{\omega}(t)\Big)^{(k+1)}=
\end{equation}
$$\sum_{j=N-k-1}^{M-k-2} (-1)^j \int_{0}^{t} \int_{0}^{t_1} \cdots \int_{0}^{t_{j-1}}
\big[B^{(k+2)}\big]^{\omega} \, \mathcal{U}^{(k+2)}(t-t_1) \, \big[B^{(k+3)}\big]^{\omega} \, \mathcal{U}^{(k+3)}(t_1-t_2)\,
\,\cdots \,$$ 
$$\cdots\, \mathcal{U}^{(k+j+1)}(t_{j-1}-t_j) \big[B^{(k+j+2)}\big]^{\omega} \, \mathcal{U}^{(k+j+2)}(t_{j}) \,\gamma_0^{(k+j+2)}\,dt_j \cdots dt_2 \, dt_1.$$
This is a well-defined finite sum, which we obtain as a difference of two well-defined finite sums.

We define, for $M>N$:
\begin{equation}
\label{Phi_N^M}
\big[\Phi_{N}^{M}\big]^{\omega}:=\big[\widehat{B}\,\big]^{\omega} \big[\Gamma_M\big]^{\omega}-\big[\widehat{B}\,\big]^{\omega} \big[\Gamma_N\big]^{\omega}
\end{equation}
and:
\begin{equation}
\label{Phi_N}
\big[\Phi_{N}\big]^{\omega}:=\theta^{\,\omega}-\big[\widehat{B}\,\big]^{\omega} \big[\Gamma_N\big]^{\omega}.
\end{equation}
By Proposition \ref{limit2GP2}, we know that $\big[\Phi_{N}^{M}\big]^{\omega} \rightarrow \big[\Phi_{N}\big]^{\omega}$ in $L^{\infty}_{t \in I} L^2(\Omega) \mathcal{H}^{\alpha}_{\xi}$, as $M \rightarrow \infty$.
If we take $k \leq N-1$, then $\big(\big[\Phi_{N}^{M}\big]^{\omega}\big)^{(k+1)}$ in particular equals to the right-hand side of \eqref{M_N_identity}. This is an element of $L^{\infty}_{t \in I} L^2(\Omega) H^{\alpha}(\Lambda^{k+1} \times \Lambda^{k+1}).$

Let us furthermore note that the right-hand side of \eqref{M_N_identity} converges to the right-hand side of \eqref{identity} in $L^{\infty}_{t \in I} L^2(\Omega) H^{\alpha} (\Lambda^{k+1} \times \Lambda^{k+1})$ as $M \rightarrow \infty$. 
Moreover, the series given on the right-hand side of \eqref{identity} is well-defined and it converges absolutely with respect to $\|\cdot\|_{L^{\infty}_{t \in I} L^2(\Omega) H^{\alpha}(\Lambda^{k+1} \times \Lambda^{k+1})}$. More precisely, for $k \leq N-1,$ the difference of these two expressions estimated in $\big\|\cdot\big\|_{L^{\infty}_{t \in I} L^2(\Omega) H^{\alpha} (\Lambda^{k+1} \times \Lambda^{k+1})}$ equals to:
$$\big\|\sum_{j=M-k-1}^{\infty} (-1)^j \int_{0}^{t} \int_{0}^{t_1} \cdots \int_{0}^{t_{j-1}}
\big[B^{(k+2)}\big]^{\omega} \, \mathcal{U}^{(k+2)}(t-t_1) \, \big[B^{(k+3)}\big]^{\omega} \, \mathcal{U}^{(k+3)}(t_1-t_2)\,
\,\cdots \,$$ 
$$\cdots\, \mathcal{U}^{(k+j+1)}(t_{j-1}-t_j) \big[B^{(k+j+2)}\big]^{\omega} \, \mathcal{U}^{(k+j+2)}(t_{j}) \,\gamma_0^{(k+j+2)}\,dt_j \cdots dt_2 \, dt_1\big\|_{L^{\infty}_{t \in I} L^2(\Omega)H^{\alpha}(\Lambda^{k+1} \times \Lambda^{k+1})}.$$
By Proposition \ref{DuhamelEstimate2}, this quantity is:
$$\leq \frac{C_2^{k+1}}{(\xi')^{k+2}} \cdot \sum_{j=M-k-1}^{\infty} \Big(\frac{C_1T}{\xi'}\Big)^j \cdot (\xi')^{k+j+2} \cdot \big\|\gamma_0^{(k+j+2)}\big\|_{H^{\alpha}(\Lambda^{k+j+2} \times \Lambda^{(k+j+2)})}
$$
$$\leq \frac{C_2^{k+1}}{(\xi')^{k+2}}  \cdot \big\|\Gamma(0)\big\|_{\mathcal{H}^{\alpha}_{\xi'}} \cdot \sum_{j=M-k-1}^{\infty} \Big(\frac{C_1T}{\xi'}\Big)^j.$$
This quantity converges to zero as $M \rightarrow \infty$ since $\Gamma(0) \in \mathcal{H}^{\alpha}_{\xi'}$ and $\frac{C_1T}{\xi'} \in (0,1)$.

By the convergence result of Proposition \ref{limit2GP2}, we observe that $\big(\big[\Phi_{N}^{M}\big]^{\omega}\big)^{(k+1)} \rightarrow \big(\big[\Phi_{N}\big]^{\omega}\big)^{(k+1)}$ in $L^{\infty}_{t \in I} L^2(\Omega) H^{\alpha}(\Lambda^{k+1} \times \Lambda^{k+1})$ as $M \rightarrow \infty$. The identity \eqref{identity} now follows from the above arguments. Equality here should be taken in the sense of $L^{\infty}_{t \in I} L^2(\Omega) H^{\alpha}(\Lambda^{k+1} \times \Lambda^{k+1})$, i.e. we note that for $k \leq N-1$:

$$\big\|\big(\big[\Phi_{N}\big]^{\omega}\big)^{(k+1)}-\sum_{j=N-k-1}^{\infty} (-1)^j \int_{0}^{t} \int_{0}^{t_1} \cdots \int_{0}^{t_{j-1}}
\big[B^{(k+2)}\big]^{\omega} \, \mathcal{U}^{(k+2)}(t-t_1) \, \,\cdots \,$$ 
$$\cdots\, \big[B^{(k+j+2)}\big]^{\omega} \, \mathcal{U}^{(k+j+2)}(t_{j}) \,\gamma_0^{(k+j+2)}\,dt_j \cdots dt_2 \, dt_1 \big\|_{L^{\infty}_{t \in I} L^2(\Omega) H^{\alpha}(\Lambda^{k+1} \times \Lambda^{k+1})}=0.$$
We observe that, by construction:
$$\big\|\theta^{\,\omega}-\big[\widehat{B}\,\big]^{\omega} \big[\Gamma_N\big]^{\omega}-\big[\Phi_{N}\big]^{\omega}\big\|_{L^{\infty}_{t \in I} L^2(\Omega) \mathcal{H}^{\alpha}_{\xi}}=0.$$
By applying \eqref{mathcalUSigma} and \eqref{BhatSigmabound} from Lemma \ref{BhatSigma}, it follows that:
$$\big\|\big[\widehat{B}\,\big]^{\omega} \,\mathcal{U}(t-s)\, \big(\theta^{\,\omega}-\big[\widehat{B}\,\big]^{\omega} \big[\Gamma_N\big]^{\omega}-\big[\Phi_{N}\big]^{\omega}\big)(s)\big\|_{L^{\infty}_{s \in I} L^{\infty}_{t \in I} L^2(\Omega) \mathcal{H}^{\alpha}_{\xi}}=0$$
Hence, 
$$\big\|\big[\widehat{B}\,\big]^{\omega} \,\mathcal{U}(t-s)\, \big(\theta^{\,\omega}-\big[\widehat{B}\,\big]^{\omega} \big[\Gamma_N\big]^{\omega}\big)(s)\big\|_{L^{\infty}_{s \in I} L^{\infty}_{t \in I} L^2(\Omega) \mathcal{H}^{\alpha}_{\xi}}=$$
$$=\big\|\big[\widehat{B}\,\big]^{\omega} \,\mathcal{U}(t-s)\, \big[\Phi_{N}\big]^{\omega}(s)\big\|_{L^{\infty}_{s \in I} L^{\infty}_{t \in I} L^2(\Omega) \mathcal{H}^{\alpha}_{\xi}}.$$
In particular, \eqref{Third_term_boundGP2} will follow if we prove that:

\begin{equation}
\label{Third_term_boundGP2B}
\big\|\big[\widehat{B}\,\big]^{\omega} \,\mathcal{U}(t-s)\, \big[\Phi_{N}\big]^{\omega}(s)\big\|_{L^{\infty}_{s \in I} L^{\infty}_{t \in I} L^2(\Omega) \mathcal{H}^{\alpha}_{\xi}} \rightarrow 0,
\end{equation}
as $N \rightarrow \infty$.

In order to prove \eqref{Third_term_boundGP2B}, let us first estimate:
$$\big\|\big[\widehat{B}\,\big]^{\omega} \,\mathcal{U}(t-s)\, \big[\Phi_{N}^{M}\big]^{\omega}(s)\big\|_{L^{\infty}_{s \in I} L^{\infty}_{t \in I} L^2(\Omega) \mathcal{H}^{\alpha}_{\xi}}$$
for fixed $M>N$.

For fixed $s,t \in I$, and $k \leq N-1$, we can write:
$$\Big(\big[\widehat{B}\,\big]^{\omega} \, \mathcal{U}(t-s) \, \big[\Phi_{N}^{M}\big]^{\omega}(s)\Big)^{(k)}=$$
$$=\sum_{j=N-k-1}^{M-k-2} (-1)^j \int_{0}^{s} \int_{0}^{t_1} \cdots \int_{0}^{t_{j-1}}
\big[B^{(k+1)}\big]^{\omega}\,\mathcal{U}^{(k+1)}(t-s)\,\big[B^{(k+2)}\big]^{\omega} \, \mathcal{U}^{(k+2)}(s-t_1) \, \big[B^{(k+3)}\big]^{\omega}$$ 
$$\, \mathcal{U}^{(k+3)}(t_1-t_2)\,\cdots\, \mathcal{U}^{(k+j+1)}(t_{j-1}-t_j) \big[B^{(k+j+2)}\big]^{\omega} \, \mathcal{U}^{(k+j+2)}(t_{j}) \,\gamma_0^{(k+j+2)}\,dt_j \cdots dt_2 \, dt_1.$$
Arguing as in the proof of \eqref{Proposition4.3_auxiliary_bound} in Proposition \ref{DuhamelEstimate2}, it follows that:
\begin{equation}
\notag
\Big\|\Big(\big[\widehat{B}\,\big]^{\omega} \, \mathcal{U}(t-s) \, \big[\Phi_{N}^{M}\big]^{\omega}(s) \Big)^{(k)}\Big\|_{L^2(\Omega) H^{\alpha}(\Lambda^k \times \Lambda^k)} \leq
\end{equation} 
$$\leq \sum_{j=N-k-1}^{M-k-2} \frac{T^j}{j!} \cdot 2^{j+2} \cdot C_0^{k+j+2} \cdot k \cdot (k+j+1)^{j+1} \cdot \big\|S^{(k+j+2,\alpha)} \gamma_0^{(k+j+2)}\big\|_{L^2(\Lambda^{k+j+2} \times \Lambda^{k+j+2})}.$$
Here, $C_0$ is the constant from Proposition \ref{non-resonant}.

By applying Stirling's formula \eqref{Stirling's formula} as in \eqref{Termj_bound_Proposition4.3}, this is:
\begin{equation}
\label{k1bound}
\leq \sum_{j=N-k-1}^{M-k-2}  \frac{j+1}{T} \cdot (C_1T)^{j+1} \cdot C_2^k \cdot \big\|S^{(k+j+2,\alpha)}\gamma_0^{(k+j+2)}\big\|_{L^2(\Lambda^{k+j+2} \times \Lambda^{k+j+2})}.
\end{equation}
Here, $C_1$ and $C_2$ are the same constants as in Proposition \ref{DuhamelEstimate2}.

By an analogous argument, for $k \geq N$
\begin{equation}
\notag
\Big\|\Big(\big[\widehat{B}\,\big]^{\omega} \, \mathcal{U}(t-s) \, \big[\Phi_{N}^{M}\big]^{\omega}(s) \Big)^{(k)}\Big\|_{L^2(\Omega) H^{\alpha}(\Lambda^k \times \Lambda^k)} \leq
\end{equation}
\begin{equation}
\label{k2bound}
\leq \sum_{j=0}^{M-k-2}  \frac{j+1}{T} \cdot (C_1T)^{j+1} \cdot C_2^k \cdot \big\|S^{(k+j+2,\alpha)}\gamma_0^{(k+j+2)}\big\|_{L^2(\Lambda^{k+j+2} \times \Lambda^{k+j+2})},
\end{equation}
for the same constants $C_1$ and $C_2$. Here, we used the fact that, for $k \geq N$, it is the case that $\big(\big[\Phi_{N}^{M}\big]^{\omega}\big)^{(k+1)}=\big(\big[\widehat{B}\,\big]^{\omega} \big[\Gamma_M\big]^{\omega}\big)^{(k+1)}$, since $\big(\big[\widehat{B}\,\big]^{\omega} \big[\Gamma_N\big]^{\omega}\big)^{(k+1)}=0$. By our convention for sums from Section \ref{Notation}, the sum in \eqref{k2bound} equals zero if $k \geq M-1$.

From \eqref{k1bound} and \eqref{k2bound}, it follows that:
\begin{equation}
\label{Important_Calculation}
\big\|\big[\widehat{B}\,\big]^{\omega} \, \mathcal{U}(t-s) \big(\, \big[\Phi_{N}^{M}\big]^{\omega}(s) \big)\big\|_{L^{\infty}_{s \in I} L^{\infty}_{t \in I} L^2(\Omega) \mathcal{H}^{\alpha}_{\xi}}
\end{equation}
$$\leq \sum_{k=1}^{N-1}  \sum_{j=N-k-1}^{M-k-2} \xi^k \cdot  \frac{j+1}{T} \cdot (C_1T)^{j+1} \cdot C_2^k \cdot \big\|S^{(k+j+2,\alpha)}\gamma_0^{(k+j+2)}\big\|_{L^2(\Lambda^{k+j+2} \times \Lambda^{k+j+2})}+$$
$$+\sum_{k=N}^{\infty} \sum_{j=0}^{M-k-2}  \xi^k \cdot \frac{j+1}{T} \cdot (C_1T)^{j+1} \cdot C_2^k \cdot \big\|S^{(k+j+2,\alpha)}\gamma_0^{(k+j+2)}\big\|_{L^2(\Lambda^{k+j+2} \times \Lambda^{k+j+2})}$$
$$\leq \frac{C_1}{(\xi')^2} \cdot \sum_{j+k \geq N-1} \Big(\frac{C_2 \xi}{\xi'}\Big)^k \cdot (j+1) \cdot \Big(\frac{C_1 T}{\xi'}\Big)^j \cdot (\xi')^{k+j+2} \cdot \big\|S^{(k+j+2,\alpha)}\gamma_0^{(k+j+2)}\big\|_{L^2(\Lambda^{k+j+2} \times \Lambda^{k+j+2})}$$
$$\leq \frac{C_1}{(\xi')^2} \cdot \big\|\Gamma(0)\big\|_{\mathcal{H}^{\alpha}_{\xi'}} \cdot \sum_{j+k \geq N-1} \Big(\frac{C_2 \xi}{\xi'}\Big)^k \cdot (j+1) \cdot \Big(\frac{C_1 T}{\xi'}\Big)^j$$ 
$$\leq \frac{C_1}{(\xi')^2} \cdot \big\|\Gamma(0)\big\|_{\mathcal{H}^{\alpha}_{\xi'}} \cdot \Big\{\sum_{j\geq \lfloor \frac{N-1}{2} \rfloor; \, k \in \mathbb{N}} \Big(\frac{C_2 \xi}{\xi'}\Big)^k \cdot (j+1) \cdot \Big(\frac{C_1 T}{\xi'}\Big)^j+
\sum_{j \in \mathbb{N}; \, k\geq \lfloor\frac{N-1}{2}\rfloor} \Big(\frac{C_2 \xi}{\xi'}\Big)^k \cdot (j+1) \cdot \Big(\frac{C_1 T}{\xi'}\Big)^j\Big\}.$$
We recall from \eqref{xi_prime2} and \eqref{xi2} that $\frac{C_2 \xi}{\xi'},\frac{C_1 T}{\xi'} \in (0,1)$.
Hence, the above expression is:
$$\leq \frac{C_1}{(\xi')^2} \cdot  \big\|\Gamma(0)\big\|_{\mathcal{H}^{\alpha}_{\xi'}} \cdot \Big\{
\big\{\sup_{j \geq \big\lfloor \frac{N-1}{2} \big\rfloor} (j+1) \cdot \Big(\frac{C_1 T}{\xi'}\Big)^j \big\}  \cdot \sum_{k \in \mathbb{N}} \Big(\frac{C_2 \xi}{\xi'}\Big)^k + \Big(\frac{C_2 \xi}{\xi'}\Big)^{\big\lfloor \frac{N-1}{2} \big\rfloor} \cdot
\sum_{j \in \mathbb{N}} \Big(\frac{C_1 T}{\xi'}\Big)^j\Big\}.$$
This quantity converges to zero as $N \rightarrow \infty$. Let us note that the convergence is uniform in $M$.

We note that, by construction $\big[\theta_{N}\big]^{\omega}$ is obtained as the limit of $\big[\theta_{N}^{M}\big]^{\omega}$ as $M \rightarrow \infty$ in $L^{\infty}_{t \in I} L^2(\Omega) \mathcal{H}^{\alpha}_{\xi}$. The order $k$ limiting object $\big(\big[\theta_{N}\big]^{\omega}\big)^{(k)}$ is then given as an infinite sum. We can formally obtain each component of $\big[\widehat{B}\,\big]^{\omega} \, \mathcal{U}(t-s) \big(\, \big[\Phi_{N}\big]^{\omega}(s) \big)$ by working on the Fourier domain and using the explicit definition of $\big[\widehat{B}\,\big]^{\omega}$ and $\mathcal{U}(t-s)$. In this way, we do not have to appeal to any more general continuity properties of the operators that we are using.  

This step is made rigorous by noting that, from the calculation given in \eqref{Important_Calculation},  the infinite sum which defines:

$$\big[\widehat{B}\,\big]^{\omega} \, \mathcal{U}(t-s) \big(\, \big[\Phi_{N}\big]^{\omega}(s) \big)$$
converges absolutely with respect to $\big\|\cdot\big\|_{L^{\infty}_{s \in I} L^{\infty}_{t \in I} L^2(\Omega) \mathcal{H}^{\alpha}_{\xi}}$. Hence, the construction outlined above is well-defined in this space if we replace $\big[\Phi_{N}^{M}\big]^{\omega}$ by $\big[\Phi_{N}\big]^{\omega}$, i.e. if we formally let $M=\infty$. Moreover,
$$\big\|\big[\widehat{B}\,\big]^{\omega} \, \mathcal{U}(t-s) \big(\, \big[\Phi_{N}\big]^{\omega}(s) \big)\big\|_{L^{\infty}_{s \in I} L^{\infty}_{t \in I} L^2(\Omega) \mathcal{H}^{\alpha}_{\xi}}$$
equals
$$\lim_{M \rightarrow \infty} \big\|\big[\widehat{B}\,\big]^{\omega} \, \mathcal{U}(t-s) \big(\, \big[\Phi_{N}^{M}\big]^{\omega}(s) \big)\big\|_{L^{\infty}_{s \in I} L^{\infty}_{t \in I} L^2(\Omega) \mathcal{H}^{\alpha}_{\xi}},$$
which by \eqref{Important_Calculation} converges to zero as $N \rightarrow \infty$. 

The identity \eqref{Third_term_boundGP2B} now follows. As we noted above, this identity implies \eqref{Third_term_boundGP2}. The latter identity, in turn, implies that the third term on the right-hand side of \eqref{theta_omega_differenceGP2} converges to zero as $N \rightarrow \infty$.

\textbf{4)}
Finally, we note that the fourth term on the right-hand side of \eqref{theta_omega_differenceGP2} equals zero by the construction of $\big[\Gamma_N\big]^{\omega}$.

The proposition now follows.
\end{proof}

Let us now prove the analogue of  the convergence result given in Proposition \ref{CauchySequence2} for \eqref{RandomizedGP2}.

\begin{proposition}
\label{CauchySequence2GP2}
The sequence $\big(\big[\Gamma_N\big]^{\omega}\big)_N$ is Cauchy in $L^{\infty}_{t \in I}L^2(\Omega) \mathcal{H}^{\alpha}_{\xi}$. Moreover, there exists $\big[\Gamma\,\big]^{\omega} \in L^{\infty}_{t \in I}  L^2(\Omega) \mathcal{H}^{\alpha}_{\xi}$ such that $\big[\Gamma_N\big]^{\omega} \rightarrow \big[\Gamma\,\big]^{\omega}$ strongly in $L^{\infty}_{t \in I} L^2(\Omega) \mathcal{H}^{\alpha}_{\xi}$ as $N \rightarrow \infty$.

\medskip

The obtained $\big[\Gamma\,\big]^{\omega}$ satisfies the a priori bound:

\begin{equation}
\label{Gamma_omega_bound}
\big\|\big[\Gamma\,\big]^{\omega}(t)\big\|_{L^{\infty}_{t \in I} L^2(\Omega) \mathcal{H}^{\alpha}_{\xi}} \lesssim_{\,\xi,\,\xi',\,\alpha} \big\|\Gamma(0)\big\|_{\mathcal{H}^{\alpha}_{\xi'}}.
\end{equation}
\end{proposition}

\begin{proof}
The proof is similar to that of Proposition \ref{CauchySequence2}.
Let us fix $N_1,N_2 \in \mathbb{N}$ with $N_1<N_2$. By construction of $\big[\Gamma_{N_j}\big]^{\omega}$, for $j=1,2$, it follows that:
$$\big\|\big[\Gamma_{N_1}\big]^{\omega}(t)-\big[\Gamma_{N_2}\big]^{\omega}(t)\big\|_{L^{\infty}_{t \in I} L^2(\Omega)\mathcal{H}^{\alpha}_{\xi}} \leq \big\|\,\mathcal{U}(t)\,\big(\Gamma_{N_1}(0)-\Gamma_{N_2}(0)\big)\big\|_{L^{\infty}_{t \in I} L^2(\Omega)\mathcal{H}^{\alpha}_{\xi}}+$$
$$+\big\|\int_{0}^{t}\,\mathcal{U}(t-s)\,\big[\widehat{B}\,\big]^{\omega} \big(\big[\Gamma_{N_1}\big]^{\omega}(s)-\big[\Gamma_{N_2}\big]^{\omega}(s)\big)\,ds\,\big\|_{L^{\infty}_{t \in I} L^2(\Omega)\mathcal{H}^{\alpha}_{\xi}}$$
\begin{equation}
\label{CauchySequence2GP2_bound}
\leq \big\|\Gamma_{N_1}(0)-\Gamma_{N_2}(0)\big\|_{\mathcal{H}^{\alpha}_{\xi'}} +  T \cdot \big\|\big[\widehat{B}\,\big]^{\omega} \big[\Gamma_{N_1}\big]^{\omega}(s)-\big[\widehat{B}\,\big]^{\omega} \big[\Gamma_{N_2}\big]^{\omega}(s)\big\|_{L^{\infty}_{s \in I} L^2(\Omega) \mathcal{H}^{\alpha}_{\xi}}.
\end{equation}
Here, we used the unitarity of the free evolution operator $\mathcal{U}(t)$. 

Let us recall that $\Gamma(0) \in \mathcal{H}^{\alpha}_{\xi'}$. Furthermore, let us recall from Proposition \ref{limit2GP2} that $\big[\widehat{B}\,\big]^{\omega} \big[\Gamma_N\big]^{\omega} \rightarrow \theta^{\,\omega}$ strongly in $L^{\infty}_{t \in I} L^2(\Omega) \mathcal{H}^{\alpha}_{\xi}$ as $N \rightarrow \infty$. As a result, it follows that $\big(\big[\Gamma_N\big]^{\omega}\big)_N$ is Cauchy in $L^{\infty}_{t \in I} L^2(\Omega) \mathcal{H}^{\alpha}_{\xi}$. Since $L^{\infty}_{t \in I} L^2(\Omega) \mathcal{H}^{\alpha}_{\xi}$ is a Banach space, it follows that there exists $\big[\Gamma\,\big]^{\omega} \in L^{\infty}_{t \in I} L^2(\Omega) \mathcal{H}^{\alpha}_{\xi}$ such that:
\begin{equation}
\notag
\big[\Gamma_N\big]^{\omega} \rightarrow \big[\Gamma \,\big]^{\omega}
\end{equation}
strongly in $L^{\infty}_{t \in I} L^2(\Omega) \mathcal{H}^{\alpha}_{\xi}$ as $N \rightarrow \infty.$
Let us now prove \eqref{Gamma_omega_bound}. We fix $N_1$.
By arguing as in \eqref{CauchySequence2GP2_bound}, we note that:
$$\big\|\big[\Gamma_{N_1}\big]^{\omega}\big\|_{L^{\infty}_{t \in I} L^2(\Omega) \mathcal{H}^{\alpha}_{\xi}} 
\leq \big\|\Gamma_{N_1}(0)\big\|_{\mathcal{H}^{\alpha}_{\xi}} +
T \cdot \big\|\big[\widehat{B}\,\big]^{\omega} \big[\Gamma_N\,\big]^{\omega} \big\|_{L^{\infty}_{t \in I} L^2(\Omega) \mathcal{H}^{\alpha}_{\xi}}
.$$
By \eqref{BN1GP2bound} from Remark \ref{BN1GP2}, it follows that this quantity is:
$$\lesssim_{\,T,\,\xi,\,\xi',\,\alpha} \big\|\Gamma_{N_1}(0)\big\|_{\mathcal{H}^{\alpha}_{\xi}} + \big\|\Gamma(0)\big\|_{\mathcal{H}^{\alpha}_{\xi'}}.$$
We recall that $T$ is chosen to depend on $\xi,\xi'$, and $\alpha$. Hence:
$$\big\|\big[\Gamma_{N_1}\big]^{\omega}\big\|_{L^{\infty}_{t \in I} L^2(\Omega) \mathcal{H}^{\alpha}_{\xi}} \lesssim_{\,\xi,\xi',\alpha} \big\|\Gamma(0)\big\|_{\mathcal{H}^{\alpha}_{\xi'}}.$$
We let $N_1 \rightarrow \infty$ and we use the fact that $\big[\Gamma_{N_1}\,\big]^{\omega} \rightarrow \big[\Gamma\,\big]^{\omega}$ in $L^{\infty}_{t \in I} L^2(\Omega) \mathcal{H}^{\alpha}_{\xi}$ in order to deduce \eqref{Gamma_omega_bound}.

\end{proof}

Let us now find the equation which is satisfied by the limiting object $\big[\Gamma\,\big]^{\omega}$. We will prove the following analogue of Proposition \ref{Gamma_omega_star_equation}.

\begin{proposition}
\label{Gamma_omega_equation}
For $\big[\Gamma\,\big]^{\omega}$ constructed in Proposition \ref{CauchySequence2GP2} and for all $\xi_0>0$, the following identity holds:
$$\big\|\big[\Gamma\,\big]^{\omega}(t)-\mathcal{U}(t)\,\Gamma(0)+i\int_{0}^{t} \mathcal{U}(t-s)\,\theta^{\,\omega}(s)\,ds \,\big\|_{L^{\infty}_{t \in I} L^2(\Omega) \mathcal{H}^{\alpha}_{\xi_0}}=0.$$
\end{proposition}

\begin{proof}
As before, it suffices to show the claim for $\xi_0=\xi$, where $\xi$ is as in \eqref{xi2}.
We note that:
$$\big\|\big[\Gamma\,\big]^{\omega}(t)-\mathcal{U}(t)\,\Gamma(0)+i\int_{0}^{t} \mathcal{U}(t-s)\,\theta^{\,\omega}(s)\,ds \,\big\|_{L^{\infty}_{t \in I} L^2(\Omega) \mathcal{H}^{\alpha}_{\xi}}$$
$$
\leq \big\|\big[\Gamma\,\big]^{\omega}(t)-\big[\Gamma_N\big]^{\omega}(t)\big\|_{L^{\infty}_{t \in I} L^2(\Omega) \mathcal{H}^{\alpha}_{\xi}} + \big\|\,\mathcal{U}(t)\big(\Gamma(0)-\Gamma_N(0)\big)\big\|_{L^{\infty}_{t \in I} L^2(\Omega) \mathcal{H}^{\alpha}_{\xi}}+
$$
$$+\big\|\int_{0}^{t}\,\mathcal{U}(t-s)\,\big(\theta^{\,\omega}(s)-\big[\widehat{B}\,\big]^{\omega}\big[\Gamma_N\big]^{\omega}(s)\big)\,ds\,\big\|_{L^{\infty}_{t \in I} L^2(\Omega) \mathcal{H}^{\alpha}_{\xi}}+$$
$$+\big\|\big[\Gamma_N\big]^{\omega}(t)-\mathcal{U}(t)\,\Gamma_N(0)+i\int_{0}^{t}\,\mathcal{U}(t-s)\,\big[\widehat{B}\,\big]^{\omega}\big[\Gamma_N\big]^{\omega}(s)\,ds\big\|_{L^{\infty}_{t \in I} L^2(\Omega)\mathcal{H}^{\alpha}_{\xi}}.$$

\textbf{1)}
The first term converges to zero as $N \rightarrow \infty$ by using Proposition \ref{CauchySequence2GP2}.

\textbf{2)}
By unitarity, the second term equals: 
$$\big\|\Gamma(0)-\Gamma_N(0)\big\|_{L^2(\Omega) \mathcal{H}^{\alpha}_{\xi}}= \big\|\mathbf{P}_{>N}\Gamma(0)\big\|_{L^2(\Omega)\mathcal{H}^{\alpha}_{\xi}}=\big\|\mathbf{P}_{>N}\Gamma(0)\big\|_{\mathcal{H}^{\alpha}_{\xi}}$$
by the property of probability measures.
Since $\xi<\xi'$, this quantity is bounded from above by $\big\|\mathbf{P}_{>N}\Gamma(0)\big\|_{\mathcal{H}^{\alpha}_{\xi'}}$, which converges to zero as $N \rightarrow \infty$ since $\Gamma(0) \in \mathcal{H}^{\alpha}_{\xi'}$.

\textbf{3)}
By Minkowski's inequality and unitarity, the third term is:
%$$\leq \Big\| \int_{0}^{t} \big\|\theta^{\,\omega^{*}}(s)-\big[\widehat{B}\,\big]^{\omega^{*}}\big[\Gamma_N\big]^{\omega^{*}}(s)\big\|_{L^{\infty}_{t \in I} L^2(\Omega^{*})\mathcal{H}^{\alpha}_{\xi_1}}\,ds\,\Big\|_{L^{\infty}_{t \in I}}$$
$$\leq T \cdot \big\|\theta^{\,\omega}(s)-\big[\widehat{B}\,\big]^{\omega}\big[\Gamma_N\big]^{\omega}(s)\big\|_{L^{\infty}_{s \in I} L^2(\Omega) \mathcal{H}^{\alpha}_{\xi}},$$
which converges to zero as $N \rightarrow \infty$ by Proposition \ref{limit2GP2}.

\textbf{4)}
The fourth term equals to zero by construction of $\big[\Gamma_N\big]^{\omega}$.
The proposition now follows.
\end{proof}

We can now state and prove the main result of this section:
\begin{theorem}
\label{Theorem1GP2} 
For $\big[\Gamma\,\big]^{\omega}$ constructed in Proposition \ref{CauchySequence2GP2} and for all $\xi_0>0$, the following identity holds:
\begin{equation}
\label{Theorem1identity1GP2}
\big\|\big[\Gamma\,\big]^{\omega}(t)-\mathcal{U}(t)\,\Gamma(0)+i \int_{0}^{t}\,\mathcal{U}(t-s)\,\big[\widehat{B}\,\big]^{\omega} \big[\Gamma\,\big]^{\omega}(s)\,ds\,\big\|_{L^{\infty}_{t \in I} L^2(\Omega) \mathcal{H}^{\alpha}_{\xi_0}}=0.
\end{equation}
In particular, for almost all $t \in I$, it is the case that for all $k \in \mathbb{N}$:
\begin{equation}
\label{Theorem1identity2GP2}
\big\|\big(\big[\Gamma\,\big]^{\omega}(t)-\mathcal{U}(t)\,\Gamma(0)+i \int_{0}^{t}\,\mathcal{U}(t-s)\,
\big[\widehat{B}\,\big]^{\omega}\big[\Gamma\,\big]^{\omega}(s)\,ds\,\big)^{(k)}\big\|_{L^2(\Omega) H^{\alpha}(\Lambda^k \times \Lambda^k)}
=0.
\end{equation}
Hence, $\big[\Gamma\,\big]^{\omega}$ is a solution to the dependently randomized Gross-Pitaevskii hierarchy which satisfies the a priori bound \eqref{Gamma_omega_bound}.
\end{theorem}

\begin{proof}
Let us note that \eqref{Theorem1identity1GP2} implies \eqref{Theorem1identity2GP2}. We now prove \eqref{Theorem1identity1GP2}.
It suffices to prove the claim for $\xi_0=\xi$, where $\xi$ is as in \eqref{xi2}.
By Proposition \ref{Gamma_omega_equation} and by \eqref{BhatSigmabound} from Lemma \ref{BhatSigma}, it follows that:
\begin{equation}
\notag
%\label{Theorem1AGP2}
\big\|\big[\widehat{B}\,\big]^{\omega} \big[\Gamma\,\big]^{\omega}(t)-\big[\widehat{B}\,\big]^{\omega} \,\mathcal{U}(t)\,\Gamma(0)+i\int_{0}^{t}\big[\widehat{B}\,\big]^{\omega}\,\mathcal{U}(t-s)\,\theta^{\,\omega}(s)\,ds\big\|_{L^{\infty}_{t \in I} L^2(\Omega) \mathcal{H}^{\alpha}_{\xi}}=0.
\end{equation}
We now apply Proposition \ref{theta_omega_equation} in order to deduce that:
\begin{equation}
\label{Important_Bound}
\big\|\big[\widehat{B}\,\big]^{\omega}\big[\Gamma\,\big]^{\,\omega}(t)-\theta^{\,\omega}(t)\big\|_{L^{\infty}_{t \in I} L^2(\Omega) \mathcal{H}^{\alpha}_{\xi}}=0.
\end{equation}
Consequently, by Minkowski's inequality and unitarity:

\begin{equation}
\label{Theorem1CGP2}
\big\|\int_{0}^{t}\,\mathcal{U}(t-s)\,\big(\theta^{\,\omega}(s)-\big[\widehat{B}\,\big]^{\omega} \big[\Gamma\,\big]^{\omega}(s)\big)\,ds\,\big\|_{L^{\infty}_{t \in I} L^2(\Omega) \mathcal{H}^{\alpha}_{\xi}}
\end{equation}
$$\leq T \cdot \big\|\theta^{\,\omega}-\big[\widehat{B}\,\big]^{\omega}\big[\Gamma\,\big]^{\omega}\big\|_{L^{\infty}_{t \in I} L^2(\Omega) \mathcal{H}^{\alpha}_{\xi}}=0.$$
Finally, Proposition \ref{Gamma_omega_equation} and \eqref{Theorem1CGP2} imply:
$$\big\|\big[\Gamma\,\big]^{\omega}(t)-\mathcal{U}(t)\,\Gamma(0)+i\int_{0}^{t} \,\mathcal{U}(t-s)\,\big[\widehat{B}\,\big]^{\omega} \big[\Gamma\,\big]^{\omega}(s)\,ds\,\big\|_{L^{\infty}_{t \in I}L^2(\Omega)\mathcal{H}^{\alpha}_{\xi}}=0.$$
Theorem \ref{Theorem1GP2} now follows.
\end{proof}

\begin{remark}
\label{BhatGamma_omega_A}
From the identity \eqref{Important_Bound} and from the estimate \eqref{theta_omega_bound1} in Remark \ref{theta_omega_boundA}, we can deduce the following analogue of Remark \ref{BhatGamma_omega_starA} in the dependently randomized setting:

\begin{equation}
\label{BhatGamma_omega_1}
\big\|\big[\widehat{B}\,\big]^{\omega}\big[\Gamma\,\big]^{\,\omega}\big\|_{L^{\infty}_{t \in I} L^2(\Omega) \mathcal{H}^{\alpha}_{\xi}} \lesssim_{\,T,\,\xi,\,\xi',\,\alpha} \big\|\Gamma(0)\big\|_{\mathcal{H}^{\alpha}_{\xi'}}.
\end{equation}
As was the case in Remark \ref{BhatGamma_omega_starA}, this is a probabilistic version of the a priori bound that was assumed in the work of Klainerman and Machedon \cite{KM}.
\end{remark}

\subsection{The case of additional regularity in the initial data $\Gamma(0)$}
\label{additional_regularityGP2}

\medskip

In this section, we will argue as in Subsection \ref{additional_regularity} and we will show that the identity \eqref{Theorem1identity2GP2} in Theorem \ref{Theorem1GP2} holds for \emph{all times $t \in I$} if one is willing to take slightly more regular initial data $\Gamma(0)$. In other words, we prove an analogue of Theorem \ref{Theorem2} for the dependently randomized GP hierarchy \eqref{RandomizedGP2}. The main result of this section is the following:

\begin{theorem}
\label{Theorem2GP2}
Let $\alpha_0>\alpha \geq 0$ be given.
There exist constants $\widetilde{C_1}=\widetilde{C_1}(\alpha,\alpha_0)>0$ and $\widetilde{C_2}=\widetilde{C_2}(\alpha,\alpha_0)>0$, such that, whenever $\xi,\xi'>0$ with $\xi \in (0,\xi')$ and $T>0$ satisfy the following assumptions:
\begin{itemize}
\item[i)] $\frac{\widetilde{C_1}T}{\xi'}<1$
\item[ii)] $\frac{\widetilde{C_2} \,\xi}{\xi'}<1$
\end{itemize}
and whenever $\Gamma(0) \in \mathcal{H}^{\alpha_0}_{\xi'} \, \cap \, \mathcal{N}$, then there exists $\big[\Gamma\,\big]^{\omega} \in C_{t \in I} L^2(\Omega) \mathcal{H}^{\alpha}_{\xi}$ such that, for all  $t \in I=[0,T]$ and for all $\xi_0>0:$
\begin{equation}
\label{Theorem2identity1GP2}
\big\|\big[\Gamma\,\big]^{\omega}(t)-\mathcal{U}(t)\,\Gamma(0)+i \int_{0}^{t}\,\mathcal{U}(t-s)\,\big[\widehat{B}\,\big]^{\omega} \big[\Gamma\,\big]^{\omega}(s)\,ds\,\big\|_{L^2(\Omega) \mathcal{H}^{\alpha}_{\xi_0}}=0.
\end{equation}
In particular, for all $t \in I$, it is the case that for all $k \in \mathbb{N}$:
\begin{equation}
\label{Theorem2identity2GP2}
\big\|\big(\big[\Gamma\,\big]^{\omega}(t)-\mathcal{U}(t)\,\Gamma(0)+i \int_{0}^{t}\,\mathcal{U}(t-s)\,
\big[\widehat{B}\,\big]^{\omega}\big[\Gamma\,\big]^{\omega}(s)\,ds\,\big)^{(k)}\big\|_{L^2(\Omega) H^{\alpha}(\Lambda^k \times \Lambda^k)}
\end{equation}
\begin{equation}
\notag
=0.
\end{equation}
Moreover, $\big[\Gamma\,\big]^{\omega}$ satisfies the a priori bound:
\begin{equation}
\label{Theorem2GP2_a_priori_bound}
\big\|\big[\Gamma\,\big]^{\omega}\big\|_{L^{\infty}_{t \in I} L^2(\Omega) \mathcal{H}^{\alpha}_{\xi}} \lesssim_{\,\xi,\,\xi',\alpha,\alpha_0} \big\|\Gamma(0)\big\|_{\mathcal{H}^{\alpha}_{\xi'}}.
\end{equation}
\end{theorem}

\begin{remark}
We note that the constants $\widetilde{C_1}$ and $\widetilde{C_2}$ are 
chosen in such a way that they satisfy the constraints \eqref{widetildeC1C2A} and \eqref{widetildeC1C2B} below. These constraints are given in terms of $\alpha$ and $\alpha_0$ and hence $\widetilde{C_1}$ and $\widetilde{C_2}$ can be chosen to depend only on $\alpha$ and $\alpha_0$.
\end{remark}

\begin{remark}
With the same terminology as in Remark \ref{strong solution}, the solution $\big[\Gamma\,\big]^{\omega}$ in Theorem \ref{Theorem2GP2} can be thought of as a \textbf{strong solution} to the dependently randomized GP hierarchy, whereas the solution in Theorem \ref{Theorem1GP2} can be thought of as a \textbf{weak solution}. 
\end{remark}

\subsubsection{\textbf{An explicit formula for the difference of two Duhamel expansions}}
\label{An explicit formula for the difference of two Duhamel expansions}

The proof of Theorem \ref{Theorem2GP2} requires an additional step as compared to the proof of Theorem \ref{Theorem2}. The difficulty lies in the fact that the operator $\big[\widehat{B}\,\big]^{\omega} \big[\Gamma_N\big]^{\omega}$ has $\omega$ dependence in both $\big[\widehat{B}\,\big]^{\omega}$ and in $\big[\Gamma_N\big]^{\omega}$ and hence can not be estimated directly. The way that we are going to estimate the Duhamel expansions in which such operators occur is to explicitly write these expansions out in terms of the initial data $\Gamma(0)$. 

Identities of this type were already obtained in Subsection 6.2 of \cite{SoSt} and they require some combinatorial analysis of the frequencies. In this context, we treat all of the frequencies as formal objects and not as elements of $\mathbb{Z}^3$. In other words, the frequency $\xi_1$ and $\xi_2$ will be thought of as different formal objects even though they can be equal as elements of $\mathbb{Z}^3$. 
Furthermore, we recall the convention stated in Section \ref{Notation} by which we extend the definition of  $\big[B_{j,k+1}\,\big]^{\omega}$ to density matrices $\sigma^{(\ell)}$ of order $\ell>k+1$ by  acting only in the variables $\vec{x}_{k+1}$ and $\vec{x}'_{k+1}$.

Let us now summarize the analysis from \cite{SoSt}. We can restate formula (69) from Subsection 6.2 in \cite{SoSt} in the form of the following lemma:

\begin{lemma}
\label{Duhamel_Expansion}
Let us fix $k,j \in \mathbb{N}$ and $\ell_1,\ldots \ell_j, n_1, \ldots, n_j \in \mathbb{N}$, with $\ell_1<n_1 \leq k+j+1, \ell_2<n_2 \leq k+j+1, \ldots, 
\ell_{j+1}< n_{j+1} \leq k+j+1$. Furthermore, we fix $t,t_1,t_2, \ldots, t_j \in \mathbb{R}$ and a choice of $\pm$ signs for each of the collision operators. Then, for all $k+j+1$-particle density matrices $\sigma^{(k+j+1)}$:
\begin{equation}
\label{Duhamel_Expansion_j+1} 
\Big(\,\big[B^{\pm}_{\ell_1,n_1}\big]^{\omega}\,\mathcal{U}^{(k+1)}(t-t_1)\big[B^{\pm}_{\ell_2,n_2}\big]^{\omega}\,\mathcal{U}^{(k+2)}(t_1-t_2) \big[B^{\pm}_{\ell_3,n_3}\big]^{\omega} \cdots 
\end{equation}
\begin{equation}
\notag
\cdots \big[B^{\pm}_{\ell_j,n_j}\big]^{\omega} \, \mathcal{U}^{(k+j)}(t_{j-1}-t_j) \,\big[B^{\pm}_{\ell_{j+1},n_{j+1}}\big]^{\omega}\,\sigma^{(k+j+1)}
\Big)\,\,\widehat{}\,\,(\xi_1,\ldots,\xi_k; \xi'_1,\ldots,\xi'_k)=
\end{equation}

$$\mathop{\sum_{\eta_1,\ldots,\eta_{k+j+1}}}_{\eta'_1,\ldots,\eta'_{k+j+1}}^{*} 
e^{i(t-t_1)(\cdots)} \cdot e^{i(t_1-t_2)(\cdots)} \cdots e^{i(t_{j}-t_{j+1})(\cdots)} \cdot \widehat{\sigma}^{(k+j+1)}(\eta_1,\ldots,\eta_{k+j+1}; \eta'_1, \ldots, \eta'_{k+j+1}) \cdot$$
$$
\cdot \mathop{\prod_{1 \leq r \leq k}}_{\xi_r \in \mathcal{A}}
\Big\{h_{\xi_r}(\omega) \cdot h_{\eta_{1}^{r,1}} (\omega) \cdot h_{\eta_{2}^{r,1}} (\omega) \cdots h_{\eta_{N_r}^{r,1}}(\omega) \Big\} \cdot \mathop{\prod_{1 \leq r \leq k}}_{\xi'_r \in \mathcal{B}}
\Big\{h_{\xi'_r}(\omega) \cdot h_{\eta_{1}^{r,2}} (\omega) \cdot h_{\eta_{2}^{r,2}} (\omega) \cdots h_{\eta_{M_r}^{r,2}}(\omega) \Big\}.$$
The notation is explained below.
\end{lemma}

%The notation which we use in the above lemma is the same as in Subsection 6.2 of \cite{SoSt}, when we gave the precise form of the Duhamel expansion term. The only new ingredient is the function $F$ which we will define below. Before we do this, let us first summarize, for completeness, the notation from \cite{SoSt}.

In the above formula $(\cdots)$ denotes a real-valued expression given in terms of the frequencies. We say that $\xi_r \in \mathcal{A}$ for some $1 \leq r \leq k$ if one of the collision operators $[B_{\ell,n}^{+}]^{\omega}$ was applied to this frequency in the above expansion. Namely, this means that $\xi_r \in \mathcal{A}$ if $\xi_r$ does not appear in the list $(\eta_1,\ldots,\eta_{k+j+1};\eta'_1,\ldots,\eta'_{k+j+1})$. Here, we recall that we think of the frequencies as formal symbols and do not consider their values in $\mathbb{Z}^3$. Similarly, we say that $\xi'_r \in \mathcal{B}$ for some $1 \leq r \leq k$ if one of the collision operators $[B_{\ell,n}^{-}]^{\omega}$ was applied to this frequency. In particular, $\xi'_r$ then does not appear in the list $(\eta_1,\ldots,\eta_{k+j+1};\eta'_1,\ldots,\eta'_{k+j+1})$.

As in \cite{SoSt}, the summation
\begin{equation}
\label{Sigma_Star}
\mathop{\sum_{\eta_1,\ldots,\eta_{k+j+1}}}_{\eta'_1,\ldots,\eta'_{k+j+1}}^{*}(\cdots)
\end{equation}
is a symbol for the sum in $\eta_1,\ldots, \eta_{k+j+1}; \eta'_1,\ldots,\eta'_{k+j+1}$, which satisfy:
\begin{itemize}
\item[$i)$] The set $\big\{\eta_1,\ldots,\eta_{k+j+1},\eta'_1,\ldots,\eta'_{k+j+1}\big\}$ can be written as:
$$\Big(\mathop{\bigcup_{1 \leq r \leq k}}_{\xi_r \in \mathcal{A}}\big\{\eta^{r,1}_1,\ldots,\eta^{r,1}_{N_r}\big\}\Big) \cup \Big(\mathop{\bigcup_{1 \leq r \leq k}}_{\xi'_r \in \mathcal{B}}\big\{\eta^{r,2}_1,\ldots,\eta^{r,2}_{M_r}\big\}\Big)$$
(these represent the new frequencies obtained after applying the collision operators) and:

$$\big\{\xi_1,\ldots,\xi_k,\xi'_1,\ldots,\xi'_k\big\} \setminus \big(\mathcal{A}\cup \mathcal{B}\big).$$
(these represent the frequencies on which we do not apply the collision operators, which hence remain the same).
\item[$ii)$] Each $\xi_r \in \mathcal{A}$ can be written as:
$$\xi_r=\sum_{s=1}^{N_r}  \epsilon (\eta_s^{r,1}) \cdot \eta_s^{r,1},$$
for some fixed sequence of coefficients $\epsilon (\eta_s^{r,1}) \in \{1,-1\}.$
Here, we note that it is important to treat the frequency $\eta_s^{r,1}$ as a formal symbol and not as a value in $\mathbb{Z}^3$. The $\epsilon (\eta_s^{r,1})$ are uniquely determined by the collision operators that we apply. In this case, we write:
$$\mathcal{A}_{r}:=\big\{\eta_1^{r,1},\eta_2^{r,1}, \ldots, \eta_{N_{r}}^{r,1}\big\}$$
%$$\xi_{\ell_1}=\sum_{a \in \mathcal{A}_{\ell_1}} \epsilon(a) \cdot a.$$

\item[$iii)$] Similarly, each $\xi'_r \in \mathcal{B}$ can be written as:
$$\xi'_r= \sum_{s=1}^{M_r} \epsilon (\eta_{s}^{r,2}) \cdot  \eta_{s}^{r,2},$$
for some fixed sequence of coefficients $\epsilon (\eta_s^{r,2}) \in \{1,-1\}$ and in this case, we write:
$$\mathcal{B}_{r}:=\big\{\eta_1^{r,2},\eta_2^{r,2}, \ldots, \eta_{M_{r}}^{r,2}\big\}.$$ 
\end{itemize}

\begin{remark}
We note that the sets $\mathcal{A}$ and $\mathcal{B}$ and the numbers $N_r,M_r$, for $r=1,\ldots,k$, are uniquely determined by the choice of the collision operators that we are applying, i.e. by the choice of $\ell_r,n_r$ and by the choice of $\pm$ signs.
\end{remark}

The result that we want is an explicit formula for the \emph{difference of two Duhamel expansions}. We will prove the following claim:

\begin{lemma}
\label{Duhamel_Expansion_Difference}
%Let us fix $k,j \in \mathbb{N}$ and $\ell_1,\ldots \ell_j, n_1, \ldots, n_j \in \mathbb{N}$, with $\ell_1<n_1 \leq k+j+1, \ell_2<n_2 \leq k+j+1, \ldots, 
%\ell_{j+1}< n_{j+1} \leq k+j+1$. Furthermore, we fix $\delta>0$ and $t,t_1,t_2, \ldots, t_j \in \mathbb{R}$ and a choice of $\pm$ signs for each of the collision operators. Then, for all $k+j+1$-particle density matrices $\sigma^{(k+j+1)}$:
Let us fix $\delta>0$. With the assumptions and notation as in Lemma \ref{Duhamel_Expansion}, the following identity holds:
\begin{equation}
\label{Duhamel_Expansion_Difference1}
\Big(\,\big[B^{\pm}_{\ell_1,n_1}\big]^{\omega}\,\big\{\mathcal{U}^{(k+1)}(t+\delta)-\mathcal{U}^{(k+1)}(t)\big\}\big[B^{\pm}_{\ell_2,n_2}\big]^{\omega}\,\mathcal{U}^{(k+2)}(t_1-t_2) \big[B^{\pm}_{\ell_3,n_3}\big]^{\omega} \cdots 
\end{equation}
\begin{equation}
\notag
\cdots \big[B^{\pm}_{\ell_j,n_j}\big]^{\omega} \, \mathcal{U}^{(k+j)}(t_{j-1}-t_j) \,\big[B^{\pm}_{\ell_{j+1},n_{j+1}}\big]^{\omega}\,\sigma^{(k+j+1)}
\Big)\,\,\widehat{}\,\,(\xi_1,\ldots,\xi_k; \xi'_1,\ldots,\xi'_k)=
\end{equation}

$$\mathop{\sum_{\eta_1,\ldots,\eta_{k+j+1}}}_{\eta'_1,\ldots,\eta'_{k+j+1}}^{*} 
\{e^{-i\delta F}-1\} \cdot e^{it(\cdots)} \cdot e^{i(t_1-t_2)(\cdots)} \cdots e^{i(t_{j}-t_{j+1})(\cdots)} \cdot \widehat{\sigma}^{(k+j+1)}(\eta_1,\ldots,\eta_{k+j+1}; \eta'_1, \ldots, \eta'_{k+j+1}) \cdot$$
$$
\cdot \mathop{\prod_{1 \leq r \leq k}}_{\xi_r \in \mathcal{A}}
\Big\{h_{\xi_r}(\omega) \cdot h_{\eta_{1}^{r,1}} (\omega) \cdot h_{\eta_{2}^{r,1}} (\omega) \cdots h_{\eta_{N_r}^{r,1}}(\omega) \Big\} \cdot \mathop{\prod_{1 \leq r \leq k}}_{\xi'_r \in \mathcal{B}}
\Big\{h_{\xi'_r}(\omega) \cdot h_{\eta_{1}^{r,2}} (\omega) \cdot h_{\eta_{2}^{r,2}} (\omega) \cdots h_{\eta_{M_r}^{r,2}}(\omega) \Big\},$$
for some real-valued function $F=F(\eta_1,\ldots,\eta_{k+j+1};\eta'_1,\ldots,\eta'_{k+j+1})$, such that:
\begin{equation}
\label{Estimate_On_F}
|F| \leq C_3^{k+j+1} \cdot \big(|\eta_1|^2+\cdots+|\eta_{k+j+1}|^2+|\eta'_1|^2+\cdots+|\eta'_{k+j+1}|^2\big),
\end{equation}
for some universal constant $C_3>0$.
In particular, for all $\rho \in [0,1]$:
\begin{equation}
\label{Estimate_On_F2}
\big|e^{-i\delta F}-1\big| \lesssim_{\,\rho} \delta^{\rho} \cdot C_3^{(k+j+1) \cdot \rho} \cdot (|\eta_1|^2+\cdots+|\eta_{k+\ell}|^2+|\eta'_1|^2+\cdots+|\eta'_{k+\ell}|^2)^{\rho}.
\end{equation}
\end{lemma}

We now give a proof of Lemma \ref{Duhamel_Expansion}.

\begin{proof}
Let us consider two cases separately, depending on whether the first collision operator is $\big[B^{+}_{\ell_1,n_1}\big]^{\omega}$ or $\big[B^{-}_{\ell_1,n_1}\big]^{\omega}$.

\medskip

\textbf{Case 1:} Suppose that we are first applying the collision operator $\big[B^{+}_{\ell_1,n_1}\big]^{\omega}$. 
Since the collision operator $\big[B^{+}_{\ell_1,n_1}\big]^{\omega}$ is applied to the frequency $\xi_{\ell_1}$, it follows that $\xi_{\ell_1} \in \mathcal{A}$.  

We first use Lemma \ref{Duhamel_Expansion} (with appropriately shifted indices) to compute the expression:

\begin{equation}
\label{Duhamel_Expansion_j}
\Big(\,\big[B^{\pm}_{\ell_2,n_2}\big]^{\omega}\,\mathcal{U}^{(k+2)}(t_1-t_2) \big[B^{\pm}_{\ell_3,n_3}\big]^{\omega} \cdots
\end{equation}
\begin{equation}
\notag 
\cdots \big[B^{\pm}_{\ell_j,n_j}\big]^{\omega} \, \mathcal{U}^{(k+j)}(t_{j-1}-t_j) \,\big[B^{\pm}_{\ell_{j+1},n_{j+1}}\big]^{\omega}\,\sigma^{(k+j+1)}
\Big)\,\,\widehat{}\,\,(\zeta_1,\ldots,\zeta_{k+1};\zeta'_1,\ldots,\zeta'_{k+1}),
\end{equation}
for frequencies $(\zeta_1,\ldots,\zeta_{k+1};\zeta'_1,\ldots,\zeta'_{k+1})$.
We then recall the inductive procedure used in Subsection 6.2. of \cite{SoSt}. By using the argument from \cite{SoSt}, and the definition of the operators $\big[B_{\ell_1,n_1}^{+}\big]^{\omega}$ and $\mathcal{U}^{(k+1)}(t+\delta)-\mathcal{U}^{(k)}(t)$ on the Fourier domain, it is possible to deduce \eqref{Duhamel_Expansion_Difference1} from the formula for \eqref{Duhamel_Expansion_j}. We will omit the details and we refer the interested reader to Subsection 6.2 of \cite{SoSt}.
From this argument, it follows that we can write the set $\mathcal{A}_{\ell_1}$ (obtained from \eqref{Duhamel_Expansion_j+1} ) as a disjoint union of non-empty sets $\mathcal{A}^{1}_{\ell_1},\mathcal{A}^{2}_{\ell_1},\mathcal{A}^{3}_{\ell_1}$ such that if we set:
$$\nu_{\ell_1,n_1}:=\sum_{a \in \mathcal{A}^{1}_{\ell_1}}\epsilon (a) \cdot a$$
and 
$$\nu'_{\ell_1,n_1}:=-\sum_{a \in \mathcal{A}^{2}_{\ell_1}} \epsilon(a) \cdot a,$$
then we can take $F$ to be:
\begin{equation}
\label{F_Case1}
F=|\xi_1|^2+\cdots+|\xi_{\ell_1-1}|^2+|\xi_{\ell_1}-\nu_{\ell_1,n_1}+\nu'_{\ell_1,n_1}|^2+
\end{equation}
$$+|\xi_{\ell_1+1}|^2+\cdots+|\xi_k|^2+|\nu_{\ell_1,n_1}|^2-|\xi'_1|^2-\cdots-|\xi'_k|^2-|\nu'_{\ell_1,n_1}|^2,$$
and the relation \eqref{Duhamel_Expansion_Difference1} will hold.
It follows from \eqref{F_Case1} that:
$$|F| \leq |\xi_1|^2+\cdots+|\xi_{\ell_1-1}|^2+|\xi_{\ell_1}-\nu_{\ell_1,n_1}+\nu'_{\ell_1,n_1}|^2+
$$
$$+|\xi_{\ell_1+1}|^2+\cdots+|\xi_k|^2+|\nu_{\ell_1,n_1}|^2+|\xi'_1|^2+\cdots+|\xi'_k|^2+|\nu'_{\ell_1,n_1}|^2.$$
The estimate \eqref{Estimate_On_F} now follows by construction of the set $\{\eta_1,\ldots,\eta_{k+j+1};\eta'_1,\ldots,\eta'_{k+j+1}\}$. The power of the universal constant $C_3$ comes from each application of the \emph{fractional Leibniz rule}, i.e. by writing $\xi_r$ and $\xi'_r$ as in $i)$ and $ii)$ respectively. We note that there are at most $2(k+j+1)$ such applications of the fractional Leibniz rule. 

We observe that, in Case 1:

$$\xi_{\ell_1}-\nu_{\ell_1,n_1}+\nu'_{\ell_1,n_1}=\sum_{a \in \mathcal{A}_{\ell_1}^3} \epsilon (a) \cdot a.$$

\medskip

\textbf{Case 2:} Suppose now that we are first applying the collision operator $\big[B^{-}_{\ell_1,n_1}\big]^{\omega}$. 
Since the collision operator $\big[B^{-}_{\ell_1,n_1}\big]^{\omega}$ is applied to the frequency $\xi'_{\ell_1}$, it follows that $\xi'_{\ell_1} \in \mathcal{B}$.  

Arguing as in Case 1, it follows that we can write the set $\mathcal{B}_{\ell_1}$ (obtained from \eqref{Duhamel_Expansion_j+1}) as a disjoint union of non-empty sets $\mathcal{B}^{1}_{\ell_1},\mathcal{B}^{2}_{\ell_1},\mathcal{B}^{3}_{\ell_1}$ such that if we set:
$$\mu_{\ell_1,n_1}:=-\sum_{b \in \mathcal{B}^{1}_{\ell_1}}\epsilon (b) \cdot b$$
and 
$$\mu'_{\ell_1,n_1}:=\sum_{b \in \mathcal{B}^{2}_{\ell_1}} \epsilon(b) \cdot b,$$
then we can take $F$ to be:
\begin{equation}
\label{F_Case2}
F=|\xi_1|^2+\cdots+|\xi_k|^2+|\mu_{\ell_1,n_1}|^2-|\xi'_1|^2-\cdots-|\xi_{\ell_1-1}|^2
\end{equation}
$$-|\xi'_{\ell_1}-\mu'_{\ell_1,n_1}+\mu_{\ell_1,n_1}|^2-|\xi'_{\ell_1+1}|^2-\cdots-|\xi'_k|^2-|\mu'_{\ell_1,n_1}|^2$$
and the relation \eqref{Duhamel_Expansion_Difference1} will hold. The obtained $F$ will satisfy the estimate \eqref{Estimate_On_F} as in Case 1.
We note that, in Case 2:
$$\xi'_{\ell_1}-\mu'_{\ell_1,n_1}+\mu_{\ell_1,n_1}=\sum_{b \in \mathcal{B}_{\ell_1}^3} \epsilon(b) \cdot b.$$
The identity \eqref{Duhamel_Expansion_Difference1} now follows.

Finally, we note that \eqref{Estimate_On_F2} when $\rho=1$ follows from \eqref{Estimate_On_F} by applying the same argument we used to prove \eqref{difference2}. The estimate when $\rho=0$ follows since $F$ is real-valued. The claim for general $\rho \in [0,1]$ then follows by interpolation.

\end{proof}

In order to make the ideas in the proof of Lemma \ref{Duhamel_Expansion_Difference} more transparent, let us give a concrete example. We will modify Example 2 from Subsection 6.2 of \cite{SoSt}.

\begin{example}
\label{Example}
We suppose that $k=2, j=3$. As in \cite{SoSt}, for simplicity of notation, let us analyze the case when $t_1=t_2=t_3=t_4=0$. For the collision operators, we take $[B_{1,2}^{+}]^{\omega}, [B_{2,3}^{-}]^{\omega}, [B_{4,5}^{-}]^{\omega}$.
\\
\\
We note that:
$$\Big(\big[B_{1,2}^{+}\big]^{\omega} \big(\,\mathcal{U}^{(3)}(t+\delta)-\mathcal{U}^{(3)}(t)\big) \big[B_{2,3}^{-}\big]^{\omega} \big[B_{4,5}^{-}\big]^{\omega} \sigma^{(5)} \Big)\,\,\widehat{}\,\,(\xi_1,\xi_2;\xi'_1,\xi'_2)=$$
$$=\sum_{\xi_3,\xi'_3} \Big(\big(\,\mathcal{U}^{(3)}(t+\delta)-\mathcal{U}^{(3)}(t)\big) [B_{2,3}^{-}]^{\omega} [B_{4,5}^{-}]^{\omega} \sigma^{(5)}\Big)\,\,\widehat{}\,\,(\xi_1-\xi_3+\xi'_3,\xi_3,\xi_2; \xi'_1,\xi'_3,\xi'_2) \, \cdot$$
$$\cdot \, h_{\xi_1}(\omega) \cdot h_{\xi_1-\xi_3+\xi'_3}(\omega) \cdot h_{\xi_3}(\omega) \cdot h_{\xi'_3}(\omega)=$$
$$=\sum_{\xi_3,\xi'_3} \big\{e^{-i\delta(|\xi_1-\xi_3+\xi'_3|^2+|\xi_2|^2+|\xi_3|^2-|\xi'_1|^2-|\xi'_2|^2-|\xi'_3|^2)}-1\big\} \cdot e^{-it(|\xi_1-\xi_3+\xi'_3|^2+|\xi_2|^2+|\xi_3|^2-|\xi'_1|^2-|\xi'_2|^2-|\xi'_3|^2)} \cdot $$
$$\cdot \, \Big([B_{2,3}^{-}]^{\omega} [B_{4,5}^{-}]^{\omega} \sigma^{(5)}\Big)\,\,\widehat{}\,\,(\xi_1-\xi_3+\xi'_3,\xi_3,\xi_2; \xi'_1,\xi'_3,\xi'_2) \cdot h_{\xi_1}(\omega) \cdot h_{\xi_1-\xi_3+\xi'_3}(\omega) \cdot h_{\xi_3}(\omega) \cdot h_{\xi'_3}(\omega)$$
By using the expansions as in the mentioned example from \cite{SoSt}, this expression equals:
$$=\sum_{\xi_3,\xi'_3,\xi_4,\xi'_4,\xi_5,\xi'_5}
\big\{e^{-i\delta(|\xi_1-\xi_3+\xi'_3|^2+|\xi_2|^2+|\xi_3|^2-|\xi'_1|^2-|\xi'_2|^2-|\xi'_3|^2)}-1 \big\} \cdot e^{-it(|\xi_1-\xi_3+\xi'_3|^2+|\xi_2|^2+|\xi_3|^2-|\xi'_1|^2-|\xi'_2|^2-|\xi'_3|^2)} \, \cdot
$$
$$\cdot \, \widehat{\sigma}^{(5)}(\xi_1-\xi_3+\xi'_3,\xi_3,\xi_4,\xi_2,\xi_5; \xi'_1,\xi'_3-\xi'_4+\xi_4,\xi'_4,\xi'_2-\xi'_5+\xi_5,\xi'_5) \cdot$$
$$\cdot \, h_{\xi_1}(\omega) \cdot h_{\xi_1-\xi_3+\xi'_3}(\omega) \cdot h_{\xi_3}(\omega) \cdot h_{\xi'_3-\xi'_4+\xi_4}(\omega) \cdot h_{\xi_4}(\omega) \cdot h_{\xi'_4}(\omega) \, \cdot$$
$$\cdot \, h_{\xi'_2}(\omega) \cdot h_{\xi'_2-\xi'_5+\xi_5}(\omega) \cdot h_{\xi_5}(\omega) \cdot h_{\xi'_5}(\omega).$$
By construction, $\mathcal{A}=\{\xi_1\}$, $\mathcal{B}=\{\xi_2\}$. Furthermore,
$$\xi_1=(\xi_1-\xi_3+\xi'_3)+\xi_3+\xi_4-(\xi'_3-\xi'_4+\xi_4)-\xi'_4=\eta_{1}^{1,1}+\eta_{2}^{1,1}+\eta_{3}^{1,1}-\eta_{4}^{1,1}-\eta_{5}^{1,1}$$
and
\begin{equation}
\label{xi2prime}
\xi'_2=-\xi_5+(\xi'_2-\xi'_5+\xi_5)+\xi'_5=-\eta_{1}^{2,2}+\eta_{2}^{2,2}+\eta_{3}^{2,2}=-\eta_5+\eta_4'+\eta_5'.
\end{equation}
With notation as in the proof of Lemma \ref{Duhamel_Expansion_Difference}, we can take:

\begin{equation}
\label{nu12}
\nu_{1,2}:=\xi_3=\eta_{2}^{1,1}=\eta_2
\end{equation}
and
\begin{equation}
\label{nu12prime}
\nu'_{1,2}:=\xi'_3=-\xi_4+(\xi'_3-\xi'_4+\xi_4)+\xi'_4=-\eta_{3}^{1,1}+\eta_{4}^{1,1}+\eta_{5}^{1,1}=-\eta_3+\eta'_2+\eta'_3.
\end{equation}
We hence take:
$$F(\eta_1,\ldots,\eta_5;\eta'_1,\ldots,\eta'_5):=|\xi_1-\nu_{1,2}+\nu'_{1,2}|^2+|\xi_2|^2+|\nu_{1,2}|^2-|\xi'_1|^2-|\xi'_2|^2-|\nu'_{1,2}|^2=$$
$$=|\eta_1|^2+|\eta_4|^2+|\eta_2|^2-|\eta'_1|^2-|-\eta_5+\eta'_4+\eta'_5 |^2 - |-\eta_3+\eta'_2+\eta'_3|^2.$$
Here, we used \eqref{xi2prime}, \eqref{nu12}, \eqref{nu12prime} and the fact that, in the ebove expansion:
$$\xi_1-\nu_{1,2}+\nu'_{1,2}=\eta_1, \xi_2=\eta_4, \xi_1'=\eta_1'.$$
\end{example}

\subsubsection{\textbf{Proof of Theorem \ref{Theorem2GP2}}}
\label{Proof of Theorem2GP2}

\medskip

In the proof of Theorem \ref{Theorem2GP2}, we use the following analogue of Definition \ref{equicontinuousOmegastar}:

\begin{definition}
A sequence $(F_N)_{N \in \mathbb{N}}$ in $C_{t \in I}L^2(\Omega)\mathcal{H}^{\alpha}_{\xi}$ is said to be \textbf{equicontinuous in} $L^2(\Omega)\mathcal{H}^{\alpha}_{\xi}$ if for all $\epsilon>0$, there exists $\delta>0$ such that for all $N \in \mathbb{N}$, and for all $t_1, t_2 \in I$,  $|t_1-t_2| \leq \delta$ implies that $\big\|F_N(t_1)-F_N(t_2)\big\|_{L^2(\Omega) \mathcal{H}^{\alpha}_{\xi}} \leq \epsilon.$
\end{definition}

Let us now give the proof of Theorem \ref{Theorem2GP2}:

\begin{proof}(of Theorem \ref{Theorem2GP2})
We will show that $\big[\Gamma\,\big]^{\omega}$, constructed in Proposition \ref{CauchySequence2GP2} has the wanted properties. As before, we take into account the fact that $T$ can become smaller due to the dependence on $\alpha_0$. The a priori bound \eqref{Theorem2GP2_a_priori_bound} follows from \eqref{Gamma_omega_bound} in the same way that \eqref{Theorem2_a_priori_bound} followed from \eqref{Gamma_omega_star_bound}. The additional dependence of the implied constant on $\alpha_0$ follows from the fact that $T$ now also depends on $\alpha_0$. Moreover, we might also need to choose $\xi'$ to be smaller in terms of $\xi$ due to the additional $\alpha_0$ dependence. For simplicity of exposition, we will not explicitly emphasize this distinction in the discussion that follows. 

It suffices to prove \eqref{Theorem2identity1GP2} since this claim implies \eqref{Theorem2identity2GP2}. Furthermore, it suffices to prove \eqref{Theorem2identity1GP2} when $\xi_0 = \xi$, for $\xi$ as in the assumptions of the theorem. The claim  then will follow, for all $\xi_0$. We henceforth fix $\xi_0=\xi$.

\medskip

The proof of the theorem will be divided into several steps.

\medskip

\textbf{Step 1:} The sequence $\big(\big[\Gamma_N\,\big]^{\omega}\big)_N$, constructed as in \eqref{GammaNGP2}, is equicontinuous in $L^2(\Omega) \mathcal{H}^{\alpha}_{\xi}$.

\medskip

As in Step 1 of the proof of Theorem \ref{Theorem2}, the sequence $\big(\big[\Gamma_N\,\big]^{\omega}\big)_N$ is constructed on the time interval $[0,T]$, where $T$ is given by the assumptions of Theorem \ref{Theorem2GP2}. In particular, $T$ depends on $\alpha_0$. We will now prove the equicontinuity property.

\medskip

Let us fix $N \in \mathbb{N}$ and let us consider $\big
[\Gamma_N\,\big]^{\omega}$. We take $k \in \mathbb{N}$, $t \in I$ and $\delta>0$ small and we first observe that:
\begin{equation}
\label{GammaNomega_deltaGP2}
\big\|\big(\big[\Gamma_N\,\big]^{\omega}\big)^{(k)}(t+\delta)-\big(\big[\Gamma_N\,\big]^{\omega}\big)^{(k)}(t)\big\|_{L^2(\Omega) H^{\alpha}(\Lambda^k \times \Lambda^k)}
\end{equation}
%$$\leq \big\|\sum_{j=0}^{N-k} \big\{Duh_{\,j}^{\,\omega}(\Gamma_N(0))^{(k)}(t+\delta)- Duh_{\,j}^{\,\omega}(\Gamma_N(0))^{(k)}(t) \big\}\big\|_{L^2(\Omega) H^{\alpha}(\Lambda^k \times \Lambda^k)}$$
$$\leq \sum_{j=0}^{N-k} \big\|Duh_{\,j}^{\,\omega}(\Gamma_N(0))^{(k)}(t+\delta)- Duh_{\,j}^{\,\omega}(\Gamma_N(0))^{(k)}(t)\big\|_{L^2(\Omega) H^{\alpha}(\Lambda^k \times \Lambda^k)}.$$
By construction:
\begin{equation}
\notag
%\label{difference_deltaGP2}
Duh_{\,j}^{\,\omega}(\Gamma_N(0))^{(k)}(t+\delta)- Duh_{\,j}^{\,\omega}(\Gamma_N(0))^{(k)}(t)=
\end{equation}
$$=(-i)^j \int_0^t \int_0^{t_1} \cdots \int_0^{t_{j-1}} \big(\,\mathcal{U}^{(k)}(t+\delta-t_1)-\,\mathcal{U}^{(k)}(t-t_1)\big) \big[B^{(k+1)}\big]^{\omega} \, \mathcal{U}^{(k+1)}(t_1-t_2) \, \big[B^{(k+2)}\big]^{\omega} \cdots $$
$$\cdots \,\mathcal{U}^{(k+j-1)}(t_{j-1}-t_j) \, \big[B^{(k+j)}\big]^{\omega} \, \mathcal{U}^{(k+j)}(t_j) \gamma_0^{(k+j)}\, dt_j \, dt_{j-1} \,\cdots\, dt_2 \, dt_1$$
$$+(-i)^j \int_{t}^{t+\delta} \int_0^{t_1} \cdots \int_0^{t_{j-1}} \,\mathcal{U}^{(k)}(t+\delta-t_1) \, \big[B^{(k+1)}\big]^{\omega} \, \mathcal{U}^{(k+1)}(t_1-t_2) \, \big[B^{(k+2)}\big]^{\omega} \cdots $$
$$\cdots \,\mathcal{U}^{(k+j-1)}(t_{j-1}-t_j) \, \big[B^{(k+j)}\big]^{\omega} \, \mathcal{U}^{(k+j)}(t_j) \gamma_0^{(k+j)}\, dt_j \, dt_{j-1} \,\cdots\, dt_2 \, dt_1$$
$$=: \mathcal{I} + \mathcal{II}.$$
We now estimate the quantities $\mathcal{I}$ and $\mathcal{II}$ in the norm $\big\|\cdot\big\|_{L^2(\Omega)H^{\alpha}(\Lambda^k \times \Lambda^k)}.$

\medskip

By Minkowski's inequality:

$$\big\|\,\mathcal{I}\,\big\|_{L^2(\Omega) H^{\alpha}(\Lambda^k \times \Lambda^k)} \leq \int_0^t \int_0^{t_1} \cdots \int_0^{t_{j-1}} \big\| \big(\,\mathcal{U}^{(k)}(t+\delta-t_1)-\,\mathcal{U}^{(k)}(t-t_1)\big) \, \big[B^{(k+1)}\big]^{\omega} \, \mathcal{U}^{(k+1)}(t_1-t_2)$$
$$\big[B^{(k+2)}\big]^{\omega} \cdots \,\mathcal{U}^{(k+j-1)}(t_{j-1}-t_j) \, \big[B^{(k+j)}\big]^{\omega} \, \mathcal{U}^{(k+j)}(t_j) \gamma_0^{(k+j)}\big\|_{L^2(\Omega) H^{\alpha}(\Lambda^k \times \Lambda^k)}\, dt_j \, dt_{j-1} \,\cdots\, dt_2 \, dt_1,$$
which, by Lemma \ref{DifferenceBound} is:
$$\lesssim_{\,\alpha,\,\alpha_0} \delta^{r} \cdot \int_0^t \int_0^{t_1} \cdots \int_0^{t_{j-1}} \big\| \big[B^{(k+1)}\big]^{\omega} \, \mathcal{U}^{(k+1)}(t_1-t_2) \, \big[B^{(k+2)}\big]^{\omega} \, \mathcal{U}^{(k+2)}(t_2-t_3) \, \cdots $$
$$\cdots \,\mathcal{U}^{(k+j-1)}(t_{j-1}-t_j) \, \big[B^{(k+j)}\big]^{\omega} \, \mathcal{U}^{(k+j)}(t_j) \gamma_0^{(k+j)}\big\|_{L^2(\Omega) H^{\alpha_0}(\Lambda^k \times \Lambda^k)}\, dt_j \, dt_{j-1} \,\cdots\, dt_2 \, dt_1$$
for some $r=r(\alpha,\alpha_0) \in (0,1)$.
By using the fact that $\Gamma(0)$ was assumed to be non-resonant and by applying Proposition \ref{non-resonant} in regularity $\alpha_0$, it follows that this expression is:
$$\leq \delta^{r} \cdot  \int_0^t \int_0^{t_1} \cdots \int_0^{t_{j-1}} 2^j \, \cdot \, (C'_0)^{k+j} \, \cdot \, k \, \cdot \, (k+1) \, \cdots \, (k+j-1) \, \cdot \, \big\|\gamma_0^{(k+j)} \big\|_{H^{\alpha_0}(\Lambda^{k+j} \times \Lambda^{k+j})} \, dt_j \, dt_{j-1} \,\cdots\, dt_2 \, dt_1.$$
Here, $C'_0=C'_0(\alpha_0)>0$ is the implied constant obtained for the regularity $\alpha_0$ in Proposition \ref{non-resonant}.
Arguing as in the proof of Proposition \ref{DuhamelEstimate2}, we can deduce that:
\begin{equation}
\label{IGP2}
\big\|\,\mathcal{I}\,\big\|_{L^2(\Omega) H^{\alpha}(\Lambda^k \times \Lambda^k)} \leq \delta^r \cdot (C_1' T)^j \cdot \, ({C'_2})^k \cdot \big\| \gamma_0^{(k+j)}\big\|_{H^{\alpha_0}(\Lambda^{k+j} \times \Lambda^{k+j})}
\end{equation}
for some constants $C_1'=C_1'(\alpha,\alpha_0), C_2'=C_2'(\alpha,\alpha_0)>0$.

\medskip

We now estimate $\mathcal{II}$:
$$\big\|\,\mathcal{II}\,\big\|_{L^2(\Omega) H^{\alpha}(\Lambda^k \times \Lambda^k)} \leq 
\int_{t}^{t+\delta} \int_0^{t_1} \cdots \int_0^{t_{j-1}} \big\|\,\mathcal{U}^{(k)}(t+\delta-t_1) \, \big[B^{(k+1)}\big]^{\omega} \, \mathcal{U}^{(k+1)}(t_1-t_2) \, \big[B^{(k+2)}\big]^{\omega} \cdots $$
$$\cdots \,\mathcal{U}^{(k+j-1)}(t_{j-1}-t_j) \, \big[B^{(k+j)}\big]^{\omega} \, \mathcal{U}^{(k+j)}(t_j) \gamma_0^{(k+j)}\,\big\|_{L^2(\Omega) H^{\alpha}(\Lambda^k \times \Lambda^k)}\, dt_j \, dt_{j-1} \,\cdots\, dt_2 \, dt_1
$$
which by Proposition \ref{non-resonant} is:
$$\leq \int_{t}^{t+\delta} \int_{0}^{t_1} \cdots \int_{0}^{t_{j-1}} 2^j \, \cdot \,C_0^{k+j} \, \cdot \, k \, \cdot \, (k+1) \cdots (k+j-1) \, \cdot \, \big\|\gamma_0^{(k+j)}\big\|_{H^{\alpha}(\Lambda^{k+j} \times \Lambda^{k+j})} \,dt_j \, dt_{j-1} \cdots \,dt_2 \, dt_1.$$
Here $C_0=C_0(\alpha)$ is the implied constant for the regularity $\alpha$ in Proposition \ref{non-resonant}. 
We recall \eqref{integral_bound} and we argue as in \eqref{II}, given in Step 1 of the proof of Theorem \ref{Theorem2} to deduce that there exist $C_1''=C_1''(\alpha),C_2''=C_2''(\alpha)>0$ such that this quantity is:
\begin{equation}
\label{IIGP2}
\lesssim_{\,T} \delta \cdot (C_1''\,T)^j \cdot \,(C_2''\,)^k \cdot \big\|\gamma_0^{(k+j)}\big\|_{H^{\alpha}(\Lambda^{k+j} \times \Lambda^{k+j})}.
\end{equation}
We choose $\widetilde{C_1}$ and $\widetilde{C_2}$ to satisfy:
\begin{equation}
\label{widetildeC1C2A}
\widetilde{C_1} \geq \max\{C_1',C_1''\},\,\widetilde{C_2} \geq \max\{C_2',C_2''\}.
\end{equation}
We note that these constraints are given in terms of $\alpha$ and $\alpha_0$.
By using \eqref{IGP2} and \eqref{IIGP2}, it follows that:
$$\big\| Duh_{\,j}^{\,\omega}(\Gamma_N(0))^{(k)}(t+\delta)- Duh_{\,j}^{\,\omega}(\Gamma_N(0))^{(k)}(t)\big\|_{L^2(\Omega)H^{\alpha}(\Lambda^k \times \Lambda^k)}$$
\begin{equation}
\label{Difference_EstimateGP2}
\lesssim_{\,\alpha,\,\alpha_0,\,T} \delta^{r} \cdot (\widetilde{C_1}T)^j \cdot \widetilde{C_2}^k \cdot \big\|\gamma_0^{(k+j)}\big\|_{H^{\alpha_0}(\Lambda^{k+j} \times \Lambda^{k+j})}.
\end{equation}
Here, we used the fact that that, by construction, $r \leq 1$.

Using \eqref{GammaNomega_deltaGP2}, \eqref{Difference_EstimateGP2}, and arguing as for \eqref{Difference_Estimate_42b} in the proof of Theorem \ref{Theorem2}, it follows that:
$$\big\|\big[\Gamma_N\,\big]^{\omega}(t+\delta)-\big[\Gamma_N\,\big]^{\omega}(t)\big\|_{L^2(\Omega) \mathcal{H}^{\alpha}_{\xi}}$$
$$\lesssim_{\,\alpha,\,\alpha_0,\,T} \delta^{r} \cdot \sum_{j=0}^{\infty} \Big(\frac{\widetilde{C_1} \,T}{\xi'}\Big)^j \cdot \sum_{k=0}^{\infty} \Big(\frac{\widetilde{C_2}\,\xi}{\xi'}\Big)^k \cdot \big\|\Gamma(0)\big\|_{\mathcal{H}^{\alpha_0}_{\xi'}}.$$
By assumptions $i)$ and $ii)$ of the theorem, we take $\xi,\xi'$, and $T$ to satisfy $\frac{\widetilde{C_1} \,T}{\xi'}<1$ and $\frac{\widetilde{C_2}\,\xi}{\xi'}<1$. It then follows that the above sum is:
$$\lesssim_{\,\xi,\,\xi',\,T,\,\alpha,\alpha_0} \delta^{r} \, \cdot \, \big\|\Gamma(0)\big\|_{\mathcal{H}^{\alpha_0}_{\xi'}}.$$
Since we can choose $T$ to be a function of $\xi,\xi',\alpha,\alpha_0$, it follows that:
\begin{equation}
\label{equicontinuityGP2}
\big\|\big[\Gamma_N\,\big]^{\omega}(t+\delta)-\big[\Gamma_N\,\big]^{\omega}(t)\big\|_{L^2(\Omega)\mathcal{H}^{\alpha}_{\xi}} \lesssim_{\,\xi,\,\xi',\,\alpha,\,\alpha_0} \delta^{r} \, \cdot \, \big\|\Gamma(0)\big\|_{\mathcal{H}^{\alpha_0}_{\xi'}}.
\end{equation}
Since $\Gamma(0) \in \mathcal{H}^{\alpha_0}_{\xi'}$, the equicontinuity of $\big(\big[\Gamma_N\,\big]^{\omega}\big)_N$ in $L^2(\Omega) \mathcal{H}^{\alpha}_{\xi}$ now follows.

\medskip

\textbf{Step 2:} The limit $\big[\Gamma\,\big]^{\omega}$ of $\big(\big[\Gamma_N\,\big]^{\omega}\big)_N$ belongs to $C_{t \in I} L^2(\Omega) \mathcal{H}^{\alpha}_{\xi}.$
\medskip

This step is analogous to Step 2 of the proof of Theorem \ref{Theorem2}.
Namely, we recall that, by Proposition \ref{CauchySequence2GP2}:

$$\big[\Gamma_N\,\big]^{\omega} \rightarrow \big[\Gamma\,\big]^{\omega}$$
as $N \rightarrow \infty$ in $L^{\infty}_{t \in I} L^2(\Omega) \mathcal{H}^{\alpha}_{\xi}.$
The fact that $\big[\Gamma\,\big]^{\omega} \in C_{t \in I} L^2(\Omega) \mathcal{H}^{\alpha}_{\xi}$ now follows from the equicontinuity result in Step 1.
In particular, we note that:

\begin{equation}
\notag
%\label{equicontinuity2GP2}
\big\|\big[\Gamma\,\big]^{\omega}(t+\delta)-\big[\Gamma\,\big]^{\omega}(t)\big\|_{L^2(\Omega)\mathcal{H}^{\alpha}_{\xi}} \lesssim_{\,\xi,\,\xi',\,\alpha,\,\alpha_0} \delta^{r} \, \cdot \, \big\|\Gamma(0)\big\|_{\mathcal{H}^{\alpha_0}_{\xi'}}.
\end{equation}

\medskip

\textbf{Step 3:} $\mathcal{U}(t) \, \Gamma(0)$ belongs to $C_{t \in I} L^2(\Omega) \mathcal{H}^{\alpha}_{\xi_0}$.

\medskip

From Lemma \ref{DifferenceBound}, it follows that:

\begin{equation}
\notag
\big\|\,\mathcal{U}(t+\delta)\,\Gamma(0)-\mathcal{U}(t)\,\Gamma(0)\big\|_{L^2(\Omega) \mathcal{H}^{\alpha}_{\xi}} \lesssim_{\,\alpha,\,\alpha_0} \delta^r \, \cdot \,  \big\|\Gamma(0)\big\|_{\mathcal{H}^{\alpha_0}_{\xi}},
\end{equation}
for some $r=r(\alpha,\alpha_0)>0$. By assumption, we know that $\big\|\Gamma(0)\big\|_{\mathcal{H}^{\alpha_0}_{\xi}} \leq \big\|\Gamma(0)\big\|_{\mathcal{H}^{\alpha_0}_{\xi'}} <\infty$ and hence $\mathcal{U}(t) \, \Gamma(0) \in C_{t \in I} L^2(\Omega) \mathcal{H}^{\alpha}_{\xi_0}$.

\medskip

\textbf{Step 4:} $\big(\big[\widehat{B}\,\big]^{\omega} \big[\Gamma_N\,\big]^{\omega}\big)_N$ is equicontinuous in  $L^2(\Omega) \mathcal{H}^{\alpha}_{\xi}$.

\medskip

In this step, we cannot use the equicontinuity result from Step 1, as we did in the independently randomized setting, since there is $\omega$ dependence both in $\big[\widehat{B}\,\big]^{\omega}$ and in $\big[\Gamma_N \,\big]^{\omega}$. Instead, we have to argue directly by using Lemma \ref{Duhamel_Expansion_Difference}. Given $t$ in $I$ and $\delta>0$ small, we need to estimate the difference:
\begin{equation}
\label{BkappaGammaNomega_deltaGP2}
\big\|\big[\widehat{B}\,\big]^{\omega} \big[\Gamma_N\,\big]^{\omega}(t+\delta)-\big[\widehat{B}\,\big]^{\omega} \big[\Gamma_N\,\big]^{\omega}(t)\big\|_{L^2(\Omega)\mathcal{H}^{\alpha}_{\xi}}=
\end{equation}
$$\sum_{k=1}^{\infty} \xi^k \cdot \big\|\big(\big[\widehat{B}\,\big]^{\omega} \big[\Gamma_N\,\big]^{\omega}\big)^{(k)}(t+\delta)-\big(\big[\widehat{B}\,\big]^{\omega} \big[\Gamma_N\,\big]^{\omega}\big)^{(k)}(t)\big\|_{L^2(\Omega)H^{\alpha}(\Lambda^k \times \Lambda^k)}$$
in terms of the initial data $\Gamma(0)$.

We fix $k \in \mathbb{N}$ and we note that:
$$\big\|\big(\big[\widehat{B}\,\big]^{\omega} \big[\Gamma_N\,\big]^{\omega}\big)^{(k)}(t+\delta)-\big(\big[\widehat{B}\,\big]^{\omega} \big[\Gamma_N\,\big]^{\omega}\big)^{(k)}(t)\big\|_{L^2(\Omega)H^{\alpha}(\Lambda^k \times \Lambda^k)}$$
$$ \leq \sum_{j=0}^{N-k-1} \big\|\big[B^{(k+1)}\big]^{\omega} Duh_{\,j}^{\,\omega}(\Gamma_N(0))^{(k+1)}(t+\delta)- \big[B^{(k+1)}\big]^{\omega} Duh_{\,j}^{\,\omega}(\Gamma_N(0))^{(k+1)}(t)\big\|_{L^2(\Omega) H^{\alpha}(\Lambda^{k+1} \times \Lambda^{k+1})}.$$ 
Let us furthermore fix $j \in \{0,1,\ldots,N-k-1\}$. Then:
$$\big[B^{(k+1)}\big]^{\omega} Duh_{\,j}^{\,\omega}(\Gamma_N(0))^{(k+1)}(t+\delta)- \big[B^{(k+1)}\big]^{\omega} Duh_{\,j}^{\,\omega}(\Gamma_N(0))^{(k+1)}(t)=$$
$$=(-i)^j \int_0^t \int_0^{t_1} \cdots \int_0^{t_{j-1}} \big[B^{(k+1)}\big]^{\omega} \, \big(\,\mathcal{U}^{(k+1)}(t+\delta-t_1)-\,\mathcal{U}^{(k+1)}(t-t_1)\big) \big[B^{(k+2)}\big]^{\omega} \, \mathcal{U}^{(k+2)}(t_1-t_2) \, \big[B^{(k+3)}\big]^{\omega} \cdots $$
$$\cdots \,\mathcal{U}^{(k+j)}(t_{j-1}-t_j) \, \big[B^{(k+j+1)}\big]^{\omega} \, \mathcal{U}^{(k+j+1)}(t_j) \, \gamma_0^{(k+j+1)}\, dt_j \, dt_{j-1} \,\cdots\, dt_2 \, dt_1$$
$$+(-i)^j \int_{t}^{t+\delta} \int_0^{t_1} \cdots \int_0^{t_{j-1}} \big[B^{(k+1)}\big]^{\omega} \,\mathcal{U}^{(k+1)}(t+\delta-t_1) \, \big[B^{(k+2)}\big]^{\omega} \, \mathcal{U}^{(k+2)}(t_1-t_2) \, \big[B^{(k+3)}\big]^{\omega} \cdots $$
$$\cdots \,\mathcal{U}^{(k+j)}(t_{j-1}-t_j) \, \big[B^{(k+j+1)}\big]^{\omega} \, \mathcal{U}^{(k+j+1)}(t_j) \, \gamma_0^{(k+j+1)}\, dt_j \, dt_{j-1} \,\cdots\, dt_2 \, dt_1$$
$$=: \widetilde{\mathcal{I}} + \widetilde{\mathcal{II}}.$$
Let us now estimate the quantities $\widetilde{\mathcal{I}}$ and $\widetilde{\mathcal{II}}$ in the norm $\big\|\cdot\big\|_{L^2(\Omega)H^{\alpha}(\Lambda^k \times \Lambda^k)}.$

By Minkowski's inequality:

$$\big\|\,\widetilde{\mathcal{I}}\,\big\|_{L^2(\Omega) H^{\alpha}(\Lambda^k \times \Lambda^k)} \leq$$
$$\int_0^t \int_0^{t_1} \cdots \int_0^{t_{j-1}} \big\| \big[B^{(k+1)}\big]^{\omega} \, \big(\,\mathcal{U}^{(k+1)}(t+\delta-t_1)-\,\mathcal{U}^{(k+1)}(t-t_1)\big) \big[B^{(k+2)}\big]^{\omega} \, \mathcal{U}^{(k+2)}(t_1-t_2) \, \big[B^{(k+3)}\big]^{\omega} \cdots $$
$$\cdots
 \,\mathcal{U}^{(k+j)}(t_{j-1}-t_j) \, \big[B^{(k+j+1)}\big]^{\omega} \, \mathcal{U}^{(k+j+1)}(t_j) \, \gamma_0^{(k+j+1)}\,\big\|_{L^2(\Omega) H^{\alpha}(\Lambda^k \times \Lambda^k)}\, dt_j \, dt_{j-1} \,\cdots\, dt_2 \, dt_1.$$
Let us now consider the integrand in the above integral more closely. 
By using the identity \eqref{Duhamel_Expansion_Difference1}  and the bound \eqref{Estimate_On_F2} given in Lemma \ref{Duhamel_Expansion_Difference}, and by arguing in the same way as in the proof of Proposition \ref{non-resonant} (i.e. as in Proposition 6.2 of \cite{SoSt}), it follows that:
\begin{equation}
\label{difference_alpha_alpha0}
\big\| \big[B^{(k+1)}\big]^{\omega} \, \big(\,\mathcal{U}^{(k+1)}(t+\delta-t_1)-\,\mathcal{U}^{(k+1)}(t-t_1)\big) \big[B^{(k+2)}\big]^{\omega} \, \mathcal{U}^{(k+2)}(t_1-t_2) \, \big[B^{(k+3)}\big]^{\omega} \cdots 
\end{equation}
$$\cdots
 \,\mathcal{U}^{(k+j)}(t_{j-1}-t_j) \, \big[B^{(k+j+1)}\big]^{\omega} \, \mathcal{U}^{(k+j+1)}(t_j) \, \gamma_0^{(k+j+1)}\,\big\|_{L^2(\Omega) H^{\alpha}(\Lambda^k \times \Lambda^k)}$$
$$\lesssim_{\,\alpha,\,\alpha_0} \delta^{\rho} \cdot C_4^{k+j+1} \cdot k \cdot (k+1) \cdots (k+j) \cdot \big\|\gamma_0^{(k+j+1)}\big\|_{H^{\alpha_0}(\Lambda^{k+j+1} \times \Lambda^{k+j+1})},$$
for some $\rho=\rho(\alpha,\alpha_0) \in (0,1)$ and for some $C_4=C_4(\alpha,\alpha_0)>0$. In this step, we needed to use the assumption that $\Gamma(0) \in \mathcal{N}$. By arguing as in the proof of Proposition \ref{DuhamelEstimate2}, we note that \eqref{difference_alpha_alpha0} implies:
\begin{equation}
\label{ItildeGP2}
\big\|\,\widetilde{\mathcal{I}}\,\big\|_{L^2(\Omega) H^{\alpha}(\Lambda^k \times \Lambda^k)} \leq \delta^{\rho} \cdot (C_3' T)^j \cdot \, ({C'_4})^k \cdot \big\| \gamma_0^{(k+j+1)}\big\|_{H^{\alpha_0} (\Lambda^{k+j+1} \times \Lambda^{k+j+1})},
\end{equation}
for some constants $C_3'=C_3'(\alpha,\alpha_0), C_4'=C_4'(\alpha,\alpha_0)>0$.

By using the same argument we used to prove estimate the term $\mathcal{II}$ in \eqref{IIGP2}, it follows that:
\begin{equation}
\label{IItildeGP2}
\big\|\,\widetilde{\mathcal{II}}\,\big\|_{L^2(\Omega) H^{\alpha}(\Lambda^k \times \Lambda^k)} \lesssim_{\,T} \delta \cdot (C_3'' T)^{j} \cdot (C_4'')^{k} \cdot \big\| \gamma_0^{(k+j+1)}\big\|_{H^{\alpha} (\Lambda^{k+j+1} \times \Lambda^{k+j+1})},
\end{equation}
for some constants $C_3''=C_3''(\alpha), C_4''=C_4''(\alpha)>0$.
In addition to \eqref{widetildeC1C2A}, we choose $\widetilde{C_1}$ and $\widetilde{C_2}$ to satisfy:
\begin{equation}
\label{widetildeC1C2B}
\widetilde{C_1} \geq \max\{C_3',C_3''\},\,\widetilde{C_2} \geq \max\{C_4',C_4''\}.
\end{equation}
These constraints are given in terms of $\alpha$ and $\alpha_0$.
Let us note that, by \eqref{widetildeC1C2A} and \eqref{widetildeC1C2B}, we can choose $\widetilde{C_1}$ and $\widetilde{C_2}$ to depend on $\alpha$ and $\alpha_0$.

By \eqref{ItildeGP2} and \eqref{IItildeGP2}, it follows that:

$$\big\|\big[B_{k+1}\big]^{\omega} Duh_{\,j}^{\,\omega}(\Gamma_N(0))^{(k+1)}(t+\delta)- \big[B_{k+1}\big]^{\omega} Duh_{\,j}^{\,\omega}(\Gamma_N(0))^{(k+1)}(t)\big\|_{L^2(\Omega) H^{\alpha}(\Lambda^k \times \Lambda^k)} 
$$
\begin{equation}
\label{Difference_Estimate2GP2}
\lesssim_{\,\alpha,\,\alpha_0,\,T} \delta^{\rho} \cdot (\widetilde{C_1}T)^j \cdot \widetilde{C_2}^k \cdot \big\|\gamma_0^{(k+j)}\big\|_{H^{\alpha_0}(\Lambda^{k+j} \times \Lambda^{k+j})}.
\end{equation}
Similarly as before, we used the fact that that, by construction, $\rho \leq 1$.

Using \eqref{BkappaGammaNomega_deltaGP2} and \eqref{Difference_Estimate2GP2}, and arguing similarly as in Step 1 of the proof, it follows that:

$$\big\|\big[\widehat{B}\,\big]^{\omega} \big[\Gamma_N\,\big]^{\omega}(t+\delta)-\big[\widehat{B}\,\big]^{\omega} \big[\Gamma_N\,\big]^{\omega}(t)\big\|_{L^2(\Omega) \mathcal{H}^{\alpha}_{\xi}}$$
$$\leq \frac{\delta^{\rho}}{\xi'} \cdot \sum_{j=0}^{\infty} \Big(\frac{\widetilde{C_1} \,T}{\xi'}\Big)^j \cdot \sum_{k=0}^{\infty} \Big(\frac{\widetilde{C_2}\,\xi}{\xi'}\Big)^k \cdot \big\|\Gamma(0)\big\|_{\mathcal{H}^{\alpha_0}_{\xi'}}.$$
As in \eqref{equicontinuityGP2}, this quantity is:
\begin{equation}
\label{Step4Theorem2GP2}
\lesssim_{\,\xi,\,\xi',\,\alpha,\,\alpha_0} \delta^{\rho} \, \cdot \, \big\|\Gamma(0)\big\|_{\mathcal{H}^{\alpha_0}_{\xi'}},
\end{equation}
for $T$ as in the assumptions of the theorem.
The equicontinuity of $\big(\big[\widehat{B}\,\big]^{\omega} \big[\Gamma_N\,\big]^{\omega}\big)_N$ in $L^2(\Omega) \mathcal{H}^{\alpha}_{\xi}$ now follows since $\Gamma(0) \in \mathcal{H}^{\alpha_0}_{\xi'}$.

\medskip

\textbf{Step 5:} $\theta^{\,\omega}$ belongs to $C_{t \in I} L^2(\Omega) \mathcal{H}^{\alpha}_{\xi}$.

\medskip

We recall that $\theta^{\,\omega}$ was constructed in Proposition \ref{limit2GP2} as the strong limit of $\big[\widehat{B}\,\big]^{\omega} \big[\Gamma_N\,\big]^{\omega}$ when $N \rightarrow \infty$.
This step is now analogous to Step 2 and it follows immediately from the equicontinuity result in Step 4.
More precisely, from \eqref{Step4Theorem2GP2} it follows that, for all $t \in I$ and for all $\delta>0$:
\begin{equation}
\label{Step5Theorem2GP2}
\big\|\theta^{\,\omega}(t+\delta)-\theta^{\,\omega}(t)\big\|_{L^2(\Omega) \mathcal{H}^{\alpha}_{\xi}} \lesssim_{\,\xi,\,\xi',\,\alpha,\,\alpha_0}\delta^{\rho} \cdot \big\|\Gamma(0)\big\|_{\mathcal{H}^{\alpha_0}_{\xi'}}.
\end{equation}

\medskip

\textbf{Step 6:} $\int_{0}^{t} \, \mathcal{U}(t-s) \, \theta^{\,\omega}(s) \, ds$ belongs to $C_{t \in I} L^2(\Omega) \mathcal{H}^{\alpha}_{\xi}$.

\medskip 

We now use \eqref{Step5Theorem2GP2} and we argue analogously as in Step 5 of the proof of Theorem \ref{Theorem2}. In particular, by Minkowski's inequality and unitarity:

$$\big\|\int_{0}^{t+\delta} \,\mathcal{U}(t+\delta-s) \, \theta^{\,\omega}(t+\delta-s) \,ds - \int_{0}^{t} \,\mathcal{U}(t-s) \, \theta^{\,\omega}(t-s)\,ds \,\big\|_{L^2(\Omega) \mathcal{H}^{\alpha}_{\xi}}$$
$$\leq \int_{-\delta}^{0} \big\|\theta^{\,\omega}(s+\delta)\big\|_{L^2(\Omega) \mathcal{H}^{\alpha}_{\xi}} \, ds + \int_{0}^{t} \big\|\theta^{\,\omega}(s+\delta)-\theta^{\,\omega}(s)\big\|_{L^2(\Omega) \mathcal{H}^{\alpha}_{\xi}} \, ds$$
$$\lesssim_{\,\xi,\,\xi',\,\alpha,\,\alpha_0} \delta \cdot \big\|\theta^{\,\omega}\big\|_{L^{\infty}_{t \in I} L^2(\Omega) \mathcal{H}^{\alpha}_{\xi}} + T \cdot \delta^{\rho}  \cdot \big\|\Gamma(0)\big\|_{\mathcal{H}^{\alpha_0}_{\xi'}}.$$
Here, we used \eqref{Step5Theorem2GP2} in the estimate of the second term.
Since $\big\|\theta^{\,\omega}\big\|_{L^{\infty}_{t \in I}L^2(\Omega) \mathcal{H}^{\alpha}_{\xi}}$ and $\|\Gamma(0)\big\|_{\mathcal{H}^{\alpha_0}_{\xi'}}$ 
are finite, it follows that $\int_{0}^{t} \, \mathcal{U}(t-s) \, \theta^{\,\omega}(s) \, ds \in C_{t \in I} L^2(\Omega) \mathcal{H}^{\alpha}_{\xi_0}$. 

\medskip

\textbf{Step 7:} Conclusion of the proof.

\medskip

Let us recall that, by Proposition \ref{Gamma_omega_equation}:

\begin{equation}
\label{Step7A}
\big\|\big[\Gamma\,\big]^{\omega}(t)-\mathcal{U}(t)\,\Gamma(0)+i \int_{0}^{t}\,\mathcal{U}(t-s)\,\theta^{\,\omega}(s)\,ds\,\big\|_{L^{\infty}_{t \in I} L^2(\Omega) \mathcal{H}^{\alpha}_{\xi}}=0.
\end{equation}

By using Step 2, Step 3, and Step 6, it follows that:

\begin{equation}
\label{Step7B}
\big[\Gamma\,\big]^{\omega}(t)-\mathcal{U}(t)\,\Gamma(0)+i \int_{0}^{t}\,\mathcal{U}(t-s)\,\theta^{\,\omega}(s)\,ds \in C_{t \in I} L^2(\Omega) \mathcal{H}^{\alpha}_{\xi}.
\end{equation}

From \eqref{Step7A} and \eqref{Step7B}, we can deduce that \emph{for all} $t \in I$:

\begin{equation}
\label{Step7C}
\big\|\big[\Gamma\,\big]^{\omega}(t)-\mathcal{U}(t)\,\Gamma(0)+i \int_{0}^{t}\,\mathcal{U}(t-s)\,\theta^{\,\omega}(s)\,ds\,\big\|_{L^2(\Omega) \mathcal{H}^{\alpha}_{\xi}}=0.
\end{equation}

By Minkowski's inequality and unitarity, it follows that for all $t \in I$:
\begin{equation}
\label{Step7D}
\big\|\int_{0}^{t}\,\mathcal{U}(t-s)\,\big(\theta^{\,\omega}(s)-\big[\widehat{B}\,\big]^{\omega} \big[\Gamma\,\big]^{\omega}(s)\big)\,ds\,\big\|_{L^2(\Omega) \mathcal{H}^{\alpha}_{\xi}} 
\end{equation}
$$\leq \int_{0}^{t} \big\|\theta^{\,\omega}(s)-\big[\widehat{B}\,\big]^{\omega} \big[\Gamma\,\big]^{\omega}(s)\big\|_{L^2(\Omega) \mathcal{H}^{\alpha}_{\xi}} \, ds \leq T \cdot \big\|\theta^{\,\omega}-\big[\widehat{B}\,\big]^{\omega} \big[\Gamma\,\big]^{\omega}\big\|_{L^{\infty}_{t \in I} L^2(\Omega) \mathcal{H}^{\alpha}_{\xi}}=0$$ 
In the last equality, we used \eqref{Important_Bound}.

We combine \eqref{Step7C}  and \eqref{Step7D} and we conclude that, \emph{for all} $t \in I$:
\begin{equation}
\notag
\big\|\big[\Gamma\,\big]^{\omega}(t)-\mathcal{U}(t)\,\Gamma(0)+i \int_{0}^{t}\,\mathcal{U}(t-s)\,\big[\widehat{B}\,\big]^{\omega} \big[\Gamma\,\big]^{\omega}(s) \,ds\,\big\|_{L^2(\Omega) \mathcal{H}^{\alpha}_{\xi}}=0.
\end{equation}
The theorem now follows.
\end{proof}

\end{document}